\documentclass[11pt]{article}
\usepackage[margin=1in]{geometry}
\usepackage[british]{babel}
\usepackage{amsmath,amssymb,amsthm}
\usepackage{tikz-cd}
\usepackage{tikz}
\usetikzlibrary{positioning,arrows.meta,calc}
\usepackage{pdflscape}
\usepackage{adjustbox}
\usepackage{cite}
\usepackage[colorlinks=true,allcolors=black]{hyperref}
\usepackage{xcolor}
\usepackage{enumitem}

\newtheorem{assumption}{Assumption}[section]
\newtheorem{theorem}[assumption]{Theorem}
\newtheorem{proposition}[assumption]{Proposition}
\newtheorem{lemma}[assumption]{Lemma}
\newtheorem{corollary}[assumption]{Corollary}
\theoremstyle{definition}
\newtheorem{definition}[assumption]{Definition}
\newtheorem{example}[assumption]{Example}
\theoremstyle{remark}
\newtheorem{remark}[assumption]{Remark}
\newtheorem{conjecture}[assumption]{Conjecture}

\newcommand{\dd}{\,\mathrm{d}}
\newcommand{\I}{[0,1]}
\newcommand{\Linfty}{L^{\infty}}
\newcommand{\eps}{\varepsilon}
\newcommand{\Mdh}{ {M_{DH}}}
\newcommand{\MdTwoH}{ {M_{D^2H}}}
\newcommand{\norm}[1]{\left\lVert #1 \right\rVert}
\newcommand{\R}{\mathbb{R}}
\newcommand{\DD}{\mathrm{D}}

\numberwithin{equation}{section}

\title{Compatibility of Higher-Order Slow-Manifold Reduction and Continuum Limits in Adaptive Networks}
\author{%
Christian Kuehn$^{1,2,3}$, Fergal Murphy$^{1}$\thanks{Email: \texttt{fergal.murphy@tum.de}}, and Jan-Eric Sulzbach$^{1,4}$\\[4pt]
\small $^{1}$Technical University of Munich, Department of Mathematics, Munich, Germany\\
\small $^{2}$Munich Data Science Institute (MDSI), Munich, Germany\\
\small $^{3}$Complexity Science Hub Vienna (CSH), Vienna, Austria\\
\small $^{4}$Leiden University, Mathematical Institute, Leiden, Netherlands}
\date{\today}

\begin{document}
\maketitle

\begin{abstract}
Adaptive networks couple the evolution of node states to the evolution of the interactions between them. In fast-adapting phase oscillator networks, a slow-manifold reduction of a pairwise microscopic model can generate effective higher-order terms in the phase dynamics. We ask whether this higher-order structure survives the dense-graph continuum limit, and whether it matters if one first reduces and then passes to the continuum, or first passes to the continuum and then reduces. We prove well-posedness and discrete-to-continuum convergence for the unreduced and first-order reduced models, and we construct the continuum slow manifold directly in a Banach-space setting. Along admissible equal-cell step approximations, the two routes give the same first-order continuum vector field, including the same pairwise correction and triplet operator, up to controlled \(O(\varepsilon^2)\) remainders. A continuum mixed-derivative criterion then shows that, for suitable coupling functions, the resulting triplet operator is genuinely nonpairwise in the smooth bounded-kernel class. Thus the higher-order term is not a finite-network artefact, but persists in the macroscopic continuum description considered here.
\end{abstract}



\setcounter{tocdepth}{2}
\tableofcontents

\section{Introduction}\label{sec:introduction}

\subsection*{Background and motivation}

 Adaptive network models describe systems in which the state of each node and the
strength of its interactions evolve together. They provide a natural language for many complex systems posed on networks, especially when the network couplings
themselves respond to the evolving state of the network~\cite{GrossBlasius2008,Berneretal,
Kuramoto1984,Strogatz2000}. In many such models the adaptation is fast relative to the node
dynamics, so the problem sits naturally in the fast--slow framework of geometric singular
perturbation theory~\cite{Fenichel1979,Jones1995,Kuehn2015}. From this point of view,
slow-manifold reduction is attractive not only because it lowers dimension, but because it reveals
which effective interactions are actually seen on the slow time scale.

 At the same time, higher-order interactions have become a central theme in network
dynamics~\cite{Battiston2020PhysRep,Boccalettietal1,BickAshwinRodrigues2016Chaos,SkardalArenas1}.
In many models such terms are introduced phenomenologically, for example through simplicial or
hypergraph couplings. Here the perspective is different: the microscopic system is pairwise, and
the question is whether genuinely nonpairwise terms can emerge as an effective description after
reduction. Adaptive fast--slow networks are therefore a natural setting in which to ask how
higher-order collective dynamics can arise from pairwise mechanisms.

\subsection*{The finite-dimensional result}

The recent finite-dimensional work~\cite{KuehnMurphy} gives a rigorous mechanism for
this emergence. The starting point is the adaptive phase oscillator network
\begin{equation}\label{eq:intro-full}
\dot{\theta}_i = \omega_i + \frac{1}{N}\sum_{j=1}^N a_{ij}\,\Gamma(\theta_j-\theta_i),
\qquad
\varepsilon\dot{a}_{ij} = -a_{ij} + H(\theta_i,\theta_j),
\end{equation}
 where \(\theta_i\in\mathbb{T}=\mathbb{R}/(2\pi\mathbb{Z})\) are phases and \(a_{ij}\in\mathbb{R}\) are
adaptive weights. Throughout the paper, however, the analysis is carried out on chosen real lifts
of these torus phases. Accordingly, \(\Gamma\) is treated as a \(2\pi\)-periodic function on
\(\mathbb{R}\), and \(H\) as a function on \(\mathbb{R}^2\) that is \(2\pi\)-periodic in each
argument, so all expressions below are independent of the chosen lifts. In the fast-adaptation regime the critical manifold
\[
S_0^N
:=
\{(\theta,a)\in\mathbb T^N\times\mathbb R^{N\times N}: a_{ij}=H(\theta_i,\theta_j)\text{ for all }i,j\}
\]
is a global graph over the phase variables and is globally normally hyperbolic because the fast linearisation is
simply \(-I\). Fenichel theory therefore produces a nearby slow manifold, and substituting its
graph into the phase equation yields a first-order reduced phase-only expansion of the form
\begin{equation}\label{eq:intro-reduced}
\dot{\theta}_i
=
\omega_i
+ \frac{1}{N}\sum_{j=1}^N H(\theta_i,\theta_j)\Gamma(\theta_j-\theta_i)
+ \frac{\varepsilon}{N}\sum_{j=1}^N P_{ij}(\theta;\omega)
+ \frac{\varepsilon}{N^2}\sum_{j,k=1}^N T_{ijk}(\theta)
+ O(\varepsilon^2).
\end{equation}
 The important point is not just the detailed formula of the correction terms, but
their structure: besides a first-order pairwise correction \(P_{ij}(\theta;\omega)\), the reduced vector field
contains a genuine triplet term \(T_{ijk}\). A mixed-derivative criterion in~\cite{KuehnMurphy}
shows that this term cannot, in general, be absorbed into a pairwise representation; this is
verified explicitly for the adaptive Kuramoto model. Thus slow-manifold reduction provides a
rigorous route from a pairwise adaptive system to an effective higher-order phase model.

\subsection*{The continuum-limit question}

 The result of~\cite{KuehnMurphy} is, however, a statement for fixed \(N\). Large
oscillator systems are often analysed through continuum or mean-field type limits. For the case of dense graphs, key tools are graph limits expressed via graphons. In this framework, the
discrete index \(i\in\{1,\dots,N\}\) is replaced by a continuous label \(x\in[0,1]\) and sums
become integral operators~\cite{Neunzert1984,Lancellotti2005,Medvedev2014,ChibaMediano2019,
KaliuzhnyiVerbovetskyiMedvedev2017,AyiDuteil1}. The closest precursors of the present construction are the continuum-limit results for adaptive Kuramoto-type and adaptive interacting-particle systems developed in~\cite{GkogkasKuehnXu2023,Throm2024EJAM,CestnikMartens2025}; the present paper differs in that its central question is compatibility with Fenichel reduction in the fast-adaptation regime, and in that it works directly with a dense-graph equal-cell \(L^\infty\) step embedding rather than with general graphon, graphop, or measure-valued limits. Thus the relation to graphon theory here is structural rather than topological: the \(1/N\)-scaling and the passage from sums to integral operators are those of dense graphon-type limits, but the present analysis keeps fixed labels and requires strong \(L^\infty\)-convergence of the chosen step representatives. We use this strong \(L^\infty\) equal-cell framework deliberately, because both the Lyapunov--Perron slow-manifold construction and the nonpairwise mixed-derivative certificate are naturally formulated in sup-norm spaces. This raises a structural question that is natural from both the network and continuum points of view:
\begin{quote}
\emph{Does the emergent higher-order structure survive in the continuum limit, and
does the order in which one performs reduction and continuum limit matter?}
\end{quote}

 More precisely, two operations act on the finite-\(N\) adaptive fast--slow
system~\eqref{eq:intro-full}:
\begin{itemize}
  \item \(\mathrm{Red}_\varepsilon\): slow-manifold (Fenichel) reduction, eliminating the fast
  coupling weights and producing a phase-only system with higher-order corrections;
  \item \(\mathrm{CL}\): the dense-graph continuum limit \(N\to\infty\), taken along equal-cell
  step embeddings with compatible \(L^\infty\)-convergent data and replacing normalised sums by
  integral operators over \([0,1]\).
\end{itemize}
 These can be compared in either order, giving rise to the following
schematic comparison diagram:
\[
\begin{tikzcd}[column sep=4.2em, row sep=3.6em]
\text{Finite adaptive fast--slow system}
  \arrow[r, "\mathrm{Red}_{\varepsilon}"]
  \arrow[d, "\mathrm{CL}"']
&
\text{Finite reduced higher-order system}
  \arrow[d, "\mathrm{CL}"] \\
\text{Continuum fast--slow equation}
  \arrow[r, "\mathrm{Red}^{\infty}_{\varepsilon}"']
&
\text{Reduced continuum higher-order equation}
\end{tikzcd}
\]
 Each arrow requires its own analytical justification: the horizontal arrows are
exact slow-manifold reductions, finite-dimensional and infinite-dimensional respectively, while
the vertical arrows are dense-graph macroscopic continuum limits. The order restriction enters only
when one seeks an explicit asymptotic description of the reduced dynamics: the first genuinely
nonpairwise contribution appears at order \(\varepsilon\), so the structural question addressed here
is whether the two exact routes induce the same reduced vector field through the first non-trivial
order. Schematically, after dense-label embedding and along admissible approximating sequences, we ask whether
\begin{equation}\label{eq:intro-commutation}
\mathrm{CL}\circ\mathrm{Red}_{\varepsilon}
\;\stackrel{?}{=}\;
\mathrm{Red}^{\infty}_{\varepsilon}\circ\mathrm{CL}
{
\qquad\text{through the first non-trivial order, with }O(\varepsilon^2)\text{ remainders}.}
\end{equation}
Thus the manifold constructions themselves are not truncated objects; the truncation concerns only
the explicit vector-field expansion used to isolate the first emergent higher-order term.

 This is not a purely formal question. The vertical arrows probe whether the
large-network limit preserves the adaptive fast--slow structure, while the lower horizontal arrow
requires an infinite-dimensional slow-manifold reduction in a Banach-space setting, connecting the
problem to recent work on continuum fast--slow systems and infinite-dimensional slow
manifolds~\cite{Hummel-Kuehn,Kuehn-Sulzbach,kuehn2024infinite}. If this compatibility holds, then
the first-order nonpairwise continuum interaction is a feature of the scoped
continuum limit considered here rather than an artefact of choosing one analytical route before
the other within the stated first-order, step-field framework.

In this paper we justify both routes in the diagram rigorously and identify the induced reduced
vector fields through the first non-trivial order. We first establish continuum well-posedness for
both the unreduced adaptive fast--slow system and the reduced higher-order phase dynamics, 
and discrete-to-continuum convergence for data admitting compatible
\(L^\infty\)-convergent equal-cell step approximations. We then carry out the
continuum-level slow-manifold reduction and identify the first-order reduced vector field directly
in the Banach-space setting. The reduce-first and continuum-first routes produce the same
leading-order term and the same pairwise and triplet corrections, with \(O(\varepsilon^2)\)
remainders along admissible step-field embeddings. When the triplet operator satisfies the
continuum mixed-derivative criterion proved later, this gives a rigorous mechanism by which
genuinely nonpairwise continuum interactions arise from pairwise adaptive dynamics.

The precise compatibility statement is Theorem~\ref{thm:sec5-commutation}: along admissible
equal-cell step embeddings, the reduced vector field obtained from the Lyapunov--Perron selected
finite-\(N\) representative and the continuum-first reduced vector field agree through first order
in \(\varepsilon\), up to explicitly controlled \(O(\varepsilon^2)\) remainders. Thus the theorem is
a first-order vector-field comparison: the manifold constructions themselves are exact, while the
common asymptotic description is identified only through order \(\varepsilon\) to focus on the key effect that higher-order interactions are generated in the slow flow.

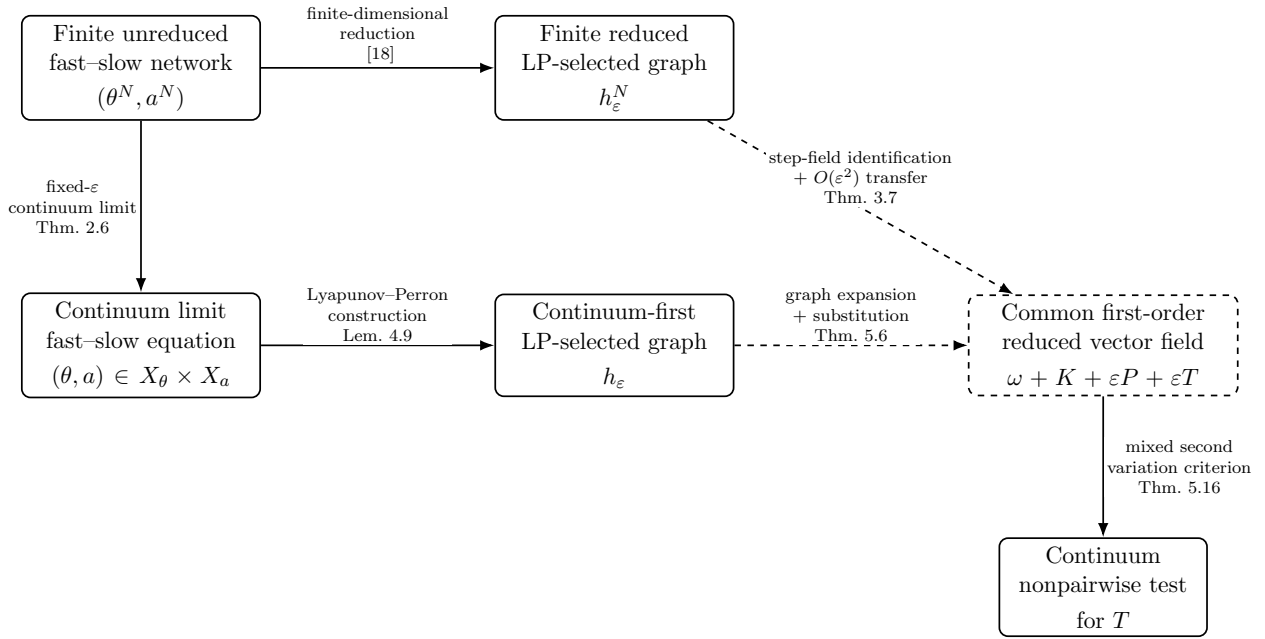
\begin{figure}[h]
\centering
\begin{adjustbox}{max width=\textwidth}
\begin{tikzpicture}[
    >=Latex,
    line width=0.8pt,
    node distance=2.8cm and 3.8cm,
    box/.style={
        draw,
        rounded corners,
        align=center,
        text width=3.6cm,
        minimum height=1.35cm,
        inner sep=4pt
    },
    widebox/.style={
        draw,
        rounded corners,
        dashed,
        align=center,
        text width=4.1cm,
        minimum height=1.45cm,
        inner sep=4pt
    },
    smallbox/.style={
        draw,
        rounded corners,
        align=center,
        text width=3.1cm,
        minimum height=1.25cm,
        inner sep=4pt
    },
    lab/.style={
        font=\scriptsize,
        align=center,
        fill=white,
        inner sep=1.5pt,
        text opacity=1
    }
]

\node[box] (FN) {%
Finite unreduced\\
fast--slow network\\[1mm]
\((\theta^N,a^N)\)
};

\node[box, right=of FN] (FR) {%
Finite reduced\\
LP-selected graph\\[1mm]
\(h^N_\varepsilon\)
};

\node[box, below=of FN] (CN) {%
Continuum limit\\
fast--slow equation\\[1mm]
\((\theta,a)\in X_\theta\times X_a\)
};

\node[box, right=of CN] (CR) {%
Continuum-first\\
LP-selected graph\\[1mm]
\(h_\varepsilon\)
};

\node[widebox, right=of CR] (TR) {%
Common first-order\\
reduced vector field\\[1mm]
\(
\omega + K + \varepsilon P + \varepsilon T
\)
};

\node[smallbox, below=2.3cm of TR] (NP) {%
Continuum\\
nonpairwise test\\[1mm]
for \(T\)
};

\draw[->] (FN) -- (FR)
    node[midway, above, lab] {finite-dimensional\\ reduction\\ \cite{KuehnMurphy}};

\draw[->] (FN) -- (CN)
    node[midway, left, lab] {fixed-\(\varepsilon\)\\ continuum limit\\ Thm.~\ref{thm:sec2-conv}};

\draw[->] (CN) -- (CR)
    node[midway, above, lab] {Lyapunov--Perron\\ construction\\ Lem.~\ref{lem:lyapunov_perron}};

\draw[->, dashed] (FR) -- (TR)
    node[midway, above, lab] {step-field identification\\ + \(O(\varepsilon^2)\) transfer\\ Thm.~\ref{thm:sec3-conv}};

\draw[->, dashed] (CR) -- (TR)
    node[midway, above, lab] {graph expansion\\ + substitution\\ Thm.~\ref{thm:sec5-identification}};

\draw[->] (TR) -- (NP)
    node[midway, right, lab] {mixed second\\ variation criterion\\ Thm.~\ref{thm:sec5-triplet-criterion}};

\end{tikzpicture}
\end{adjustbox}
\caption{Logical structure of the proof. Solid arrows denote exact continuum-limit
or Lyapunov--Perron constructions. Dashed arrows denote first-order expansion
and truncation steps.}
\label{fig:proof-roadmap}
\end{figure}

\subsection*{Organisation}

 Section~\ref{sec:unreduced-model} studies the continuum limit of the unreduced
adaptive fast--slow system. Section~\ref{sec:reduced-model} develops the corresponding continuum
limit for the reduced higher-order phase equation. Sections~\ref{sec:infdim-reduction}
and~\ref{sec:identification} then carry out the continuum-level reduction, compare the two
first-order vector-field truncations, and formulate the continuum nonpairwise criterion.
Section~\ref{sec:discussion} closes with the main implications, limitations, and possible
extensions of the theory.

\section{Continuum Limit of the Unreduced Adaptive Fast--Slow Model}\label{sec:unreduced-model}

This section sets up the continuum-limit problem for the original adaptive fast--slow system before
any reduction. The goal is to put the model in a rigorous function-space framework and state the
precise well-posedness and convergence statements needed later for the first-order
compatibility comparison.

 More precisely, we first identify the continuum fast--slow equation associated with
the dense-label embedding of the finite adaptive network and place it in a natural
\(\Linfty\)-based Banach-space setting. Within that framework, the section proves global
well-posedness for the continuum dynamics and then shows that the embedded finite-\(N\)
trajectories 
converge to this solution on every finite time interval for data satisfying the
step-approximation hypothesis introduced below.

 At a conceptual level, the analysis rests on two structural features of the model.
First, once \(\theta\) is regarded as given, the fast variable \(a\) satisfies a linear
relaxation equation, so its size can be controlled explicitly by a variation-of-constants
formula. Second, after that a priori bound is available, the coupled vector field is controlled by
a finite-time Gr\"onwall estimate, which simultaneously gives stability, continuous dependence, and
the discrete-to-continuum limit. The section is organised to make those two mechanisms visible
before turning to the formal proofs.

\subsection[From the Adaptive Network to the Continuum Equation]{From the adaptive network to the continuum equation}

 We begin with the dense-label embedding of the finite network, since this is the
point where the continuum problem enters naturally. The aim is to put the discrete system, the
ambient function spaces, and the limiting equation in place before the two main theorems are
stated.

For each \(N\in\mathbb{N}\), consider the adaptive network written in chosen real lifts of the
phase variables
\begin{equation}
\dot{\theta}_i^N
=
\omega_i^N + \frac{1}{N}\sum_{j=1}^N a_{ij}^N\,\Gamma(\theta_j^N-\theta_i^N),
\qquad
\varepsilon\dot{a}_{ij}^N = -a_{ij}^N + H(\theta_i^N,\theta_j^N),
\label{eq:sec2-discrete}
\end{equation}
for \(1\le i,j\le N\).
The \(1/N\) factor places the model in the dense-interaction scaling used for graphon limits.

Let
\[
I_1^N := [0,\,1/N],
\qquad
I_i^N := \big((i-1)/N,\,i/N\big]\quad(2\le i\le N),
\qquad
\I = \bigsqcup_{i=1}^N I_i^N,
\]
and define the embedded step fields
\begin{equation}
\theta^N(t,x) := \theta_i^N(t)\quad(x\in I_i^N),
\qquad
a^N(t,x,y) := a_{ij}^N(t)\quad((x,y)\in I_i^N\times I_j^N).
\label{eq:sec2-step-embedding}
\end{equation}
Likewise set \(\omega^N(x):=\omega_i^N\) on \(I_i^N\).

With this embedding, \eqref{eq:sec2-discrete} is equivalent to
\begin{align}
\partial_t\theta^N(t,x)
&= \omega^N(x) + \int_0^1 a^N(t,x,y)\,\Gamma\!\big(\theta^N(t,y)-\theta^N(t,x)\big)\,\dd y,
\label{eq:sec2-embedded-theta}\\
\varepsilon\partial_t a^N(t,x,y)
&= -a^N(t,x,y) + H\!\big(\theta^N(t,x),\theta^N(t,y)\big).
\label{eq:sec2-embedded-a}
\end{align}

\begin{remark}[Lift convention in the continuum setting]
The continuum phase fields \(\theta^N\) and \(\theta\) are understood as chosen real lifts of
torus-valued phase profiles. Because \(\Gamma\) is \(2\pi\)-periodic and \(H\) is \(2\pi\)-periodic
in each argument, the quantities \(\Gamma(\theta(y)-\theta(x))\) and \(H(\theta(x),\theta(y))\)
are unchanged if one modifies the lift by adding integer multiples of \(2\pi\). Hence the
continuum problem can be posed on the real-valued Banach spaces \(L^\infty(\I)\) and
\(L^\infty(\I^2)\) without ambiguity at the level of the vector field. All \(L^\infty\)-norm
estimates and convergence statements below are made after fixing compatible real lifts; in
particular, \(\|\theta^N-\theta\|_{\Linfty}\to0\) refers to these chosen representatives, not to a
quotient metric on \(L^\infty(\I;\mathbb T)\).
\end{remark}

 The embedded variables suggest a natural product space: the slow variable
\(\theta\) lives on \(\I\), while the fast interaction variable \(a\) lives on \(\I^2\). We
therefore record the Banach spaces and approximation assumptions before writing down the continuum
equation itself.

Define
\[
X_\theta := \Linfty(\I),
\qquad
X_a := \Linfty(\I^2),
\qquad
X := X_\theta\times X_a,
\]
with norm
\[
\|(\theta,a)\|_X := \|\theta\|_{\Linfty(\I)} + \|a\|_{\Linfty(\I^2)}.
\]

\begin{assumption}[Regularity and approximation hypotheses]\label{ass:sec2-main}
Fix \(\varepsilon>0\). Assume:
\begin{enumerate}[label=\textup{(A\arabic*)}, ref=\textup{(A\arabic*)}]
  \item\label{ass:sec2-A1} \(\Gamma\in C_b^1(\mathbb{R})\) is \(2\pi\)-periodic, and
  \(H\in C_b^1(\mathbb{R}^2)\) is \(2\pi\)-periodic in each argument.
  \item\label{ass:sec2-A2} \(\omega\in \Linfty(\I)\).
  \item\label{ass:sec2-A3} Initial data satisfy \((\theta_0,a_0)\in X\).
  \item\label{ass:sec2-A4} There exist embedded step approximations
  \((\theta_0^N,a_0^N,\omega^N)\) on the standard equal-cell partition such that
  \[
  \|\theta_0^N-\theta_0\|_{\Linfty(\I)}
  +\|a_0^N-a_0\|_{\Linfty(\I^2)}
  +\|\omega^N-\omega\|_{\Linfty(\I)} \to 0.
  \]
\end{enumerate}
\end{assumption}

 Within this framework, the continuum equation is obtained by replacing the step
fields with general elements of \(X_\theta\) and \(X_a\), so it is posed on all of \(X\). The
discrete-to-continuum theorem below, however, is restricted by
Assumption~\ref{ass:sec2-main}\,\ref{ass:sec2-A4} to data admitting \(L^\infty\)-convergent equal-cell step
approximations. Sufficient classes for Assumption~\ref{ass:sec2-main}\,\ref{ass:sec2-A4} include continuous data on \(\I\) and \(\I^2\), and more generally \(L^\infty\) data admitting a uniformly continuous representative on the standard equal-cell partitions. Lipschitz data give an explicit \(O(1/N)\) approximation rate, which propagates to the discrete-to-continuum bound; see Corollary~\ref{cor:sec2-rate} below. (Bounded-variation data with jump discontinuities are not in general included, since the equal-cell step approximation error in \(L^\infty\) need not vanish unless the jumps are aligned with the partitions in a compatible way.)

For each \(\omega\in X_\theta\), define operators
\(\mathcal{F}_\omega:X\to X_\theta\) and \(\mathcal{G}:X\to X_a\) by
\begin{align}
\mathcal{F}_\omega(\theta,a)(x)
&:= \omega(x)
+ \int_0^1 a(x,y)\,\Gamma\!\big(\theta(y)-\theta(x)\big)\,\dd y,
\label{eq:sec2-Fop}\\
\mathcal{G}(\theta,a)(x,y)
&:= -a(x,y)+H\!\big(\theta(x),\theta(y)\big).
\label{eq:sec2-Gop}
\end{align}
When the frequency profile is fixed, we suppress it from the notation and write
\(\mathcal F=\mathcal F_\omega\).
The continuum unreduced model is
\begin{equation}
\partial_t\theta = \mathcal{F}_\omega(\theta,a),
\qquad
\varepsilon\partial_t a = \mathcal{G}(\theta,a),
\qquad
(\theta,a)|_{t=0}=(\theta_0,a_0).
\label{eq:sec2-continuum}
\end{equation}

\begin{definition}[Strong \(\Linfty\)-solution on finite horizon]
Given \(T>0\), a pair \((\theta,a)\) is a strong solution of \eqref{eq:sec2-continuum} on
\([0,T]\) if
\[
(\theta,a)\in W^{1,\infty}([0,T];X)
\]
and \eqref{eq:sec2-continuum} holds for almost every \(t\in[0,T]\).
\end{definition}

Here \(W^{1,\infty}([0,T];X)\) denotes the Bochner--Sobolev space of
\(X\)-valued functions in \(L^\infty([0,T];X)\) having one weak time derivative in
\(L^\infty([0,T];X)\); the indices \(1\) and \(\infty\) record, respectively, one weak derivative
and essential boundedness in time.

 With the model fixed, we can state the two outputs of the section. The first theorem
shows that \eqref{eq:sec2-continuum} is globally well posed in \(X\); the second identifies it as
the dense-graph limit of the embedded finite-\(N\) dynamics for data satisfying
Assumption~\ref{ass:sec2-main}\,\ref{ass:sec2-A4}.

\begin{theorem}[Well-posedness in \(X\)]\label{thm:sec2-wp}
Under Assumption~\ref{ass:sec2-main}\,\ref{ass:sec2-A1}--\ref{ass:sec2-A3}, for every \(\varepsilon>0\) and every
\((\theta_0,a_0)\in X\), system~\eqref{eq:sec2-continuum} has a unique global strong solution.
Moreover, for every \(T>0\), every \(R>0\), and every \(\varepsilon_0>0\), the
finite-time solution maps
\[
(\theta_0,a_0,\omega)\mapsto (\theta,a)
\]
corresponding to \(\varepsilon\in(0,\varepsilon_0]\) are Lipschitz on the ball
\[
\{(\theta_0,a_0,\omega):\ \|\theta_0\|_{\Linfty(\I)}+\|a_0\|_{\Linfty(\I^2)}+\|\omega\|_{\Linfty(\I)}\le R\}
\]
as maps into \(C([0,T];X)\), with a Lipschitz constant independent of
\(\varepsilon\in(0,\varepsilon_0]\).
\end{theorem}

For later use, fix \(N\in\mathbb N\) and write
\begin{align*}
\mathfrak{S}_N^\theta
&:=
\{\theta\in X_\theta:\theta\text{ is constant a.e. on each }I_i^N\},\\
\mathfrak{S}_N^a
&:=
\{a\in X_a:a\text{ is constant a.e. on each }I_i^N\times I_j^N\},
\end{align*}
and \(\mathfrak{S}_N:=\mathfrak{S}_N^\theta\times \mathfrak{S}_N^a\). The step-coefficient map identifies
\(\mathfrak{S}_N^\theta\) with \(\R^N\) and \(\mathfrak{S}_N^a\) with \(\R^{N\times N}\).

\begin{lemma}[Step-field invariance of the unreduced continuum flow]
\label{lem:sec2-stepfield-preservation}
Assume that the frequency profile in~\eqref{eq:sec2-continuum} is a step field
\(\omega=\omega^N\in \mathfrak{S}_N^\theta\). Let \((\theta_0,a_0)\in \mathfrak{S}_N\), and let
\((\theta(\cdot),a(\cdot))\) be the unique strong solution of~\eqref{eq:sec2-continuum}
on \([0,\infty)\) with these initial data. Then \((\theta(t),a(t))\in \mathfrak{S}_N\) for every
\(t\ge 0\), and its step-coefficients \((\theta_i(t),a_{ij}(t))_{i,j=1}^N\) solve the
finite-\(N\) unreduced ODE~\eqref{eq:sec2-discrete}.
\end{lemma}

\begin{proof}
The operators \(\mathcal F_{\omega^N}\) and \(\mathcal G\) of~\eqref{eq:sec2-Fop}--\eqref{eq:sec2-Gop}
map step fields to step fields. If \(\theta\in \mathfrak{S}_N^\theta\) and \(a\in \mathfrak{S}_N^a\), choose step
representatives such that \(\theta=\theta_i\) a.e. on \(I_i^N\) and
\(a=a_{ij}\) a.e. on \(I_i^N\times I_j^N\). Then, for a.e. \(x\in I_i^N\),
\[
\mathcal F_{\omega^N}(\theta,a)(x)
=
\omega_i^N+\sum_{j=1}^N\int_{I_j^N}a_{ij}\,\Gamma(\theta_j-\theta_i)\,\dd y
=
\omega_i^N+\frac1N\sum_{j=1}^N a_{ij}\,\Gamma(\theta_j-\theta_i),
\]
and \(\mathcal G(\theta,a)(x,y)=-a_{ij}+H(\theta_i,\theta_j)\) for
\((x,y)\in I_i^N\times I_j^N\). Thus the vector field of~\eqref{eq:sec2-continuum}
preserves \(\mathfrak{S}_N\), and its restriction to \(\mathfrak{S}_N\simeq\R^N\times\R^{N\times N}\)
coincides with~\eqref{eq:sec2-discrete}. Picard--Lindel\"of on this finite-dimensional
subspace identifies the restricted flow with the finite-\(N\) solution.
\end{proof}

\begin{theorem}[Discrete-to-continuum convergence under the step-approximation hypothesis]\label{thm:sec2-conv}
Assume Assumption~\ref{ass:sec2-main}. For each \(N\), let
\((\theta^N,a^N)\) be the step-field solution induced by \eqref{eq:sec2-discrete} via
\eqref{eq:sec2-step-embedding}, initialised by the step coefficients of
\((\theta_0^N,a_0^N)\) and using the frequencies \(\omega_i^N\) associated with
\(\omega^N\) from Assumption~\ref{ass:sec2-main}\,\ref{ass:sec2-A4}. Let
\((\theta,a)\) solve \eqref{eq:sec2-continuum}.
Then for every fixed \(T>0\),
\[
\sup_{t\in[0,T]}
\left(
\|\theta^N(t)-\theta(t)\|_{\Linfty(\I)}
+\|a^N(t)-a(t)\|_{\Linfty(\I^2)}
\right)
\to 0
\quad (N\to\infty).
\]
More precisely, define
\[
R_{a,*}:=\max\!\left\{M_H,\|a_0\|_{\Linfty(\I^2)},\sup_{N\in\mathbb{N}}\|a_0^N\|_{\Linfty(\I^2)}\right\},
\]
and set
\[
E_\theta^N:=\|\theta_0^N-\theta_0\|_{\Linfty(\I)},
\qquad
E_a^N:=\|a_0^N-a_0\|_{\Linfty(\I^2)},
\qquad
E_\omega^N:=\|\omega^N-\omega\|_{\Linfty(\I)},
\]
\[
\Lambda_*:=2R_{a,*}L_\Gamma+2M_\Gamma\Mdh.
\]
Then, for all \(N\) and all \(t\in[0,T]\),
\[
\|\theta^N(t)-\theta(t)\|_{\Linfty(\I)}
\le
e^{\Lambda_* t}
\bigl(E_\theta^N+tE_\omega^N+\varepsilon M_\Gamma E_a^N\bigr),
\]
\[
\|a^N(t)-a(t)\|_{\Linfty(\I^2)}
\le
e^{-t/\varepsilon}E_a^N
+
2\Mdh e^{\Lambda_* t}
\bigl(E_\theta^N+tE_\omega^N+\varepsilon M_\Gamma E_a^N\bigr).
\]
Consequently, we have
\[
\sup_{t\in[0,T]}\|(\theta^N,a^N)(t)-(\theta,a)(t)\|_X
\le
E_a^N
+
(1+2\Mdh)e^{\Lambda_*T}
\bigl(E_\theta^N+TE_\omega^N+\varepsilon M_\Gamma E_a^N\bigr).
\]
In particular, for every \(\varepsilon_0>0\),
\[
C_{T,\varepsilon_0}
:=
1+(1+2\Mdh)e^{\Lambda_*T}\max\{1,T,\varepsilon_0M_\Gamma\}
\]
gives the uniform bound
\[
\sup_{t\in[0,T]}\|(\theta^N,a^N)(t)-(\theta,a)(t)\|_X
\le
C_{T,\varepsilon_0}
\bigl(E_\theta^N+E_a^N+E_\omega^N\bigr)
\]
for every \(\varepsilon\in(0,\varepsilon_0]\).
\end{theorem}

 The remainder of the section is organised around the two ingredients just described.
We first establish local Lipschitz control for the continuum vector field and derive the explicit
variation-of-constants formula for the fast variable. These provide the local theory and the key
uniform \(L^\infty\)-bound on \(a\). We then combine that bound with a comparison
estimate, obtained by applying Gr\"onwall's inequality to the difference of two solutions, to obtain
global continuation, Lipschitz dependence on the data, and the quantitative discrete-to-continuum
estimate.

We do not introduce a separate graphon-kernel notation here because the theorem is
formulated as a labelled dense-kernel continuum limit rather than as a graphon statement modulo
relabelling. The step embedding on \(\I^2\) together with the \(1/N\)-scaling in
\eqref{eq:sec2-discrete} places the problem in the dense-interaction regime, but the convergence
statement is deliberately stronger: it applies to bounded weighted interaction fields \(a^N\)
whose chosen equal-cell representatives converge in \(L^\infty(\I^2)\), together with compatible
\(L^\infty\)-convergence of \(\theta_0^N\) and \(\omega^N\), as required in
Assumption~\ref{ass:sec2-main}\,\ref{ass:sec2-A4}. This fixed-label sup-norm framework is narrower
than more general graphon, graphop, or measure-valued limit theories, but it is the one compatible
with the uniform \(L^\infty\)-bound on \(a\), the sup-norm stability estimates, and the later
Lyapunov--Perron construction used in this paper. Furthermore, this approach constitutes the most direct path to demonstrate the emergence of higher-order interactions on the slow manifold in the continuum limit of an adaptive network, which is the focus of this paper.

\subsection[A Priori Bounds, Stability, and Convergence]{A priori bounds, stability, and convergence}
\label{subsec:sec2-proof}

 This subsection proves Theorems~\ref{thm:sec2-wp} and~\ref{thm:sec2-conv}. Rather than splitting the discussion into
several small subsections, we proceed in one continuous chain: first the local bounds for the
vector field, then the explicit control of the fast variable, and finally the stability estimate
that upgrades the local theory to global well-posedness and yields the continuum limit.

 We begin by fixing the uniform constants that appear throughout the estimates.

Set
\[
M_\Gamma := \|\Gamma\|_{\Linfty(\mathbb{R})},
\quad
L_\Gamma := \|\Gamma'\|_{\Linfty(\mathbb{R})},
\quad
M_H := \|H\|_{\Linfty(\mathbb{R}^2)},
\quad
\Mdh := \max_{r=1,2}\|\partial_r H\|_{\Linfty(\mathbb{R}^2)}.
\]

\begin{lemma}[Lipschitz bounds for \(\mathcal{F}\) and \(\mathcal{G}\)]\label{lem:sec2-lipschitz}
Fix \(R_a>0\). If \((\theta_\ell,a_\ell)\in X\) with
\(\|a_\ell\|_{\Linfty(\I^2)}\le R_a\) for \(\ell=1,2\), then
\begin{align}
\|\mathcal{F}_\omega(\theta_1,a_1)-\mathcal{F}_\omega(\theta_2,a_2)\|_{\Linfty(\I)}
&\le M_\Gamma\|a_1-a_2\|_{\Linfty(\I^2)}
+2R_aL_\Gamma\|\theta_1-\theta_2\|_{\Linfty(\I)},
\label{eq:sec2-lip-F}\\
\|\mathcal{G}(\theta_1,a_1)-\mathcal{G}(\theta_2,a_2)\|_{\Linfty(\I^2)}
&\le \|a_1-a_2\|_{\Linfty(\I^2)}
+2\Mdh\|\theta_1-\theta_2\|_{\Linfty(\I)}.
\label{eq:sec2-lip-G}
\end{align}
In particular, the vector field
\[
\mathcal{V}_\varepsilon(\theta,a)
:=
\big(\mathcal{F}_\omega(\theta,a),\,\varepsilon^{-1}\mathcal{G}(\theta,a)\big)
\]
is locally Lipschitz on \(X\).
\end{lemma}

\begin{proof}
For \(\mathcal{F}_\omega\), add and subtract
\(a_2(x,y)\Gamma(\theta_1(y)-\theta_1(x))\):
\begin{align*}
&|\mathcal{F}_\omega(\theta_1,a_1)(x)-\mathcal{F}_\omega(\theta_2,a_2)(x)| \\
&\le \int_0^1 |a_1-a_2|(x,y)\,|\Gamma(\theta_1(y)-\theta_1(x))|\,\dd y \\
&\quad + \int_0^1 |a_2(x,y)|\,
|\Gamma(\theta_1(y)-\theta_1(x)) - \Gamma(\theta_2(y)-\theta_2(x))|\,\dd y \\
&\le M_\Gamma\|a_1-a_2\|_{\Linfty(\I^2)}
+ R_aL_\Gamma\int_0^1
\big(|\theta_1(y)-\theta_2(y)|+|\theta_1(x)-\theta_2(x)|\big)\,\dd y \\
&\le M_\Gamma\|a_1-a_2\|_{\Linfty(\I^2)}
+2R_aL_\Gamma\|\theta_1-\theta_2\|_{\Linfty(\I)}.
\end{align*}
Taking the essential supremum in \(x\) gives \eqref{eq:sec2-lip-F}.

For \(\mathcal{G}\), we use
\begin{align*}
&|\mathcal{G}(\theta_1,a_1)(x,y)-\mathcal{G}(\theta_2,a_2)(x,y)| \\
&\le |a_1-a_2|(x,y)
 + |H(\theta_1(x),\theta_1(y)) - H(\theta_2(x),\theta_2(y))| \\
&\le \|a_1-a_2\|_{\Linfty(\I^2)}
+\Mdh\big(|\theta_1(x)-\theta_2(x)|+|\theta_1(y)-\theta_2(y)|\big) \\
&\le \|a_1-a_2\|_{\Linfty(\I^2)}
+2\Mdh\|\theta_1-\theta_2\|_{\Linfty(\I)}.
\end{align*}
Taking the essential supremum in \((x,y)\) gives \eqref{eq:sec2-lip-G}.
The local Lipschitz claim for \(\mathcal{V}_\varepsilon\) is immediate.
\end{proof}

\begin{lemma}[Variation-of-constants and \(a\)-bounds]\label{lem:sec2-a-bound}
Let \((\theta,a)\) be a strong solution of \eqref{eq:sec2-continuum} on \([0,T]\).
Then the following variation-of-constants formula holds in \(X_a\) for every
\(t\in[0,T]\). After choosing a jointly measurable representative, it holds for a.e.
\((x,y)\in\I^2\) and every \(t\in[0,T]\):
\begin{equation}
a(t,x,y)
= e^{-t/\varepsilon}a_0(x,y)
+\frac{1}{\varepsilon}\int_0^t
e^{-(t-s)/\varepsilon}H\!\big(\theta(s,x),\theta(s,y)\big)\,\dd s.
\label{eq:sec2-voc-a}
\end{equation}
Consequently,
\begin{equation}
\|a(t)\|_{\Linfty(\I^2)}
\le
e^{-t/\varepsilon}\|a_0\|_{\Linfty(\I^2)}
+\big(1-e^{-t/\varepsilon}\big)M_H
\le
\max\{\|a_0\|_{\Linfty(\I^2)},M_H\}.
\label{eq:sec2-a-priori}
\end{equation}
\end{lemma}

\begin{proof}
Set
\[
f(t):=H\bigl(\theta(t,\cdot),\theta(t,\cdot)\bigr)\in X_a.
\]
More explicitly, \(f(t)(x,y)=H(\theta(t,x),\theta(t,y))\), and
\[
\|f(t)-f(s)\|_{X_a}\le 2\Mdh\|\theta(t)-\theta(s)\|_{X_\theta};
\]
hence \(t\mapsto f(t)\) is continuous as an \(X_a\)-valued function. The fast equation is the
linear inhomogeneous equation
\[
\varepsilon a'(t)=-a(t)+f(t)
\]
in \(X_a\), or \(a'(t)=-\varepsilon^{-1}a(t)+\varepsilon^{-1}f(t)\), where
\(-\varepsilon^{-1}I\) generates the strongly continuous semigroup \(e^{-t/\varepsilon}I\).
Duhamel's formula therefore gives
\[
a(t)=e^{-t/\varepsilon}a_0+\frac1\varepsilon\int_0^t e^{-(t-s)/\varepsilon}
H\bigl(\theta(s,\cdot),\theta(s,\cdot)\bigr)\,\dd s
\quad\text{in }X_a,
\]
which is \eqref{eq:sec2-voc-a}. Taking \(\|\cdot\|_{X_a}\) and using
\(\|H(\theta(s,\cdot),\theta(s,\cdot))\|_{X_a}\le M_H\) gives
\eqref{eq:sec2-a-priori}.
\end{proof}

\begin{proposition}[Local well-posedness in \(X\)]\label{prop:sec2-local}
Under Assumption~\ref{ass:sec2-main}\,\ref{ass:sec2-A1}--\ref{ass:sec2-A3}, for every initial datum
\((\theta_0,a_0)\in X\), there exists \(T_{\mathrm{loc}}>0\) and a unique strong solution of
\eqref{eq:sec2-continuum} on \([0,T_{\mathrm{loc}}]\).
\end{proposition}

\begin{proof}
By Lemma~\ref{lem:sec2-lipschitz}, \(\mathcal{V}_\varepsilon\) is locally Lipschitz on the Banach
space \(X\). Hence Picard--Lindel\"of for ODEs in Banach spaces yields local existence and uniqueness.
\end{proof}

 Lemmas~\ref{lem:sec2-lipschitz}--\ref{lem:sec2-a-bound} and
Proposition~\ref{prop:sec2-local} supply the local part of the theory. The next step is the
global comparison estimate: once the fast variable is uniformly controlled, two solutions can be
tracked against one another over finite time intervals, and the resulting Gr\"onwall argument is
exactly what drives both continuous dependence and the continuum-limit statement.

\begin{proposition}[Dissipative stability estimate on finite intervals]\label{prop:sec2-stability}
Assume Assumption~\ref{ass:sec2-main}\,\ref{ass:sec2-A1}. For \(\ell=1,2\), let
\(\omega^\ell\in\Linfty(\I)\), \((\theta_0^\ell,a_0^\ell)\in X\), and let
\((\theta^\ell,a^\ell)\) be strong solutions of \eqref{eq:sec2-continuum} on \([0,T]\) with
these data. Define
\[
\delta\theta:=\theta^1-\theta^2,\qquad
\delta a:=a^1-a^2,\qquad
\delta\omega:=\omega^1-\omega^2,
\]
and set
\[
R_a:=\max\{\|a_0^1\|_{\Linfty(\I^2)},\|a_0^2\|_{\Linfty(\I^2)},M_H\},
\qquad
\Lambda_a:=2R_aL_\Gamma+2M_\Gamma\Mdh.
\]
Then for every \(t\in[0,T]\),
\begin{equation}
\|\delta\theta(t)\|_{\Linfty(\I)}
\le
e^{\Lambda_a t}
\Big(
\|\delta\theta(0)\|_{\Linfty(\I)}
+t\|\delta\omega\|_{\Linfty(\I)}
+
\varepsilon M_\Gamma\|\delta a(0)\|_{\Linfty(\I^2)}
\Big).
\label{eq:sec2-stability-theta}
\end{equation}
Moreover, we conclude
\begin{equation}
\begin{aligned}
\|\delta a(t)\|_{\Linfty(\I^2)}
&\le
e^{-t/\varepsilon}\|\delta a(0)\|_{\Linfty(\I^2)}
\\
&\quad+
2\Mdh e^{\Lambda_a t}
\Big(
\|\delta\theta(0)\|_{\Linfty(\I)}
+t\|\delta\omega\|_{\Linfty(\I)}
+\varepsilon M_\Gamma\|\delta a(0)\|_{\Linfty(\I^2)}
\Big).
\end{aligned}
\label{eq:sec2-stability-a}
\end{equation}
Consequently, it follows that
\begin{equation}
\begin{aligned}
&\|\delta\theta(t)\|_{\Linfty(\I)}
+
\|\delta a(t)\|_{\Linfty(\I^2)}
\\
&\quad\le
e^{-t/\varepsilon}\|\delta a(0)\|_{\Linfty(\I^2)}
\\
&\quad+
(1+2\Mdh)e^{\Lambda_a t}
\Big(
\|\delta\theta(0)\|_{\Linfty(\I)}
+t\|\delta\omega\|_{\Linfty(\I)}
+\varepsilon M_\Gamma\|\delta a(0)\|_{\Linfty(\I^2)}
\Big).
\end{aligned}
\label{eq:sec2-stability}
\end{equation}
\end{proposition}

\begin{proof}
By Lemma~\ref{lem:sec2-a-bound},
\(\|a^\ell(t)\|_{\Linfty(\I^2)}\le R_a\) for \(\ell=1,2\), hence Lemma~\ref{lem:sec2-lipschitz}
applies along both trajectories. Subtracting the two
variation-of-constants formulas from Lemma~\ref{lem:sec2-a-bound} and using the Lipschitz bound for \(H\) gives
\begin{equation}
\|\delta a(t)\|_{\Linfty}
\le
e^{-t/\varepsilon}\|\delta a(0)\|_{\Linfty}
+
\frac{2\Mdh}{\varepsilon}\int_0^t
e^{-(t-s)/\varepsilon}\|\delta\theta(s)\|_{\Linfty}\,\dd s.
\label{eq:sec2-delta-a-voc}
\end{equation}
Integrating \eqref{eq:sec2-delta-a-voc} in time and applying Fubini's theorem yields
\begin{equation}
\int_0^t \|\delta a(s)\|_{\Linfty}\,\dd s
\le
\varepsilon\|\delta a(0)\|_{\Linfty}
+
2\Mdh\int_0^t\|\delta\theta(s)\|_{\Linfty}\,\dd s,
\label{eq:sec2-delta-a-integral}
\end{equation}
because
\(\int_0^t e^{-s/\varepsilon}\,\dd s\le\varepsilon\) and
\(\int_r^t \varepsilon^{-1}e^{-(s-r)/\varepsilon}\,\dd s\le1\). Integrating only the slow equation and using
\eqref{eq:sec2-lip-F} gives
\begin{align*}
\|\delta\theta(t)\|_{\Linfty}
&\le \|\delta\theta(0)\|_{\Linfty}
+t\|\delta\omega\|_{\Linfty}
+\int_0^t
\Big(M_\Gamma\|\delta a(s)\|_{\Linfty}
+2R_aL_\Gamma\|\delta\theta(s)\|_{\Linfty}\Big)\,\dd s.
\end{align*}
Substituting \eqref{eq:sec2-delta-a-integral} into this bound gives
\[
\|\delta\theta(t)\|_{\Linfty}
\le
\|\delta\theta(0)\|_{\Linfty}
+t\|\delta\omega\|_{\Linfty}
+\varepsilon M_\Gamma\|\delta a(0)\|_{\Linfty}
+\Lambda_a\int_0^t\|\delta\theta(s)\|_{\Linfty}\,\dd s.
\]
The integral form of Gr\"onwall's inequality yields
\eqref{eq:sec2-stability-theta}. Inserting that bound into
\eqref{eq:sec2-delta-a-voc} and using the monotonicity of the right-hand side in time gives
\eqref{eq:sec2-stability-a}; summing the two estimates gives \eqref{eq:sec2-stability}.
\end{proof}

 At this point the two main theorems are short consequences of the preceding
estimates. Global well-posedness comes from combining local existence with the explicit uniform
bound on \(a\), while convergence is obtained by applying the same stability estimate to the
embedded finite-\(N\) trajectory and the continuum solution.

\begin{proof}[Proof of Theorem~\ref{thm:sec2-wp}]
Local existence and uniqueness follow from Proposition~\ref{prop:sec2-local}.
To prove global existence, fix a local solution on \([0,T_{\mathrm{loc}}]\).
Lemma~\ref{lem:sec2-a-bound} gives
\[
\|a(t)\|_{\Linfty(\I^2)}\le R_a^*:=\max\{\|a_0\|_{\Linfty(\I^2)},M_H\}
\]
for all times where the solution exists. Therefore
\[
\|\partial_t\theta(t)\|_{\Linfty(\I)}
\le \|\omega\|_{\Linfty(\I)} + M_\Gamma R_a^*,
\]
and therefore
\[
\|\theta(t)\|_{\Linfty(\I)}
\le \|\theta_0\|_{\Linfty(\I)}
+t\bigl(\|\omega\|_{\Linfty(\I)}+M_\Gamma R_a^*\bigr)
\]
on every interval on which the local solution exists. The Banach-space continuation criterion
therefore rules out finite-time blow-up, and the maximal existence time is infinite.

Uniqueness follows from \eqref{eq:sec2-stability} with identical data. To obtain the Lipschitz estimate on a ball of radius \(R\), apply
Proposition~\ref{prop:sec2-stability} to two solutions with data in that ball.
Then \(R_a\le \max\{R,M_H\}\), so
\[
\Lambda_a
\le
2L_\Gamma\max\{R,M_H\}+2M_\Gamma\Mdh.
\]
For every \(\varepsilon_0>0\), the right-hand side of
\eqref{eq:sec2-stability} is therefore bounded uniformly for
\(\varepsilon\in(0,\varepsilon_0]\) by replacing
\(\varepsilon M_\Gamma\) with \(\varepsilon_0M_\Gamma\). This gives the claimed finite-time
Lipschitz bound into \(C([0,T];X)\), uniformly for
\(\varepsilon\in(0,\varepsilon_0]\).
This proves the theorem.
\end{proof}

\begin{proof}[Proof of Theorem~\ref{thm:sec2-conv}]
For each \(N\), let \((\theta^N,a^N)\) denote the embedded step-field solution of
\eqref{eq:sec2-discrete}--\eqref{eq:sec2-step-embedding}, and let \((\theta,a)\) solve
\eqref{eq:sec2-continuum}. By Lemma~\ref{lem:sec2-stepfield-preservation},
\((\theta^N,a^N)\) is a strong solution of~\eqref{eq:sec2-continuum} with data
\((\theta_0^N,a_0^N,\omega^N)\). Hence we may apply Proposition~\ref{prop:sec2-stability} with
\((\theta^1,a^1,\omega^1)=(\theta^N,a^N,\omega^N)\) and
\((\theta^2,a^2,\omega^2)=(\theta,a,\omega)\). Using
\(R_a\le R_{a,*}\) in \eqref{eq:sec2-stability-theta} and
\eqref{eq:sec2-stability-a} gives the two componentwise bounds stated in the theorem, and hence
\[
\sup_{t\in[0,T]}
\|(\theta^N,a^N)(t)-(\theta,a)(t)\|_X
\le
E_a^N
+
(1+2\Mdh)e^{\Lambda_*T}
\bigl(E_\theta^N+TE_\omega^N+\varepsilon M_\Gamma E_a^N\bigr).
\]
Assumption~\ref{ass:sec2-main}\,\ref{ass:sec2-A4}
implies the right-hand side tends to zero, proving both the quantitative estimate and convergence.
The final uniform bound for \(\varepsilon\in(0,\varepsilon_0]\) follows by replacing
\(\varepsilon M_\Gamma\) with \(\varepsilon_0M_\Gamma\).
\end{proof}

\begin{remark}[Uniform finite-time control as \(\varepsilon\downarrow0\)]\label{rem:sec2-eps-uniformity}
Although the slow-time vector field contains the singular prefactor
\(\varepsilon^{-1}\), the fast equation is contractive. Proposition~\ref{prop:sec2-stability}
uses the variation-of-constants formula for \(a\) rather than the local Lipschitz constant of the
full vector field, and therefore yields finite-time \(C([0,T];X)\) bounds that are uniform for
\(\varepsilon\in(0,\varepsilon_0]\). The initial mismatch in the fast variable contributes only
through \(\varepsilon\|\delta a(0)\|_{\Linfty}\) to the slow component, reflecting its initial-layer
character; the full \(X\)-norm estimate still contains \(\|\delta a(0)\|_{\Linfty}\) because the
interval includes \(t=0\).
\end{remark}

\begin{corollary}[Explicit \(O(1/N)\) rate for Lipschitz data]\label{cor:sec2-rate}
Assume Assumption~\ref{ass:sec2-main}\,\ref{ass:sec2-A1}--\ref{ass:sec2-A3} and that \(\omega\) is Lipschitz on \(\I\) with constant \(L_\omega\), \(\theta_0\) is Lipschitz on \(\I\) with constant \(L_{\theta_0}\), and \(a_0\) is Lipschitz on \(\I^2\) with constant \(L_{a_0}\). Define the equal-cell step approximations \((\omega^N,\theta_0^N,a_0^N)\) by cell-evaluation at any chosen point of each cell. Then
\[
\|\omega^N-\omega\|_{\Linfty(\I)}
+\|\theta_0^N-\theta_0\|_{\Linfty(\I)}
+\|a_0^N-a_0\|_{\Linfty(\I^2)}
\le \frac{L_\omega+L_{\theta_0}+\sqrt{2}\,L_{a_0}}{N},
\]
and consequently, with \(\Lambda_*\) from Theorem~\ref{thm:sec2-conv},
\[
\sup_{t\in[0,T]}\|(\theta^N,a^N)(t)-(\theta,a)(t)\|_X
\le
\frac{
\sqrt{2}\,L_{a_0}
+
(1+2\Mdh)e^{\Lambda_*T}
\bigl(L_{\theta_0}+TL_\omega+\varepsilon M_\Gamma\sqrt{2}\,L_{a_0}\bigr)
}{N}.
\]
\end{corollary}

\begin{proof}
On any cell \(I_i^N\) of length \(1/N\), the Lipschitz hypothesis on \(\omega\) gives \(|\omega(x)-\omega(x')|\le L_\omega/N\) for all \(x,x'\in I_i^N\); the same applies to \(\theta_0\) on \(\I\) and to \(a_0\) on the \(2\)-dimensional cell \(I_i^N\times I_j^N\) of diameter \(\sqrt{2}/N\). Cell-evaluation therefore gives the displayed step-approximation bound. Substituting into Theorem~\ref{thm:sec2-conv} yields the trajectory bound.
\end{proof}

 Section~\ref{sec:unreduced-model} therefore supplies more than a pair of existence
and step-approximation-based convergence statements. It identifies the
two structural inputs that will be reused later: the explicit relaxation formula
\eqref{eq:sec2-voc-a} for the fast variable, and the stability estimate
\eqref{eq:sec2-stability}, whose role reappears in the reduced setting of
Section~\ref{sec:reduced-model}. The continuum system \eqref{eq:sec2-continuum} is also the
starting point for the infinite-dimensional Lyapunov--Perron slow-manifold construction carried out in
Section~\ref{sec:infdim-reduction}.

\section{Continuum Limit of the Reduced Higher-Order Model}\label{sec:reduced-model}

This section treats the truncated reduce-first route of the comparison diagram:
start from the explicit first-order reduced finite-\(N\) phase dynamics obtained by Fenichel
reduction at the discrete level, then pass to \(N\to\infty\). Thus the continuum
limit taken in this section is the limit of the first-order truncation, not yet the continuum
limit of the exact finite-\(N\) reduced flow associated with a chosen slow-manifold graph. The
Lyapunov--Perron selected finite-\(N\) representative used for the exact comparison is brought back in Section~\ref{sec:identification},
where the \(O(\varepsilon^2)\) remainder is controlled uniformly on step fields. The analysis is structurally parallel to
Section~\ref{sec:unreduced-model}, but now the state space is the single slow variable
\(\theta\in\Linfty(\I)\), and the vector field carries the pairwise operator \(K\), the
first-order correction \(P\), and the triplet operator \(T\).

We flag the structural asymmetry with Section~\ref{sec:unreduced-model} explicitly. There we took the continuum limit of the \emph{exact} finite-\(N\) unreduced fast--slow system. Here we take the continuum limit of the \emph{first-order truncation} of the finite-\(N\) reduced system, not of the exact finite-\(N\) reduced flow associated with a chosen slow-manifold representative. This is by design: the exact finite-\(N\) slow-manifold graph carries an \(O(\varepsilon^2)\) remainder whose constant is not a priori uniform in \(N\) (Remark~\ref{rem:sec3-order}), and the cleanest route to a uniform-in-\(N\) bound is to construct the slow manifold first at the continuum level (Section~\ref{sec:infdim-reduction}) and then transfer the bound to step fields. The exact reduce-first vs continuum-first comparison, with uniform \(O(\varepsilon^2)\) graph remainder, is therefore deferred to Corollary~\ref{cor:sec5-uniform-order}.

 More precisely, we start from the explicit order-\(\varepsilon\) truncation of
the finite-\(N\) Fenichel-reduced phase dynamics, embed them in the dense-label continuum framework, and identify the corresponding
continuum reduced equation on \(\Linfty(\I)\). The section then proves global well-posedness for
that continuum dynamics and shows that the embedded finite-\(N\) truncated reduced trajectories 
converge to it on every finite time interval for data satisfying the
step-approximation hypothesis introduced below. Together with
Section~\ref{sec:unreduced-model}, this supplies the two continuum objects that must later be
compared in Section~\ref{sec:identification}.

 Relative to Section~\ref{sec:unreduced-model}, the analytical emphasis now shifts.
The fast variable has disappeared, so there is no longer a separate relaxation mechanism to
exploit; instead, the main task is to show that the reduced vector field built from \(K\), \(P\),
and especially the triplet term \(T\) is well behaved on \(\Linfty(\I)\). In other words, the
difficulty moves from coupled fast--slow dynamics to the structure of the nonlinear continuum
operators themselves.

\begin{remark}[Order of the Fenichel remainder]\label{rem:sec3-order}

The \(o(\varepsilon)\) remainder in the finite-dimensional paper~\cite{KuehnMurphy} was
sufficient for the purposes of that work: the main structural step there was a mixed-derivative
test showing that the triplet contribution cannot be absorbed into a pairwise representation,
and for that argument it is enough to know that the first-order triplet coefficient survives
after dividing by \(\varepsilon\) and sending \(\varepsilon\to 0\).

Under the finite-dimensional regularity hypotheses that give the required
\(C^2\)-control of a chosen finite-dimensional slow-manifold graph, one can sharpen the remainder
to \(O(\varepsilon^2)\) at each fixed \(N\). In the comparison below this chosen graph is the
Lyapunov--Perron selected representative; it has the same first-order expansion as the
finite-dimensional slow-manifold graph used in~\cite{KuehnMurphy}. Let
\(h_\varepsilon^N:\mathbb{T}^N\to\mathbb{R}^{N\times N}\) denote this representative, and write
\(h_\varepsilon^N=h_0^N+\varepsilon h_1^N+r_\varepsilon^N\). The expansion
\(h_\varepsilon^N=h_0^N+\varepsilon h_1^N+o(\varepsilon)\) in \(C^2(\mathbb{T}^N)\) proved there,
together with compactness of \(\mathbb{T}^N\), yields
\[
\|h_\varepsilon^N-h_0^N\|_{C^1}=O(\varepsilon),
\qquad
\|\DD h_\varepsilon^N-\DD h_0^N\|_{C^0}=O(\varepsilon).
\]
When this finite-dimensional torus graph is compared below with continuum step fields, we evaluate
it on the torus class \([\vartheta]\in\mathbb{T}^N\) of a chosen lift
\(\vartheta\in\mathbb{R}^N\). The periodicity assumptions on \(\Gamma\) and \(H\) ensure that the
resulting finite-dimensional vector field and graph are independent of the chosen lift.
Let \(f^N(\theta,a)_i=\omega_i^N+\tfrac1N\sum_{j=1}^N a_{ij}\Gamma(\theta_j-\theta_i)\) denote
the slow vector field of the unreduced network. Substituting the expansion into the invariance
equation \(-h_\varepsilon^N+h_0^N=\varepsilon\,\DD h_\varepsilon^N\,f^N(\theta,h_\varepsilon^N)\),
using the defining relation \(h_1^N=-\DD h_0^N\,f^N(\theta,h_0^N)\), and cancelling the
\(O(\varepsilon)\) contributions gives
\[
r_\varepsilon^N(\theta)
=
-\varepsilon\Big(
\DD h_\varepsilon^N\,f^N(\theta,h_\varepsilon^N)
-\DD h_0^N\,f^N(\theta,h_0^N)
\Big).
\]
The bracket is \(O(\varepsilon)\) uniformly in \(\theta\in\mathbb{T}^N\) by the \(C^1\)- and
\(C^0\)-bounds above and the smoothness of \(f^N\), so
\(\|r_\varepsilon^N\|_{C^0}=O(\varepsilon^2)\); the reduced vector field therefore carries an
\(O(\varepsilon^2)\) remainder at each fixed \(N\).

The implicit constant in this argument is not, however, a priori uniform in \(N\): the
\(C^1\)- and \(C^0\)-bounds on \(h_\varepsilon^N-h_0^N\) originate from the finite-dimensional
Fenichel theorem on \(\mathbb{T}^N\times\R^{N\times N}\), whose constants may depend on the
ambient dimension. Since the reduce-first route of the present section takes \(N\to\infty\) at
fixed \(\varepsilon\), such a uniform bound is required before the \(O(\varepsilon^2)\)
remainder can be passed through the dense-graph limit. A uniform bound is established in
Section~\ref{sec:identification} by transferring the continuum-level estimate of
Proposition~\ref{prop:sec5-h1} through a step-field embedding; see
Corollary~\ref{cor:sec5-uniform-order}.

\end{remark}

\subsection[Reduced Dynamics and the Continuum Equation]{Reduced dynamics and the continuum equation}

 We begin with the reduced finite-\(N\) model and its dense-label embedding, because
this is the most direct way to see which continuum terms have to appear in the reduce-first limit.
The key point is that the reduced dynamics now evolve on a single phase field, but their vector
field already contains both pairwise and triplet contributions. The continuum nonpairwise
character of the triplet operator is analysed later in Section~\ref{sec:identification}.

We work with the first-order truncated reduced vector field from the finite-dimensional paper~\cite{KuehnMurphy},
not with the exact finite-\(N\) reduced vector field associated with a chosen
slow-manifold graph. Equation~\eqref{eq:sec3-discrete-red}
should therefore be read as the first-order model whose continuum limit is analysed in this
section; the exact finite-\(N\) selected reduced vector field differs from it by a remainder that is
estimated uniformly only later in Corollary~\ref{cor:sec5-uniform-order}. For each
\(N\in\mathbb{N}\),
\begin{align}
\dot{\theta}_i^N
&= \omega_i^N
+ \frac{1}{N}\sum_{j=1}^N H(\theta_i^N,\theta_j^N)\Gamma(\theta_j^N-\theta_i^N)
+ \frac{\varepsilon}{N}\sum_{j=1}^N P_{ij}^N(\theta^N)
+ \frac{\varepsilon}{N^2}\sum_{j,k=1}^N T_{ijk}^N(\theta^N),
\label{eq:sec3-discrete-red}
\end{align}
where
\begin{align}
P_{ij}^N(\theta^N)
&:=
-\Gamma(\theta_j^N-\theta_i^N)
\Big(\partial_1H(\theta_i^N,\theta_j^N)\omega_i^N
+\partial_2H(\theta_i^N,\theta_j^N)\omega_j^N\Big),
\label{eq:sec3-Pij}\\
T_{ijk}^N(\theta^N)
&:=
-\Gamma(\theta_j^N-\theta_i^N)\partial_1H(\theta_i^N,\theta_j^N)H(\theta_i^N,\theta_k^N)\Gamma(\theta_k^N-\theta_i^N)
\nonumber\\
&\quad
-\Gamma(\theta_j^N-\theta_i^N)\partial_2H(\theta_i^N,\theta_j^N)H(\theta_j^N,\theta_k^N)\Gamma(\theta_k^N-\theta_j^N).
\label{eq:sec3-Tijk}
\end{align}

As in Section~\ref{sec:unreduced-model}, define step fields on \(\I\):
\[
\theta^N(t,x):=\theta_i^N(t)\quad (x\in I_i^N),
\qquad
\omega^N(x):=\omega_i^N\quad (x\in I_i^N).
\]
The \(1/N\)-normalisation in \eqref{eq:sec3-discrete-red} is the dense-interaction scaling.

 The dense-label embedding leaves only the slow variable \(\theta\) as a state
variable, so the ambient space is simpler than in Section~\ref{sec:unreduced-model}. The
analytical burden, however, has moved into the nonlinear operators themselves, and we therefore
record the function-space setting, hypotheses, and continuum vector field explicitly before stating
the main theorems.

Set
\[
X_{\mathrm{red}}:=\Linfty(\I),
\qquad
\|\theta\|_{X_{\mathrm{red}}}:=\|\theta\|_{\Linfty(\I)}.
\]

\begin{assumption}[Reduced-model hypotheses]\label{ass:sec3-main}
Fix \(\varepsilon>0\). Assume:
\begin{enumerate}[label=\textup{(B\arabic*)}, ref=\textup{(B\arabic*)}]
  \item\label{ass:sec3-B1} \(\Gamma\in C_b^1(\mathbb{R})\) is \(2\pi\)-periodic, and
  \(H\in C_b^2(\mathbb{R}^2)\) is \(2\pi\)-periodic in each argument.
  \item\label{ass:sec3-B2} \(\omega,\theta_0\in \Linfty(\I)\).
  \item\label{ass:sec3-B3} There exist step approximations \(\theta_0^N,\omega^N\) on the standard
  equal-cell partition such that
  \[
  \|\theta_0^N-\theta_0\|_{\Linfty(\I)}+\|\omega^N-\omega\|_{\Linfty(\I)}\to 0.
  \]
\end{enumerate}
\end{assumption}

\begin{remark}[Uniform frequency bound from the step hypothesis]\label{rem:sec3-step-frequency-bound}
Since \(\omega^N\to\omega\) in \(L^\infty(\I)\) by
Assumption~\ref{ass:sec3-main}\,\ref{ass:sec3-B3}, the sequence \((\omega^N)_N\) is bounded in
\(L^\infty(\I)\). Hence
\[
M_\omega^{\mathrm{step}}
:=
\max\Bigl\{\|\omega\|_{\Linfty(\I)},\sup_{N\ge 1}\|\omega^N\|_{\Linfty(\I)}\Bigr\}
<\infty.
\]
We use \(M_\omega^{\mathrm{step}}\) whenever constants must be uniform in \(N\).
\end{remark}

 The continuum equation defined below is posed on all of \(X_{\mathrm{red}}\), but
the discrete-to-continuum theorem below is restricted by Assumption~\ref{ass:sec3-main}\,\ref{ass:sec3-B3} to
data admitting \(L^\infty\)-convergent equal-cell step approximations. Sufficient classes for Assumption~\ref{ass:sec3-main}\,\ref{ass:sec3-B3} include continuous data on \(\I\), and more generally \(L^\infty\) data admitting a uniformly continuous representative on the standard equal-cell partitions. Lipschitz data give an explicit \(O(1/N)\) approximation rate, which propagates to the discrete-to-continuum bound; see Corollary~\ref{cor:sec3-rate} below.

Define operators \(K,P,T:X_{\mathrm{red}}\to X_{\mathrm{red}}\) (with \(P\) also depending on \(\omega\)):
\begin{align}
K[\theta](x)
&:= \int_0^1 H\!\big(\theta(x),\theta(y)\big)\Gamma\!\big(\theta(y)-\theta(x)\big)\,\dd y,
\label{eq:sec3-Kop}\\
P[\theta;\omega](x)
&:= -\int_0^1 \Gamma\!\big(\theta(y)-\theta(x)\big)
\Big(\partial_1H\!\big(\theta(x),\theta(y)\big)\omega(x)
+\partial_2H\!\big(\theta(x),\theta(y)\big)\omega(y)\Big)\,\dd y,
\label{eq:sec3-Pop}\\
T[\theta](x)
&:= -\int_0^1\!\int_0^1
\Gamma\!\big(\theta(y)-\theta(x)\big)\partial_1H\!\big(\theta(x),\theta(y)\big)
H\!\big(\theta(x),\theta(z)\big)\Gamma\!\big(\theta(z)-\theta(x)\big)\,\dd z\,\dd y
\nonumber\\
&\quad
-\int_0^1\!\int_0^1
\Gamma\!\big(\theta(y)-\theta(x)\big)\partial_2H\!\big(\theta(x),\theta(y)\big)
H\!\big(\theta(y),\theta(z)\big)\Gamma\!\big(\theta(z)-\theta(y)\big)\,\dd z\,\dd y.
\label{eq:sec3-Top}
\end{align}
The continuum equation associated with this first-order truncated reduce-first model is
\begin{equation}
\partial_t\theta = \omega + K[\theta] + \varepsilon P[\theta;\omega] + \varepsilon T[\theta],
\qquad
\theta|_{t=0}=\theta_0.
\label{eq:sec3-continuum-red}
\end{equation}

 The three operators play distinct roles. The term \(K\) is the leading-order
pairwise continuum interaction, \(P\) is the first-order pairwise correction weighted by the
frequencies, and \(T\) is the first-order triplet contribution. Writing the reduced equation in
this decomposed form makes clear which pieces must later be matched by the continuum-first
reduction and tested for continuum nonpairwise representability in
Section~\ref{sec:identification}.

\begin{definition}[Strong reduced solution on finite horizon]
A function \(\theta\) is a strong solution of \eqref{eq:sec3-continuum-red} on \([0,T]\) if
\(\theta\in W^{1,\infty}([0,T];X_{\mathrm{red}})\) and
\eqref{eq:sec3-continuum-red} holds for a.e. \(t\in[0,T]\).
\end{definition}

\begin{lemma}[Step-field invariance of the truncated reduced continuum flow]
\label{lem:sec3-stepfield-preservation}
Fix \(N\in\mathbb N\) and let \(\omega=\omega^N\in \mathfrak{S}_N^\theta\). If
\(\theta_0\in \mathfrak{S}_N^\theta\), then the unique strong solution (provided by Theorem~\ref{thm:sec3-wp} below) of
\eqref{eq:sec3-continuum-red} with data \((\theta_0,\omega)\) satisfies
\(\theta(t)\in \mathfrak{S}_N^\theta\) for all \(t\ge 0\), and its step-coefficients solve
\eqref{eq:sec3-discrete-red}.
\end{lemma}

\begin{proof}
For \(x\in I_i^N\), the operators \(K[\theta](x)\), \(P[\theta;\omega](x)\), and
\(T[\theta](x)\) are constant on \(I_i^N\) whenever \(\theta,\omega\in \mathfrak{S}_N^\theta\). Indeed,
their integrals over cells \(I_j^N\) and \(I_k^N\) reduce exactly to the \(1/N\)- and
\(1/N^2\)-weighted sums in \eqref{eq:sec3-discrete-red}. Hence the reduced continuum vector
field is tangent to \(\mathfrak{S}_N^\theta\), and Picard--Lindel\"of on the finite-dimensional subspace
identifies the restricted flow with the finite-\(N\) reduced ODE.
\end{proof}

 With the reduced continuum equation now fixed, the two main results can be stated
succinctly. The first theorem gives global well-posedness in \(X_{\mathrm{red}}\), and the second
shows that the dense-label embedding of the finite-\(N\) reduced dynamics converges to this
continuum equation on finite time intervals for data satisfying
Assumption~\ref{ass:sec3-main}\,\ref{ass:sec3-B3}.

\begin{theorem}[Well-posedness of reduced continuum dynamics]\label{thm:sec3-wp}
Under Assumption~\ref{ass:sec3-main}\,\ref{ass:sec3-B1}--\ref{ass:sec3-B2}, equation~\eqref{eq:sec3-continuum-red}
has a unique global strong solution for every \(\varepsilon>0\).
Moreover, for each \(T>0\) and \(R>0\), the map
\((\theta_0,\omega)\mapsto \theta\) is Lipschitz on the \(R\)-ball in
\(\Linfty(\I)\times\Linfty(\I)\), as a map into \(C([0,T];X_{\mathrm{red}})\).
\end{theorem}

\begin{theorem}[Truncated reduced discrete-to-continuum convergence under the step-approximation hypothesis]\label{thm:sec3-conv}
Assume Assumption~\ref{ass:sec3-main}. Let \(\theta^N\) be the step-field solution induced by
\eqref{eq:sec3-discrete-red} with initial datum \(\theta_0^N\) and frequency field \(\omega^N\)
from Assumption~\ref{ass:sec3-main}\,\ref{ass:sec3-B3}, and let \(\theta\) solve
\eqref{eq:sec3-continuum-red}.
Then for every fixed \(T>0\),
\[
\sup_{t\in[0,T]}\|\theta^N(t)-\theta(t)\|_{\Linfty(\I)}\to 0
\qquad (N\to\infty).
\]
More precisely, with constants from Proposition~\ref{prop:sec3-stability},
\[
\sup_{t\in[0,T]}\|\theta^N(t)-\theta(t)\|_{\Linfty(\I)}
\le C_T^{\mathrm{red}}
\Big(\|\theta_0^N-\theta_0\|_{\Linfty(\I)}+\|\omega^N-\omega\|_{\Linfty(\I)}\Big),
\]
where \(C_T^{\mathrm{red}}\) is independent of \(N\).
\end{theorem}

 The rest of the section mirrors the logic of Section~\ref{sec:unreduced-model}, but
with the operator structure now taking centre stage. We first prove boundedness and Lipschitz
estimates for \(K\), \(P\), and \(T\) on bounded subsets of \(\Linfty(\I)\). Once these estimates
are in place, Picard--Lindel\"of and the uniform bound on the reduced vector field yield global
well-posedness. A final comparison estimate, obtained by applying Gr\"onwall's
inequality to the difference of two solutions, then yields both Lipschitz dependence on the data and
the discrete-to-continuum convergence statement.

\subsection[Operator Bounds, Stability, and Convergence]{Operator bounds, stability, and convergence}

 We now turn to the proof in one continuous chain. First we establish the operator
bounds for \(K\), \(P\), and \(T\); next we use those bounds to place the reduced equation in the
standard Banach-space ODE framework and obtain global well-posedness; finally we compare two
solutions on a fixed time interval to derive continuous dependence and the discrete-to-continuum
estimate.

 We therefore begin with the operator bounds, since they are the common input for the
entire section: they control the nonlinear reduced vector field, rule out finite-time blow-up, and
ultimately allow the discrete and continuum dynamics to be compared on a fixed time horizon.

The constants \(M_\Gamma,L_\Gamma,M_H,\Mdh\) are those already defined in
Section~\ref{sec:unreduced-model}. The higher regularity of \(H\) requires in addition the
second-derivative bound
\[
\MdTwoH:=\max_{r,s=1,2}\|\partial_{rs}^2H\|_{\Linfty(\mathbb{R}^2)}.
\]

\begin{lemma}[Bounds for reduced operators]\label{lem:sec3-op-bounds}
Let \(M_\omega>0\). There exist constants
\(C_K,C_P^{\theta},C_P^{\omega},C_T^{\theta}>0\), depending only on
\(M_\omega\), \(M_\Gamma\), \(L_\Gamma\), \(M_H\), \(\Mdh\), \(\MdTwoH\), such that
for all \(\theta_1,\theta_2,\omega_1,\omega_2\in \Linfty(\I)\) with
\(\|\omega_\ell\|_{\Linfty(\I)}\le M_\omega\), \(\ell=1,2\),
\begin{align}
\|K[\theta_1]-K[\theta_2]\|_{\Linfty(\I)}
&\le C_K\|\theta_1-\theta_2\|_{\Linfty(\I)},
\label{eq:sec3-lip-K}\\
\|P[\theta_1;\omega_1]-P[\theta_2;\omega_2]\|_{\Linfty(\I)}
&\le C_P^{\theta}\|\theta_1-\theta_2\|_{\Linfty(\I)}
+C_P^{\omega}\|\omega_1-\omega_2\|_{\Linfty(\I)},
\label{eq:sec3-lip-P}\\
\|T[\theta_1]-T[\theta_2]\|_{\Linfty(\I)}
&\le C_T^{\theta}\|\theta_1-\theta_2\|_{\Linfty(\I)}.
\label{eq:sec3-lip-T}
\end{align}
Moreover,
\begin{align}
\|K[\theta]\|_{\Linfty(\I)} &\le M_HM_\Gamma,
\label{eq:sec3-bd-K}\\
\|P[\theta;\omega]\|_{\Linfty(\I)} &\le 2M_\Gamma \Mdh\|\omega\|_{\Linfty(\I)},
\label{eq:sec3-bd-P}\\
\|T[\theta]\|_{\Linfty(\I)} &\le 2M_H\Mdh M_\Gamma^2.
\label{eq:sec3-bd-T}
\end{align}
\end{lemma}

\begin{proof}
The bounds \eqref{eq:sec3-bd-K}--\eqref{eq:sec3-bd-T} follow immediately from
\eqref{eq:sec3-Kop}--\eqref{eq:sec3-Top} and \(|\I|=1\).

\emph{Lipschitz bound for \(K\).}
Add and subtract \(H(\theta_2(x),\theta_2(y))\Gamma(\theta_1(y)-\theta_1(x))\) in the integrand:
\begin{align*}
|K[\theta_1](x)-K[\theta_2](x)|
&\le \int_0^1 |H(\theta_1(x),\theta_1(y))-H(\theta_2(x),\theta_2(y))|\,
|\Gamma(\theta_1(y)-\theta_1(x))|\,\dd y \\
&\quad + \int_0^1 |H(\theta_2(x),\theta_2(y))|\,
|\Gamma(\theta_1(y)-\theta_1(x))-\Gamma(\theta_2(y)-\theta_2(x))|\,\dd y.
\end{align*}
The first integral is bounded by \(2\Mdh M_\Gamma\|\theta_1-\theta_2\|_{\Linfty}\) and the second
by \(2M_HL_\Gamma\|\theta_1-\theta_2\|_{\Linfty}\), giving
\(C_K=2(\Mdh M_\Gamma+M_HL_\Gamma)\).

\emph{Lipschitz bound for \(P\).}
Split
\[
P[\theta_1;\omega_1]-P[\theta_2;\omega_2]
=\big(P[\theta_1;\omega_1]-P[\theta_2;\omega_1]\big)
+\big(P[\theta_2;\omega_1]-P[\theta_2;\omega_2]\big).
\]
The second term is linear in \(\omega\) with bounded kernel, giving the
\(C_P^{\omega}\|\omega_1-\omega_2\|_{\Linfty}\) contribution with
\(C_P^{\omega}=2M_\Gamma \Mdh\).
For the first term, each of the two summands in \(P\) contains the product
\(\Gamma(\theta(y)-\theta(x))\partial_rH(\theta(x),\theta(y))\omega(\cdot)\). The
\(\Gamma\)-factor contributes at most \(2L_\Gamma \Mdh M_\omega
\|\theta_1-\theta_2\|_{\Linfty}\), and the \(\partial_rH\)-factor contributes at most
\(2M_\Gamma \MdTwoH M_\omega\|\theta_1-\theta_2\|_{\Linfty}\). Summing over
\(r=1,2\), we may take
\[
C_P^\theta=4M_\omega\bigl(L_\Gamma\Mdh+M_\Gamma\MdTwoH\bigr).
\]

\emph{Lipschitz bound for \(T\).}
Each of the two integrands in \eqref{eq:sec3-Top} is a product of four bounded
Lipschitz factors depending on \(\theta(x),\theta(y),\theta(z)\). Using the telescoping identity
\[
\textstyle\prod_{m=1}^4 a_m-\prod_{m=1}^4 b_m
=\sum_{m=1}^4 (a_m-b_m)\prod_{n<m}b_n\prod_{n>m}a_n,
\]
one double-integral term is bounded by
\[
2\Bigl(2L_\Gamma\Mdh M_HM_\Gamma+\MdTwoH M_HM_\Gamma^2+\Mdh^2M_\Gamma^2\Bigr)
\|\theta_1-\theta_2\|_{\Linfty}.
\]
The second double-integral term has the same bound, so \eqref{eq:sec3-lip-T}
holds with
\[
C_T^\theta
=
4\Bigl(2L_\Gamma\Mdh M_HM_\Gamma+\MdTwoH M_HM_\Gamma^2+\Mdh^2M_\Gamma^2\Bigr).
\]
\end{proof}

\begin{proposition}[Local and global well-posedness chain]\label{prop:sec3-local-global}
Under Assumption~\ref{ass:sec3-main}\,\ref{ass:sec3-B1}--\ref{ass:sec3-B2}, equation~\eqref{eq:sec3-continuum-red}
has a unique local strong solution.
In fact the solution is global, and
\begin{equation}
\|\partial_t\theta(t)\|_{\Linfty(\I)}
\le
\|\omega\|_{\Linfty(\I)} + M_HM_\Gamma
+\varepsilon\,2M_\Gamma \Mdh\|\omega\|_{\Linfty(\I)}
+\varepsilon\,2M_H\Mdh M_\Gamma^2.
\label{eq:sec3-time-der-bound}
\end{equation}
for a.e.\ \(t\ge 0\).
\end{proposition}

\begin{proof}
Local existence/uniqueness follows from Lemma~\ref{lem:sec3-op-bounds} and Picard--Lindel\"of on
\(X_{\mathrm{red}}\). Bound \eqref{eq:sec3-time-der-bound} is the direct estimate of the right-hand
side in \eqref{eq:sec3-continuum-red} using \eqref{eq:sec3-bd-K}--\eqref{eq:sec3-bd-T}.
Let \(M_{\partial}\) denote the right-hand side of
\eqref{eq:sec3-time-der-bound}. Integrating the derivative bound gives
\[
\|\theta(t)\|_{\Linfty(\I)}
\le
\|\theta_0\|_{\Linfty(\I)}+tM_{\partial}
\]
on every interval on which the local solution exists. Hence the
\(X_{\mathrm{red}}=\Linfty(\I)\) norm stays finite on every finite time interval, and the
standard Banach-space continuation criterion rules out finite-time blow-up. The local solution
therefore extends globally.
\end{proof}

 Once global existence is in place, the only remaining task is comparison. This is
the point where the decomposition into \(K\), \(P\), and \(T\) becomes most useful: the Lipschitz
bounds from Lemma~\ref{lem:sec3-op-bounds} feed directly into a Gr\"onwall
estimate for the difference of two solutions, which simultaneously yields continuous dependence on
the data and the continuum-limit bound.

\begin{proposition}[Finite-time stability for reduced dynamics]\label{prop:sec3-stability}
Assume Assumption~\ref{ass:sec3-main}\,\ref{ass:sec3-B1}. For \(\ell=1,2\), let
\(\theta_0^\ell,\omega^\ell\in\Linfty(\I)\), and let \(\theta^\ell\) solve
\eqref{eq:sec3-continuum-red} with data
\((\theta_0^\ell,\omega^\ell)\).
Set
\[
\Lambda_\varepsilon:=C_K+\varepsilon C_P^{\theta}+\varepsilon C_T^{\theta},
\qquad
B_\varepsilon:=1+\varepsilon C_P^{\omega},
\]
where constants are from Lemma~\ref{lem:sec3-op-bounds} with
\(M_\omega=\max\{\|\omega^1\|_{\Linfty},\|\omega^2\|_{\Linfty}\}\).
Then for all \(t\in[0,T]\),
\begin{equation}
\|\theta^1(t)-\theta^2(t)\|_{\Linfty(\I)}
\le
e^{\Lambda_\varepsilon t}
\Big(
\|\theta_0^1-\theta_0^2\|_{\Linfty(\I)}
+tB_\varepsilon\|\omega^1-\omega^2\|_{\Linfty(\I)}
\Big).
\label{eq:sec3-stability}
\end{equation}
\end{proposition}

\begin{proof}
Subtract the two reduced equations and use
\eqref{eq:sec3-lip-K}--\eqref{eq:sec3-lip-T}:
\[
\|\theta^1(t)-\theta^2(t)\|_{\Linfty}
\le
\|\theta_0^1-\theta_0^2\|_{\Linfty}
+tB_\varepsilon\|\omega^1-\omega^2\|_{\Linfty}
+\Lambda_\varepsilon\int_0^t \|\theta^1(s)-\theta^2(s)\|_{\Linfty}\,\dd s.
\]
Apply Gr\"onwall.
\end{proof}

 At this stage the proofs of the two main theorems are short. Theorem~\ref{thm:sec3-wp}
is simply the global existence statement together with the stability estimate on bounded sets of
data, while Theorem~\ref{thm:sec3-conv} is obtained by applying the same comparison argument to an
embedded finite-\(N\) reduced trajectory and the continuum solution.

\begin{proof}[Proof of Theorem~\ref{thm:sec3-wp}]
Global existence and uniqueness are given by
Proposition~\ref{prop:sec3-local-global}.
Lipschitz dependence on \((\theta_0,\omega)\) on \(R\)-balls follows from
Proposition~\ref{prop:sec3-stability} with
\(M_\omega\le R\).
\end{proof}

\begin{proof}[Proof of Theorem~\ref{thm:sec3-conv}]
By Lemma~\ref{lem:sec3-stepfield-preservation}, \(\theta^N\) is a strong solution of
\eqref{eq:sec3-continuum-red} with data \((\theta_0^N,\omega^N)\). Apply
Proposition~\ref{prop:sec3-stability} with
\((\theta^1,\omega^1)=(\theta^N,\omega^N)\) and
\((\theta^2,\omega^2)=(\theta,\omega)\). By
Remark~\ref{rem:sec3-step-frequency-bound}, the constant \(M_\omega\) in
Proposition~\ref{prop:sec3-stability} may be chosen uniformly in \(N\), namely
\(M_\omega\le M_\omega^{\mathrm{step}}\). This yields
\[
\sup_{t\in[0,T]}\|\theta^N(t)-\theta(t)\|_{\Linfty(\I)}
\le
e^{\Lambda_\varepsilon T}\max\{1,TB_\varepsilon\}
\Big(
\|\theta_0^N-\theta_0\|_{\Linfty(\I)}
+\|\omega^N-\omega\|_{\Linfty(\I)}
\Big).
\]
Set \(C_T^{\mathrm{red}}:=e^{\Lambda_\varepsilon T}\max\{1,TB_\varepsilon\}\).
Assumption~\ref{ass:sec3-main}\,\ref{ass:sec3-B3} implies convergence to zero.
\end{proof}

\begin{corollary}[Explicit \(O(1/N)\) rate for Lipschitz data, reduced setting]\label{cor:sec3-rate}
Assume Assumption~\ref{ass:sec3-main}\,\ref{ass:sec3-B1}--\ref{ass:sec3-B2} and that \(\omega\) is Lipschitz on \(\I\) with constant \(L_\omega\) and \(\theta_0\) is Lipschitz on \(\I\) with constant \(L_{\theta_0}\). Define the equal-cell step approximations \((\omega^N,\theta_0^N)\) by cell-evaluation at any chosen point of each cell. Then, with \(C_T^{\mathrm{red}}\) from Theorem~\ref{thm:sec3-conv},
\[
\sup_{t\in[0,T]}\|\theta^N(t)-\theta(t)\|_{\Linfty(\I)}
\le \frac{C_T^{\mathrm{red}}\,(L_\omega+L_{\theta_0})}{N}.
\]
\end{corollary}

\begin{proof}
On any cell of length \(1/N\), the Lipschitz hypothesis gives the cell-evaluation bound \(\|\omega^N-\omega\|_{\Linfty(\I)}\le L_\omega/N\) and \(\|\theta_0^N-\theta_0\|_{\Linfty(\I)}\le L_{\theta_0}/N\). Substituting into Theorem~\ref{thm:sec3-conv} yields the displayed estimate.
\end{proof}

Section~\ref{sec:reduced-model} thus establishes the truncated reduce-first continuum leg needed
for the later comparison with the continuum-first route. The truncated continuum equation
\eqref{eq:sec3-continuum-red} is well posed, finite-\(N\) trajectories of the truncated reduced
model with data satisfying Assumption~\ref{ass:sec3-main}\,\ref{ass:sec3-B3} converge to it, and the operators
\(K\), \(P\), and \(T\) retain a clear interpretation as the leading-order, first-order pairwise,
and first-order triplet contributions inherited from the finite-dimensional Fenichel expansion.
The task in Section~\ref{sec:identification} is then to recover exactly these same continuum
operators by reducing the unreduced continuum fast--slow system directly.

\begin{remark}[Truncation of the Fenichel expansion]\label{rem:sec3-truncation}
 {Remark~\ref{rem:sec3-order} shows that, for each fixed \(N\), the
finite-dimensional reduced vector field associated with the chosen slow-manifold representative
differs from its first-order truncation by \(O(\varepsilon^2)\). In the
present section, however, we pass to the continuum limit only for the truncated model
\eqref{eq:sec3-discrete-red}. Thus the reduce-first route is identified here at the level of the
first-order truncation itself. What is left open at this stage is whether the \(O(\varepsilon^2)\)
remainder from the finite-dimensional reduction can be controlled by a constant that is uniform in
\(N\), so that it also survives the dense-graph limit.}
\end{remark}

\section[Infinite-Dimensional Reduction of the Continuum Fast--Slow Model]{Infinite-Dimensional Reduction of the\\ Continuum Fast--Slow Model}\label{sec:infdim-reduction}

This section focuses on the continuum limit. Starting from the continuum fast-slow equation from Section~\ref{sec:unreduced-model}, the aim is to construct a slow manifold in the Banach-space setting. The proof strategy follows the Lyapunov-Perron approach developed for infinite-dimensional evolution equations and fast-reaction systems, with the present model fitting a particularly favourable special case because the fast linearisation is the constant operator $-I$ on $\Linfty([0,1]^2)$ and there are no spatial differential operators involved.

\subsection{Abstract fast-slow formulation}

Following Section~\ref{sec:unreduced-model}, define the Banach spaces
\begin{equation}
X_\theta := \Linfty(\I),\qquad X_a := \Linfty(\I^2),
\end{equation}
with norm
\begin{equation}
\norm{(a,\theta)}_{X_a\times X_\theta} := \norm{a}_{X_a} +\norm{\theta}_{X_\theta}.
\end{equation}
Compared with Section~\ref{sec:unreduced-model}, we order the product space here as \(X_a\times X_\theta\), placing the fast component first because the Lyapunov--Perron construction treats the slow manifold as a graph \(a=h(\theta)\) and then works in rectified coordinates \((\xi,\theta)\).

Then the continuum model can be written as an abstract Cauchy problem over these Banach spaces:
\begin{align}\label{eq:abstract-fastreaction}
\partial_t a& =\frac{1}{\eps} \mathcal{G}(\theta,a),
\qquad a(0)=a_0\\ \label{eq:abstract-slowterm}
\partial_t \theta &= \mathcal{F}(\theta,a),\qquad \theta(0)=\theta_0
\end{align}
where, as in Section~\ref{sec:unreduced-model},
\begin{align*}
\mathcal{G}(\theta,a)(x,y) &=-a(x,y) +H\bigl(\theta (x),\theta (y)\bigr), \\
\mathcal{F}(\theta,a)(x) &= \omega(x)+\int_0^1 a(x,y)\,\Gamma\bigl(\theta(y)-\theta(x)\bigr)\,\dd y.
\end{align*}

\begin{assumption}[Section 4 regularity hypotheses]\label{ass:section4}
Assume the following.
\begin{enumerate}[label= {(C\arabic*)}]
    \item $\Gamma\in C_b^1(\R)$ is $2\pi$-periodic, and $H\in C_b^2(\R^2)$ is $2\pi$-periodic in each argument.
    \item $\omega\in X_\theta$ and $a_0\in X_a$, $\theta_0\in X_\theta$.
    \item All estimates are carried out on bounded subsets of $X_a$ and $X_\theta$.
\end{enumerate}
\end{assumption}

\begin{remark}[Role of the strengthened regularity]
Section~\ref{sec:unreduced-model} only required $H\in C_b^1(\R^2)$ for well-posedness of the continuum limit of the adaptive network. For the slow-manifold expansion and the first-order identification of the reduced vector field, the natural threshold is $H\in C_b^2(\R^2)$, consistent with the continuum reduced equation in Section~\ref{sec:reduced-model}.
\end{remark}

\begin{remark}[Reduced and layer problems at \(\varepsilon=0\)]
In the slow-time formulation, setting \(\varepsilon=0\) does not produce an
evolution equation on all of \(X_a\times X_\theta\); it imposes the algebraic constraint
\(\mathcal G(\theta,a)=0\), namely \(a=h_0(\theta)\). The singular objects used below are therefore
the reduced slow flow on the critical manifold and the layer problem in the fast time
\(\tau=t/\varepsilon\), rather than an off-manifold \(\varepsilon=0\) slow-time Cauchy problem.
\end{remark}

Throughout Sections~\ref{sec:infdim-reduction}--\ref{sec:identification}, set
\begin{align*}
M_\omega&:=\|\omega\|_{L^\infty([0,1])},\quad
M_H:=\|H\|_{L^\infty(\R^2)},\quad
M_\Gamma:= \|\Gamma\|_{L^\infty(\R)},\\
\Mdh&:= \max_{q=1,2}\|\partial_q H\|_{L^\infty(\R^2)},\quad
\MdTwoH:= \max_{q,\ell=1,2}\|\partial_{q\ell}^2 H\|_{L^\infty(\R^2)},\quad
L_\Gamma:=\|\Gamma'\|_{L^\infty(\R)}.
\end{align*}

\subsection{Critical manifold and normal hyperbolicity}

The singular limit \(\varepsilon=0\) of \eqref{eq:abstract-fastreaction}-\eqref{eq:abstract-slowterm} yields the algebraic constraint
\(\mathcal{G}(\theta,a)=0\), i.e. \(a=h_0(\theta)\), where
\[ h_0(\theta)(x,y)=H(\theta(x),\theta(y)). \]
The reduced slow equation is therefore
\[ \partial_t\theta=\mathcal{F}(\theta,h_0(\theta)). \]

Hence the critical set is
\begin{equation}
S_0 := \{(a,\theta)\in X_a\times X_\theta :\, a = h_0(\theta)\},
\qquad
h_0(\theta)(x,y):=H\bigl(\theta(x),\theta(y)\bigr).
\end{equation}

\begin{proposition}[Critical manifold as a global graph]
Under Assumption~\ref{ass:section4}, $S_0$ is a $C^1$ graph over $X_\theta$. More precisely, the map
\begin{equation}
h_0 : X_\theta\to X_a,
\qquad
h_0(\theta)(x,y)=H\bigl(\theta(x),\theta(y)\bigr),
\end{equation}
is well defined and locally Lipschitz on bounded sets. Since \(H\in C_b^2(\R^2)\) by Assumption~\ref{ass:section4}, the map \(h_0\) is Fr\'echet differentiable with derivative
\begin{equation}\label{eq:Dh0}
\bigl(\DD h_0(\theta)\varphi\bigr)(x,y)
=
\partial_1 H\bigl(\theta(x),\theta(y)\bigr)\varphi(x)
+
\partial_2 H\bigl(\theta(x),\theta(y)\bigr)\varphi(y).
\end{equation}
\end{proposition}

\begin{proposition}[Normal hyperbolicity of the critical graph]
Along \(S_0\), the layer problem in the fast time \(\tau=t/\varepsilon\) has neutral tangent
dynamics and uniformly exponentially stable normal dynamics. More precisely, in the coordinates
\[
\xi=a-h_0(\theta),
\]
the linearised layer equation is
\[
\partial_\tau \xi=-\xi,\qquad \partial_\tau\theta=0.
\]
Thus, in these coordinates, the tangent linearised flow is the identity semigroup on
\(X_\theta\), while the normal linearised flow is the exponentially contracting semigroup
\(e^{-\tau}I_{X_a}\). In particular, \(S_0\) is normally hyperbolic, with neutral tangent dynamics
and uniform normal contraction rate \(1\).
\end{proposition}

\begin{proof}
Since \(\mathcal{G}(\theta,a)= -a+h_0(\theta)\), one has
\(D_2\mathcal G(\theta,a)=-I_{X_a}\). For the full layer linearisation at
\((h_0(\theta),\theta)\),
\[
\partial_\tau\delta a=-\delta a+\DD h_0(\theta)\delta\theta,\qquad
\partial_\tau\delta\theta=0.
\]
Writing \(\delta\xi=\delta a-\DD h_0(\theta)\delta\theta\) gives
\(\partial_\tau\delta\xi=-\delta\xi\), which is the claimed splitting.
\end{proof}

\begin{remark}[Connection with the infinite-dimensional theory]
This places the present model within the Banach-space invariant-manifold theory for
semiflows, in particular the framework developed by Bates, Lu, and Zeng
\cite{BatesLuZeng1998,BatesLuZeng2000}, and more specifically within the recent
infinite-dimensional fast--slow theory for Cauchy problems developed in
\cite{Hummel-Kuehn,kuehn2024infinite,Kuehn-Sulzbach,kuehn2026approximate}. In the present model
the critical manifold is a graph and the attraction in the fast variable is governed by a fixed
bounded operator. Moreover, there is no differential operator in the slow variable, so the main
technical burden present in those fast--slow works reduces here to explicit Lipschitz and fixed-point
estimates.
\end{remark}

\subsection{Existence and properties of the slow manifold}

The main result of this subsection is Theorem~\ref{thm:sec4-slow-manifold}, which constructs a global slow manifold $S_{\eps}$ for the continuum fast--slow problem, proves that its graph is $O(\eps)$-close to the critical manifold $S_0$, and records the regularity and attraction estimates needed later in Section~\ref{sec:identification}.

\begin{remark}[Structure of the construction]
The proof first passes to a rectified coordinate system \(\xi=a-h_0(\theta)\) and introduces a cut-off in the fast variable. The cut-off gives global Lipschitz bounds for the auxiliary Lyapunov--Perron problem, while the technical estimates in Appendix~\ref{app:lp-technical} give differentiability of the selected graph. We then show that the graph lies inside the cut-off-inactive region, transfer forward invariance to the original system, and finally check that the resulting graph is independent of the chosen cut-off level.
\end{remark}

We start with the rectifying change of variables
\[
\xi:=a-h_0(\theta).
\]
Since
\[
\partial_t a=\partial_t\xi+\DD h_0(\theta)\partial_t\theta,
\]
and \(\partial_t\theta=\mathcal{F}(\theta,h_0(\theta)+\xi)\), the original system \eqref{eq:abstract-fastreaction}--\eqref{eq:abstract-slowterm} is equivalent to
\begin{align}
\partial_t \xi&= -\eps^{-1}\xi - \DD h_0(\theta) B(\xi,\theta),\quad \xi(0)=\xi_0= a_0-h_0(\theta_0),\label{eq:sec4-transformed-fast}\\
\partial_t \theta&= B(\xi,\theta),\quad \theta(0)=\theta_0,\label{eq:sec4-transformed-slow}
\end{align}
where \(B(\xi,\theta)=\mathcal{F}(\theta,h_0(\theta)+\xi)\) denotes the slow vector field in rectified coordinates. The symbol \(B\) is used throughout Section~\ref{sec:infdim-reduction} to distinguish this rectified-coordinate slow field from the original slow vector field \(\mathcal F\).

The Lyapunov--Perron argument cannot be run directly with \(B\), because the slow field is only locally Lipschitz in \(\theta\), with a Lipschitz constant depending on the fast variable.

To remove this dependence, choose an odd function \(\chi\in C_b^\infty(\R)\) such that
\[
\chi(z)=z\quad\text{for }|z|\le 1,\qquad |\chi(z)|\le 2,\qquad |\chi'(z)|\le 1, \quad \text{for all } z\in \R,
\]
i.e., \(\chi\) is constant for \(|z|\) sufficiently large. For each \(r>0\) define
\[
\sigma_r:\R\to\R,\qquad \sigma_r(z):=r\,\chi(z/r).
\]
Then
\[
\sigma_r(z)=z\quad\text{for }|z|\le r,\qquad |\sigma_r(z)|\le 2r,\qquad |\sigma_r'(z)|\le 1,\quad \text{for all } z\in \R
\]
and \(\sigma_r\in C_b^\infty(\R)\).

The associated Nemytskii operator
\[
\mathsf{N}_r:X_a\to X_a,\qquad (\mathsf{N}_r a)(x,y):=\sigma_r(a(x,y))
\]
satisfies
\[
\|\mathsf{N}_r a\|_{X_a}\le 2r,\qquad \|\mathsf{N}_r a_1-\mathsf{N}_r a_2\|_{X_a}\le \|a_1-a_2\|_{X_a}
\]
for all \(a,a_1,a_2\in X_a\). By the classical superposition-operator theorem on
\(L^\infty\), \(\sigma_r\in C_b^\infty(\R)\) induces a \(C^\infty\) Nemytskii map
\(\mathsf{N}_r:X_a\to X_a\), with derivative
\[
\bigl(\DD\mathsf{N}_r(a)b\bigr)(x,y)=\sigma_r'(a(x,y))\,b(x,y),
\qquad
\|\DD\mathsf{N}_r(a)\|_{\mathcal{B}(X_a,X_a)}\le 1.
\]

The finite-window \(L^\infty\) superposition facts used below are collected in Lemma~\ref{lem:sec4-nemytskii-calculus} in Appendix~\ref{app:lp-technical}.

Fix from now on a cut-off level \(r>M_H\), and define the truncated slow vector field
\[
\mathcal{F}_r(\theta,a)(x):=\omega(x)+\int_0^1 (\mathsf{N}_r a)(x,y)\,\Gamma\bigl(\theta(y)-\theta(x)\bigr)\,\dd y.
\]
After the change of variables \(\xi=a-h_0(\theta)\), set
\[
B_r(\xi,\theta):=\mathcal{F}_r(\theta,h_0(\theta)+\xi).
\]
This is the truncated counterpart of \(B(\xi,\theta)\); the \(B\)-notation is kept for the same reason as above, to keep \(B_r\) clearly separate from the original slow field \(\mathcal F\).
The auxiliary transformed system used in the fixed-point construction is therefore
\begin{align}
\partial_t \xi&= -\eps^{-1}\xi - \DD h_0(\theta) B_r(\xi,\theta),\quad \xi(0)=\xi_0= a_0-h_0(\theta_0),\label{eq:sec4-transformed-fast-r}\\
\partial_t \theta&= B_r(\xi,\theta),\quad \theta(0)=\theta_0.\label{eq:sec4-transformed-slow-r}
\end{align}

The Fr\'echet derivative of \(h_0\) is
\[
\bigl(\DD h_0(\theta)\phi\bigr)(x,y)= \partial_1 H(\theta(x),\theta(y))\phi(x)+ \partial_2 H(\theta(x),\theta(y))\phi(y).
\]
Then
\begin{align*}
\|h_0(\theta_1)-h_0(\theta_2)\|_{X_a}&\le 2\Mdh\|\theta_1-\theta_2\|_{X_\theta},\\
\| \DD h_0(\theta) \|_{\mathcal{B}(X_\theta,X_a)}&\le 2 \Mdh,\\
\| \DD h_0(\theta_1)- \DD h_0(\theta_2) \|_{\mathcal{B}(X_\theta,X_a)}&\le 4 \MdTwoH\|\theta_1-\theta_2\|_{X_\theta}.
\end{align*}
Moreover, for all \((\theta,a),(\theta_1,a_1),(\theta_2,a_2)\in X_\theta\times X_a\) and \((u,v),(\tilde u,\tilde v)\in X_a\times X_\theta\),
\begin{align*}
\|\mathcal{F}_r(\theta,a)\|_{X_\theta}&\le C_2(r):=M_\omega +2rM_\Gamma,\\
\|\mathcal{F}_r(\theta_1,a_1)-\mathcal{F}_r(\theta_2,a_2)\|_{X_\theta}&\le M_\Gamma \|a_1-a_2\|_{X_a} + 4rL_\Gamma\|\theta_1-\theta_2\|_{X_\theta},\\
\|B_r(u,v)\|_{X_\theta}&\le C_2(r),\\
\|B_r(u,v)-B_r(\tilde u,\tilde v)\|_{X_\theta}&\le M_\Gamma \|u-\tilde u\|_{X_a} + \bigl(2M_\Gamma \Mdh +4rL_\Gamma\bigr)\|v-\tilde v\|_{X_\theta},\\
\| \DD h_0(v) B_r(u,v)-\DD h_0(\tilde v)B_r(\tilde u,\tilde v)\|_{X_a} &\le C_3(r)\bigl(\|u-\tilde u\|_{X_a}+\|v-\tilde v\|_{X_\theta}\bigr),
\end{align*}
where
\begin{equation}\label{eq:sec4-Cconstants}
\begin{aligned}
C_1(r)&:=2\Mdh\,C_2(r),
\\
C_3(r)&:=2\Mdh M_\Gamma + 4\MdTwoH\,C_2(r) + 2\Mdh\bigl(2M_\Gamma\Mdh+4rL_\Gamma\bigr),
\\
C_4(r)&:=M_\Gamma+2M_\Gamma\Mdh+4rL_\Gamma.
\end{aligned}
\end{equation}
 The key point is that the four constants \(C_1(r)\), \(C_2(r)\), \(C_3(r)\), and
\(C_4(r)\) depend only on the cut-off level \(r\) and on the model coefficients, not on any
trajectory radius in the fast variable.

For the differentiability argument below, set
\[
A_r(\xi,\theta):=\DD h_0(\theta)B_r(\xi,\theta).
\]
The regularity bounds for \(A_r\) and \(B_r\) are collected in Lemma~\ref{lem:sec4-Ar-Br-C1} in Appendix~\ref{app:lp-technical}.

For \(\eta<0\), define the weighted backward-trajectory spaces
\begin{align*}
C_{\eta,X_a}&:=\Bigl\{\xi\in C((-\infty,0];X_a): \|\xi\|_{C_{\eta,X_a}}:=\sup_{t\le 0} e^{-\eta t}\|\xi(t)\|_{X_a}<\infty\Bigr\},\\
C_{\eta,X_\theta}&:=\Bigl\{\theta\in C((-\infty,0];X_\theta): \|\theta\|_{C_{\eta,X_\theta}}:=\sup_{t\le 0} e^{-\eta t}\|\theta(t)\|_{X_\theta}<\infty\Bigr\},
\end{align*}
We write
\begin{equation}\label{eq:sec4-weighted-product}
\begin{aligned}
\mathcal X_\eta&:=C_{\eta,X_a}\times C_{\eta,X_\theta},\\
\|(\xi,\theta)\|_{\mathcal X_\eta}
&:=\|\xi\|_{C_{\eta,X_a}}+\|\theta\|_{C_{\eta,X_\theta}}.
\end{aligned}
\end{equation}
These are the weighted backward-trajectory spaces used in the Lyapunov--Perron fixed-point argument below. Since \(\eta<0\), membership in
\(\mathcal X_\eta\) allows at most exponential growth as \(t\to-\infty\), which is compatible
with the backward slow integral. The condition \(1+\varepsilon\eta>0\), imposed below, ensures
that the fast improper integral from \(-\infty\) is well defined.

For given anchor \(\theta_0\in X_\theta\), define the Lyapunov--Perron operator
\begin{align}
  \mathcal{L}_{\eps,\theta_0,r}:{\mathcal X_\eta}&\to {\mathcal X_\eta}
\intertext{by}
\mathcal{L}_{\eps,\theta_0,r}^{(1)}(\xi,\theta)(t)
&:=-\int_{-\infty}^t e^{-\eps^{-1}(t-s)} \,\DD h_0(\theta(s))B_r(\xi(s),\theta(s))\,\dd s, \\
\mathcal{L}_{\eps,\theta_0,r}^{(2)}(\xi,\theta)(t)
&:= \theta_0 + \int_0^t B_r\bigl(\xi(s),\theta(s)\bigr)\,\dd s.
\end{align}

\begin{remark}[Interpretation of the Lyapunov-Perron operator]
The first component is the explicit mild formula for the fast relaxation equation, integrated backward from $-\infty$. 
As $t-s>0$ the exponential term is rapidly decaying. 
The second component expresses the slow variable as a backward integral anchored at time $0$. 
This is the standard Lyapunov-Perron mechanism specialised to the present model.
\end{remark}

\begin{lemma}[Fixed point of the Lyapunov-Perron operator]\label{lem:lyapunov_perron}
Fix the cut-off level \(r>M_H\). Then there exist \(\eta_r<0\) and \(\eps_{0,r}>0\) such that for every anchor \(\theta_0\in X_\theta\) and every \(\eps\in(0,\eps_{0,r}]\), the map \(\mathcal L_{\eps,\theta_0,r}\) has a unique fixed point in the weighted product space \(\mathcal X_{\eta_r}\) defined in \eqref{eq:sec4-weighted-product}. 
\end{lemma}

\begin{proof}
Fix \(\eta<0\), \(\theta_0\in X_\theta\), and \((\xi,\theta),(\tilde \xi,\tilde \theta)\in {\mathcal X_\eta}\). Using the estimates above, together with
\[
e^{-\eta t}\int_{-\infty}^t e^{-\eps^{-1}(t-s)}e^{\eta s}\,\dd s
=\int_0^\infty e^{-(\eps^{-1}+\eta)\tau}\,\dd\tau
=\frac{\eps}{1+\eps\eta},
\qquad
e^{-\eta t}\int_t^0 e^{\eta s}\,\dd s\le \frac1{|\eta|},
\]
we obtain
\begin{align*}
\|\mathcal L^{(1)}_{\eps,\theta_0,r}(\xi,\theta)\|_{C_{\eta,X_a}}
&\le C_1(r)\frac{\eps}{1+\eps\eta},\\
\|\mathcal L^{(2)}_{\eps,\theta_0,r}(\xi,\theta)\|_{C_{\eta,X_\theta}}
&\le \|\theta_0\|_{C_{\eta,X_\theta}}+\frac{C_2(r)}{|\eta|}=\|\theta_0\|_{X_\theta}+\frac{C_2(r)}{|\eta|}.
\end{align*}
Taking the difference between two solutions yields
\begin{align*}
\|\mathcal L^{(1)}_{\eps,\theta_0,r}(\xi,\theta)-\mathcal L^{(1)}_{\eps,\theta_0,r}(\tilde \xi,\tilde \theta)\|_{C_{\eta,X_a}}
&\le C_3(r)\frac{\eps}{1+\eps\eta}\,{\|(\xi,\theta)-(\tilde \xi,\tilde \theta)\|_{\mathcal X_\eta}},\\
\|\mathcal L^{(2)}_{\eps,\theta_0,r}(\xi,\theta)-\mathcal L^{(2)}_{\eps,\theta_0,r}(\tilde \xi,\tilde \theta)\|_{C_{\eta,X_\theta}}
&\le \frac{C_4(r)}{|\eta|}\,{\|(\xi,\theta)-(\tilde \xi,\tilde \theta)\|_{\mathcal X_\eta}}.
\end{align*}

Set
\[
q_{\eps,\eta,r}
:=
C_3(r)\frac{\eps}{1+\eps\eta}
+\frac{C_4(r)}{|\eta|}.
\]

Choose \(\eta_r<0\) so that \(C_4(r)/|\eta_r|<1/2\). Then choose
\(\eps_{0,r}>0\) sufficiently small that
\[
1+\eps\eta_r>0,\qquad q_{\eps,\eta_r,r}<1,
\]
for every \(\eps\in(0,\eps_{0,r}]\). With these choices,
\(\mathcal L_{\eps,\theta_0,r}\) maps the Banach space
{\(\mathcal X_{\eta_r}\)} into itself by the bounds above, and is a
\(q_{\eps,\eta_r,r}\)-contraction there by the difference estimates. Hence the Banach
fixed-point theorem gives a unique fixed point in \(\mathcal X_{\eta_r}\).
\end{proof}

From here on we fix this admissible weight \(\eta_r\) and work throughout in the
single space \(\mathcal X_{\eta_r}\), the corresponding member of the family
\eqref{eq:sec4-weighted-product}. The generic weight \(\eta<0\) appears only in the contraction
estimates above, at which point \(\eta_r\) is selected.

\begin{remark}[Choice of the admissible \(\varepsilon\)-range]
After possibly decreasing \(\eps_{0,r}\), and keeping the same notation, we shall assume both
\begin{equation}\label{eq:sec4-cutoff-smallness}
C_1(r)\eps_{0,r}<r-M_H
\end{equation}
and
\[
1+\eps\eta_r\ge \frac12
\qquad\text{for all }0<\eps\le\eps_{0,r}.
\]
These reductions preserve the contraction condition in Lemma~\ref{lem:lyapunov_perron}.
\end{remark}

For the rest of this subsection set
\begin{equation}\label{eq:sec4-qstar}
q_*:=
\sup_{0<\eps\le\eps_{0,r}}
\left(
C_3(r)\frac{\eps}{1+\eps\eta_r}
+\frac{C_4(r)}{|\eta_r|}
\right)<1 .
\end{equation}
Further decreases of \(\eps_{0,r}\) below preserve \(q_*<1\), after replacing \(q_*\) by the
corresponding supremum on the smaller interval.

The auxiliary fixed-point estimates and differentiability statements used in the proof below are collected in Lemmas~\ref{lem:sec4-variational-lp}--\ref{lem:sec4-derivative-closeness} in Appendix~\ref{app:lp-technical}.

\begin{remark}
  The coordinate change to rectified variables is classical; in the dynamical-systems literature it appears at least as far back as \cite{neishtadt1987persistence,neishtadt1988persistence}.
\end{remark}
\begin{remark}
  Using the change of variables also removes a constraint of the earlier papers \cite{Hummel-Kuehn,Kuehn-Sulzbach}, where one assumes that the Lipschitz constant of the nonlinearity in the fast variable is less than one.
  Moreover, as we will see in the next step, the new coordinate system also simplifies the estimates and computations in establishing the Lyapunov--Perron slow-manifold properties. 
\end{remark}
Now we can state the Lyapunov--Perron slow-manifold theorem for this setting, with the
slow manifold understood as the invariant graph selected by the Lyapunov--Perron construction.
\begin{theorem}[Existence and Regularity of the Lyapunov--Perron Slow Manifold]\label{thm:sec4-slow-manifold}
Fix a cut-off level \(r>M_H\), and let \(\eta_r<0\) and \(\eps_{0,r}>0\) be given by Lemma~\ref{lem:lyapunov_perron}. Then, for each \(\eps\in(0,\eps_{0,r}]\), there exists a global Lipschitz map
\[
h_\eps:X_\theta\to X_a
\]
whose Lipschitz constant is bounded uniformly for \(\eps\in(0,\eps_{0,r}]\), and which satisfies the following properties.
\begin{enumerate}[label=\textup{(\roman*)},leftmargin=2.5em]
\item The set
\begin{equation}\label{eq:Me-def}
S_\varepsilon := \{(h_\varepsilon(\theta),\theta):\theta\in X_\theta\}\subset X_a\times X_\theta
\end{equation}
is an invariant manifold for the original system
\eqref{eq:abstract-fastreaction}--\eqref{eq:abstract-slowterm}, namely the
Lyapunov--Perron selected slow manifold used throughout the sequel.
\item The graph of the slow manifold is \(O(\eps)\)-close to the graph of the critical manifold:
\begin{equation}\label{eq:close-final}
\norm{h_\eps(\theta)-h_0(\theta)}_{X_a}\le C_r\eps
\qquad\text{for all } \theta\in X_\theta,
\end{equation}
where \(C_1(r)\) is defined in \eqref{eq:sec4-Cconstants}; in particular, one may
take \(C_r=C_1(r)\).
\item 
The graph map \(h_\varepsilon:X_\theta\to X_a\) is \(C^1\). Moreover, after possibly decreasing
\(\varepsilon_{0,r}\), there is a constant \(C_{D,r}>0\), independent of
\(\varepsilon\in(0,\varepsilon_{0,r}]\), such that
\[
\sup_{\theta\in X_\theta}
\|\DD h_\varepsilon(\theta)-\DD h_0(\theta)\|_{\mathcal B(X_\theta,X_a)}
\le C_{D,r}\varepsilon .
\]
Consequently \(S_\varepsilon\) is a \(C^1\)-manifold.

\item Let \((a,\theta)\) be a solution of
\eqref{eq:abstract-fastreaction}--\eqref{eq:abstract-slowterm} on \([0,T]\), the original
system without cut-off. Then there exists \(\mu_r>0\) such that
\begin{equation}
\label{eq:fast-attraction}
\|a(t)-h_\eps(\theta(t))\|_{X_a}
\le
e^{-\mu_r t/\eps}\,\|a_0-h_\eps(\theta_0)\|_{X_a},
\qquad t\in [0,T].
\end{equation}
Moreover, let \((\bar a,\bar\theta)\) denote the solution on \(S_{\eps}\) with
\[
\bar\theta(0)=\theta(0),
\qquad
\bar a(0)=h_\eps(\theta(0)).
\]
Then, for every \(\rho>0\) and \(T>0\), there exists a constant \(C_{r,T,\rho}>0\), independent of
\(\varepsilon\in(0,\varepsilon_{0,r}]\) and of the particular solution with
\(\|a_0\|_{X_a}\le \rho\), such that
\begin{equation}
\label{eq:theta-comparison}
\sup_{0\le t\le T}\|\theta(t)-\bar\theta(t)\|_{X_\theta}
\le
C_{r,T,\rho}\,\eps\,\|a_0-h_\eps(\theta_0)\|_{X_a},
\end{equation}
and
\begin{equation}
\label{eq:a-comparison}
\|a(t)-\bar a(t)\|_{X_a}
\le
e^{-\mu_r t/\eps}\,\|a_0-h_\eps(\theta_0)\|_{X_a}
+
C_{r,T,\rho}\,\eps\,\|a_0-h_\eps(\theta_0)\|_{X_a},
\qquad t\in [0,T].
\end{equation}
\end{enumerate}
\end{theorem}
The dynamics on the slow manifold are then given by
\begin{align}\label{eq:slow-dynamics}
\partial_t \theta&= \mathcal{F}\bigl(\theta,h_\eps(\theta)\bigr),\quad \theta(0)=\theta_0\in X_\theta.
\end{align}

\begin{proof}
Fix the cut-off level \(r>M_H\), and let \(\eta_r<0\), \(\eps_{0,r}>0\) be given by Lemma~\ref{lem:lyapunov_perron}. 
For \(\theta_0\in X_\theta\) and \(\eps\in(0,\eps_{0,r}]\), denote by
\[
(\xi_{\eps,r}^{\theta_0},\theta_{\eps,r}^{\theta_0})
\]
the unique fixed point of \(\mathcal L_{\eps,\theta_0,r}\) in
{\(\mathcal X_{\eta_r}\)}. 

\emph{Property (i):}
Define
\begin{equation}\label{eq:he-def}
h_{\eps,r}(\theta_0):=h_0(\theta_0)+\xi_{\eps,r}^{\theta_0}(0),
\end{equation}
and
\[
a_{\eps,r}^{\theta_0}(t):=h_0(\theta_{\eps,r}^{\theta_0}(t))+\xi_{\eps,r}^{\theta_0}(t),\qquad t\le 0.
\]
By construction, \((a^{\theta_0}_{\varepsilon,r},\theta^{\theta_0}_{\varepsilon,r})\) is a mild solution
of the truncated system on \((-\infty,0]\), with
\[ \theta^{\theta_0}_{\varepsilon,r}(0)=\theta_0,\qquad a^{\theta_0}_{\varepsilon,r}(0)=h_{\varepsilon,r}(\theta_0). \]
This yields the graph representation.

Before claiming invariance for the original system, we first show that the cut-off is
inactive on the Lyapunov--Perron orbit. Because
\((\xi_{\eps,r}^{\theta_0},\theta_{\eps,r}^{\theta_0})\) is a fixed point,
\begin{align*}
\norm{\xi_{\eps,r}^{\theta_0}(t)}_{X_a}
&\leq \int_{-\infty}^t e^{-\eps^{-1}(t-s)}
\|\DD h_0(\theta_{\eps,r}^{\theta_0}(s))B_r(\xi_{\eps,r}^{\theta_0}(s),
\theta_{\eps,r}^{\theta_0}(s))\|_{X_a}\,\dd s\\
&\leq C_1(r)\int_{-\infty}^t e^{-\eps^{-1}(t-s)}\,\dd s
=\eps C_1(r)
\end{align*}
for all \(t\le 0\). Hence, at \(t=0\),
\begin{equation}\label{eq:sec4-initial-graph-close}
\norm{h_{\eps,r}(\theta_0)-h_0(\theta_0)}_{X_a}
= \norm{\xi_{\eps,r}^{\theta_0}(0)}_{X_a}\le \eps C_1(r),
\end{equation}
and for every \(t\le 0\),
\[
\|a_{\eps,r}^{\theta_0}(t)\|_{X_a}
\le \|h_0(\theta_{\eps,r}^{\theta_0}(t))\|_{X_a}
+\|\xi_{\eps,r}^{\theta_0}(t)\|_{X_a}
\le M_H+\eps C_1(r)<r
\]
by \eqref{eq:sec4-cutoff-smallness}. Therefore
\(\mathsf{N}_r(a_{\eps,r}^{\theta_0}(t))=a_{\eps,r}^{\theta_0}(t)\) for all \(t\le 0\), so the
truncated orbit is also an orbit of the original transformed system
\eqref{eq:sec4-transformed-fast}--\eqref{eq:sec4-transformed-slow}. We set
\[
h_\eps:=h_{\eps,r},\qquad
\theta_\eps^{\theta_0}:=\theta_{\eps,r}^{\theta_0},\qquad
a_\eps^{\theta_0}:=a_{\eps,r}^{\theta_0}.
\]

Next, we show the Lipschitz continuity of the map $h_\eps$.
Let \(\theta_0,\widetilde\theta_0\in X_\theta\), and write
\[
(\xi,\theta):=(\xi_{\eps,r}^{\theta_0},\theta_{\eps,r}^{\theta_0}),
\qquad
(\widetilde \xi,\widetilde \theta):=(\xi_{\eps,r}^{\widetilde \theta_0},\theta_{\eps,r}^{\widetilde \theta_0}).
\]
Using the fixed-point identities from Lemma \ref{lem:lyapunov_perron} we obtain
\[
{\norm{(\xi,\theta)-(\widetilde \xi,\widetilde \theta)}_{\mathcal X_{\eta_r}}}
\le q\,{\norm{(\xi,\theta)-(\widetilde \xi,\widetilde \theta)}_{\mathcal X_{\eta_r}}} + \norm{\theta_0-\widetilde \theta_0}_{X_\theta},
\]
where \(q:=q_{\eps,\eta_r,r}<1\). Hence
\begin{equation}\label{eq:fp-lip}
{\norm{(\xi,\theta)-(\widetilde \xi,\widetilde \theta)}_{\mathcal X_{\eta_r}}}
\le \frac{1}{1-q}\norm{\theta_0-\widetilde \theta_0}_{X_\theta}.
\end{equation}
By \eqref{eq:sec4-qstar}, \(q\le q_*<1\), so \(1/(1-q)\le 1/(1-q_*)\)
uniformly in \(\eps\). Using the Lipschitz estimate for \(h_0\), we conclude that
\begin{align*}
\norm{h_{\eps,r}(\theta_0)-h_{\eps,r}(\widetilde \theta_0)}_{X_a}
&\le \norm{h_0(\theta_0)-h_0(\widetilde \theta_0)}_{X_a}+\norm{\xi_{\eps,r}^{\theta_0}(0)-\xi_{\eps,r}^{\widetilde \theta_0}(0)}_{X_a}\\
&\le \Bigl(2\Mdh+\frac{1}{1-q_*}\Bigr)\norm{\theta_0-\widetilde \theta_0}_{X_\theta}.
\end{align*}
Thus \(h_{\eps,r}\) is globally Lipschitz, with a constant uniform for \(\eps\in(0,\eps_{0,r}]\).

Concerning invariance, let \(\tau\le 0\) and define
\[
\xi_\tau(t):=\xi_{\eps,r}^{\theta_0}(t+\tau),
\qquad
\theta_\tau(t):=\theta_{\eps,r}^{\theta_0}(t+\tau).
\]
The shifted pair belongs to \(\mathcal X_{\eta_r}\). Indeed, if
\(u\) denotes either \(\xi_{\eps,r}^{\theta_0}\) or
\(\theta_{\eps,r}^{\theta_0}\), with target space \(Y=X_a\) or \(Y=X_\theta\), then with
\(s=t+\tau\)
\[
e^{-\eta_r t}\|u(t+\tau)\|_Y
=
e^{\eta_r\tau}e^{-\eta_r s}\|u(s)\|_Y
\le e^{\eta_r\tau}\|u\|_{C_{\eta_r,Y}}<\infty.
\]
A direct change of variables in the definition of
\(\mathcal L_{\eps,\theta_0,r}\) then shows that \((\xi_\tau,\theta_\tau)\) is a fixed point of
\(\mathcal L_{\eps,\theta_\eps^{\theta_0}(\tau),r}\); the weighted bound above, together with
\(1+\varepsilon\eta_r>0\), is the condition ensuring convergence of the improper fast integral at
\(-\infty\). By Lemma~\ref{lem:lyapunov_perron}, the fixed point in
\(\mathcal X_{\eta_r}\) is unique. Hence
\[
h_\eps(\theta_\eps^{\theta_0}(\tau))
= h_0(\theta_\eps^{\theta_0}(\tau))+\xi_\tau(0)
= h_0(\theta_\eps^{\theta_0}(\tau))+\xi_{\eps,r}^{\theta_0}(\tau)
= a_\eps^{\theta_0}(\tau).
\]
Thus every Lyapunov--Perron backward orbit lies on the graph. Forward invariance for the truncated system is the content of Lemma~\ref{lem:sec4-forward-invariance} below: applied with anchor \(\theta_\ast\), it states that the unique forward solution \((\widetilde a,\widetilde\theta)\) of the truncated system on \([0,T]\) starting at \((h_\varepsilon(\theta_\ast),\theta_\ast)\) satisfies \(\widetilde a(\tau)=h_\varepsilon(\widetilde\theta(\tau))\) for every \(\tau\in[0,T]\). The proof, recorded separately below, is the standard six-step concatenation argument: concatenate the backward Lyapunov--Perron orbit anchored at \(\theta_\ast\) with the forward solution into a single curve, observe that translating in time produces another fixed point of the Lyapunov--Perron operator, and apply uniqueness of the fixed point. The estimate proved
above,
\[
\|h_\varepsilon(\theta)\|_{X_a}\le M_H+C_1(r)\varepsilon<r,
\]
shows that the cut-off is inactive on the resulting orbit, so it is also a solution of the
original (uncut-off) system. Uniqueness of the original Cauchy problem then gives forward
invariance of \(S_\varepsilon\) for the original dynamics on \([0,T]\); since \(T\) was arbitrary,
this holds for all \(t\ge 0\).

\emph{Property (ii).}
The estimate already obtained at \(t=0\),
namely \eqref{eq:sec4-initial-graph-close},
is \eqref{eq:close-final} with \(C_r=C_1(r)\).

\emph{Property (iii).}

By definition,
\[
h_\eps(\theta_0)=h_0(\theta_0)+\Theta_\eps(\theta_0).
\]
Lemma~\ref{lem:sec4-evaluation-C1} shows that \(\Theta_\eps\) is \(C^1\) and that
\[
\DD h_\eps(\theta_0)\phi
=
\DD h_0(\theta_0)\phi+U_{\theta_0}^{\phi}(0).
\]
Therefore \(h_\eps\in C^1(X_\theta,X_a)\). Lemma~\ref{lem:sec4-derivative-closeness} gives
\[
\sup_{\theta_0\in X_\theta}
\|\DD h_\eps(\theta_0)-\DD h_0(\theta_0)\|_{\mathcal B(X_\theta,X_a)}
\le
\frac{2C_3(r)}{1-q_*}\,\eps .
\]
Thus property (iii) holds with
\[
C_{D,r}:=\frac{2C_3(r)}{1-q_*},
\]
and \(S_\eps\) is a \(C^1\)-manifold.

\emph{Property (iv).}
By the cut-off inactivity estimate proved in Property (i),
\(\mathcal F_r(\theta,h_\eps(\theta))
=\mathcal F(\theta,h_\eps(\theta))\); hence the slow-manifold identities used below are the
same for the truncated and original systems.
Let
\[
L:=2\Mdh+\frac{1}{1-q_*}.
\]
By the previous properties
\[
L_{h,\eps,r}:=\sup\|\DD h_\eps(\vartheta)\|_{\mathcal B(X_\theta,X_a)}\le 2\Mdh+\frac{1}{1-q_*}=:L_r<\infty,
\]
for all \(\varepsilon\in(0,\varepsilon_{0,r}]\).
After possibly shrinking \(\varepsilon_{0,r}\) once more, define
\[ \mu_r:=1-\varepsilon_{0,r}M_\Gamma L_r>0. \]
Hence, for every \(\varepsilon\in(0,\varepsilon_{0,r}]\),
\[ \varepsilon^{-1}-M_\Gamma L_{h,\varepsilon,r} \ge \varepsilon^{-1}-M_\Gamma L_r = \frac{1-\varepsilon M_\Gamma L_r}{\varepsilon} \ge \frac{\mu_r}{\varepsilon}. \]

Now let \((a,\theta)\) be a solution of the original system on \([0,T]\), and set
\[
d(t):=a(t)-h_\eps(\theta(t)).
\]
Invariance of the graph gives
\[
\DD h_\eps(\theta)\,\mathcal{F}(\theta,h_\eps(\theta))=\eps^{-1}\bigl(-h_\eps(\theta)+h_0(\theta)\bigr),
\]
so subtracting the equation for \(h_\eps(\theta(t))\) from the equation for \(a(t)\) yields
\[
\partial_t d
=
-\eps^{-1}d
-\DD h_\eps(\theta)\Bigl(\mathcal{F}(\theta,h_\eps(\theta)+d)-\mathcal{F}(\theta,h_\eps(\theta))\Bigr).
\]
Since \(\mathcal{F}\) is globally Lipschitz in its second variable with constant \(M_\Gamma\),
\[
\|\mathcal{F}(\theta,h_\eps(\theta)+d)-\mathcal{F}(\theta,h_\eps(\theta))\|_{X_\theta}\le M_\Gamma \|d\|_{X_a}.
\]
Therefore the upper Dini derivative of \(\|d(t)\|_{X_a}\) satisfies
\[
D^+\|d(t)\|_{X_a}
\le -\Bigl(\eps^{-1}-M_\Gamma L_{h,\eps,r}\Bigr)\|d(t)\|_{X_a}
\le -\frac{\mu_r}{\eps}\|d(t)\|_{X_a},
\]
and Gronwall's inequality gives \eqref{eq:fast-attraction}.

Next, let \((\bar a,\bar\theta)\) be the solution on \(S_{\eps}\) with \(\bar\theta(0)=\theta(0)\), and define
\[
\delta\theta(t):=\theta(t)-\bar\theta(t).
\]
Then
\[
\partial_t\delta\theta
=
\mathcal{F}(\theta,h_\eps(\theta)+d)-\mathcal{F}(\bar\theta,h_\eps(\bar\theta)).
\]
Let
\[ R_a:=\max\{\|a_0\|_{X_a},M_H,r\}. \]
By the a priori bound from Section 2, \(\|a(t)\|_{X_a}\le \max\{\|a_0\|_{X_a},M_H\}\), while
\(\|\bar a(t)\|_{X_a}\le r\). Hence
\[ \|\partial_t\delta\theta(t)\|_{X_\theta} \le M_\Gamma\|d(t)\|_{X_a} + (M_\Gamma L_{h,\varepsilon,r}+2R_aL_\Gamma)\|\delta\theta(t)\|_{X_\theta} . \]
Since \(R_a\le \max\{\rho,M_H,r\}\) under the standing hypothesis \(\|a_0\|_{X_a}\le \rho\), the Gronwall coefficient \(M_\Gamma L_{h,\varepsilon,r}+2R_a L_\Gamma\) is bounded by a constant that depends only on \(r\), \(\rho\), and the data of the problem. Consequently the constant \(C_{r,T,\rho}\) appearing below depends on \(r\), \(T\), and \(\rho\), but is otherwise independent of the particular initial datum \((a_0,\theta_0)\).
Integrating in time, using \(\delta\theta(0)=0\), and applying \eqref{eq:fast-attraction} and Gronwall's inequality, we obtain
\[
\sup_{0\le t\le T}\|\delta\theta(t)\|_{X_\theta}
\le C_{r,T,\rho}\int_0^T e^{-\mu_r s/\eps}\,\dd s\,\|d(0)\|_{X_a}
\le C_{r,T,\rho}\,\eps\,\|d(0)\|_{X_a},
\]
which is \eqref{eq:theta-comparison}. Finally,
\[
a(t)-\bar a(t)=d(t)+h_\eps(\theta(t))-h_\eps(\bar\theta(t)),
\]
so the Lipschitz bound for \(h_\eps\), together with \eqref{eq:fast-attraction} and \eqref{eq:theta-comparison}, yields \eqref{eq:a-comparison}.
\end{proof}

We now record separately the forward-invariance concatenation argument used in Property~(i) above. It is reusable: the same step structure transfers any Lyapunov--Perron backward orbit into a forward-invariance statement on \([0,T]\), once a forward solution of the truncated system is available with matching initial datum.

\begin{lemma}[Forward-invariance concatenation]\label{lem:sec4-forward-invariance}
Fix a cut-off level \(r>M_H\) and \(\varepsilon\in(0,\varepsilon_{0,r}]\), and let \(h_{\varepsilon,r}=h_\varepsilon\) be the Lyapunov--Perron graph from the construction of Theorem~\ref{thm:sec4-slow-manifold}. For every \(\theta_\ast\in X_\theta\) and every \(T>0\), the unique forward solution \((\widetilde a,\widetilde\theta):[0,T]\to X_a\times X_\theta\) of the transformed truncated system \eqref{eq:sec4-transformed-fast-r}--\eqref{eq:sec4-transformed-slow-r} with initial datum \((\widetilde a(0),\widetilde\theta(0))=(h_\varepsilon(\theta_\ast),\theta_\ast)\) satisfies
\[
\widetilde a(\tau)=h_\varepsilon(\widetilde\theta(\tau))\qquad\text{for every }\tau\in[0,T].
\]
\end{lemma}

\begin{proof}
We concatenate the backward Lyapunov--Perron orbit anchored at \(\theta_\ast\) with the forward solution and apply uniqueness of the Lyapunov--Perron fixed point.
\begin{enumerate}[label=\textup{(\roman*)},leftmargin=2.5em]
\item Let \((\xi^{\theta_\ast}_{\varepsilon,r},\theta^{\theta_\ast}_{\varepsilon,r})\) denote the backward Lyapunov--Perron fixed point anchored at \(\theta_\ast\), so that
\[
a^{\theta_\ast}_{\varepsilon,r}(t):=h_0(\theta^{\theta_\ast}_{\varepsilon,r}(t))+\xi^{\theta_\ast}_{\varepsilon,r}(t),\qquad t\le 0,
\]
is a backward solution of the truncated system with \((a^{\theta_\ast}_{\varepsilon,r}(0),\theta^{\theta_\ast}_{\varepsilon,r}(0))=(h_\varepsilon(\theta_\ast),\theta_\ast)\).
\item Existence and uniqueness of the forward solution \((\widetilde a,\widetilde\theta)\) on \([0,T]\) are guaranteed by the global Lipschitz bounds for  the truncated rectified fields \(B_r\) and \(A_r=\DD h_0\,B_r\) imposed by the cut-off \(\mathsf N_r\).
\item Concatenate the two pieces into the single curve
\[
\gamma\colon(-\infty,T]\to X_a\times X_\theta,\qquad
\gamma(t):=
\begin{cases}
(a^{\theta_\ast}_{\varepsilon,r}(t),\theta^{\theta_\ast}_{\varepsilon,r}(t)), & t\le 0,\\
(\widetilde a(t),\widetilde\theta(t)), & 0\le t\le T,
\end{cases}
\]
which agrees on \(t=0\) and is continuous; it is a mild solution of the truncated system on \((-\infty,T]\).
\item Define \(\widetilde\xi(t):=\widetilde a(t)-h_0(\widetilde\theta(t))\) on \([0,T]\). For fixed \(\tau\in[0,T]\), set, for \(t\le0\),
\[
(\xi_\tau(t),\theta_\tau(t)):=
\begin{cases}
(\xi^{\theta_\ast}_{\varepsilon,r}(t+\tau),\theta^{\theta_\ast}_{\varepsilon,r}(t+\tau)), & t+\tau\le 0,\\
(\widetilde\xi(t+\tau),\widetilde\theta(t+\tau)), & 0\le t+\tau\le T.
\end{cases}
\]
This is a backward mild solution of the transformed truncated system \eqref{eq:sec4-transformed-fast-r}--\eqref{eq:sec4-transformed-slow-r}, anchored at \(\widetilde\theta(\tau)\).
\item The pair \((\xi_\tau,\theta_\tau)\) lies in
\(\mathcal X_{\eta_r}\). On the part \(t+\tau\le0\), this follows from the same weighted
translation estimate used above. On the finite interval \(-\tau\le t\le0\), the concatenated
forward solution is continuous and the factor \(e^{-\eta_r t}\) is bounded, so the weighted norm is
finite.
\item Since \((\xi_\tau,\theta_\tau)\) is a backward mild solution and belongs to
\(\mathcal X_{\eta_r}\), the improper fast integral at \(-\infty\) is legitimate; equivalently,
the homogeneous boundary term there vanishes. Hence \((\xi_\tau,\theta_\tau)\) satisfies the
Lyapunov--Perron fixed-point equations with anchor \(\widetilde\theta(\tau)\). By the uniqueness
statement in Lemma~\ref{lem:lyapunov_perron}, it is the Lyapunov--Perron fixed point anchored at
\(\widetilde\theta(\tau)\). Evaluating at \(t=0\) gives
\[
\widetilde a(\tau)=h_0(\widetilde\theta(\tau))+\xi_\tau(0)=h_\varepsilon(\widetilde\theta(\tau)),
\]
so the forward orbit lies on the graph for every \(\tau\in[0,T]\).
\end{enumerate}
\end{proof}

The following lemma allows us to write the slow manifold independently of the cut-off level \(r\).
\begin{lemma}[Cut-off is inactive on the slow manifold]\label{lem:sec4-cutoff-inactive}
For every \(\theta\in X_\theta\) and every \(\eps\in(0,\eps_{0,r}]\),
\[
\|h_\eps(\theta)\|_{X_a}<r.
\]
Consequently,
\[
\mathsf{N}_r(h_\eps(\theta))=h_\eps(\theta),
\qquad
\mathcal{F}_r(\theta,h_\eps(\theta))=\mathcal{F}(\theta,h_\eps(\theta)).
\]
\end{lemma}

\begin{proof}
By \eqref{eq:close-final},
\[
\|h_\eps(\theta)\|_{X_a}
\le \|h_0(\theta)\|_{X_a}+\|h_\eps(\theta)-h_0(\theta)\|_{X_a}
\le M_H+C_1(r)\eps
\le M_H+C_1(r)\eps_{0,r}
<r
\]
by \eqref{eq:sec4-cutoff-smallness}. Since the \(L^\infty\)-norm is the essential supremum, this implies \(|h_\eps(\theta)(x,y)|<r\) for almost every \((x,y)\in\I^2\), hence \(\mathsf{N}_r(h_\eps(\theta))=h_\eps(\theta)\). The identity for \(\mathcal{F}_r\) follows immediately from its definition.
\end{proof}
\begin{remark}[Independence of the slow manifold from the cut-off level]
The graph $h_\varepsilon$ constructed above is independent of the choice of cut-off level
$r > M_H$, in the following precise sense.
Suppose $r < r'$ are two cut-off levels, both strictly greater than $M_H$, and let
$\varepsilon_0 \leq \min\{\varepsilon_{0,r}, \varepsilon_{0,r'}\}$ be a common admissibility threshold.
For every $\varepsilon \in (0, \varepsilon_0]$, Lemma~\ref{lem:sec4-cutoff-inactive} gives
\[
  \|h_{\varepsilon,r}(\theta)\|_{X_a} \leq M_H + C_1(r)\varepsilon < r < r',
\]
so the cut-off \(\mathsf{N}_{r'}\) is also inactive on the graph of $h_{\varepsilon,r}$.
Hence $h_{\varepsilon,r}$ is simultaneously an invariant graph for the $r'$-truncated system.
To deduce equality with $h_{\varepsilon,r'}$, fix \(\theta_0\in X_\theta\) and let
\((\xi^{\theta_0}_{\varepsilon,r},\theta^{\theta_0}_{\varepsilon,r})\) be the backward
Lyapunov--Perron orbit at cut-off level \(r\). As in the proof of
Theorem~\ref{thm:sec4-slow-manifold},
\[
\|\xi^{\theta_0}_{\varepsilon,r}(t)\|_{X_a}\le \varepsilon\, C_1(r),
\qquad
\|a^{\theta_0}_{\varepsilon,r}(t)\|_{X_a}\le M_H+\varepsilon C_1(r)<r<r'
\qquad (t\le 0),
\]
and the slow component grows at most linearly backward because
\(\|B_r(\xi,\theta)\|_{X_\theta}\le C_2(r)\). Hence
\[
\|\theta^{\theta_0}_{\varepsilon,r}(t)\|_{X_\theta}
\le \|\theta_0\|_{X_\theta}+C_2(r)|t|,
\qquad t\le0.
\]
Since \(\eta_{r'}<0\), both the bounded fast-rectified component and the at-most-linearly growing
slow component belong to
{\(\mathcal X_{\eta_{r'}}\)}. The same orbit stays strictly inside the
\(r'\)-cut-off-inactive region, so it solves the \(r'\)-truncated equations and therefore
satisfies the \(r'\)-Lyapunov--Perron fixed-point equations in the
{weighted product space \(\mathcal X_{\eta_{r'}}\)}. By uniqueness of the fixed point of
\(\mathcal{L}_{\varepsilon,\theta_0,r'}\) established in Lemma~\ref{lem:lyapunov_perron},
evaluating at \(t=0\) gives
\(h_{\varepsilon,r}(\theta_0)=h_{\varepsilon,r'}(\theta_0)\). Since \(\theta_0\) was arbitrary,
\(h_{\varepsilon,r}=h_{\varepsilon,r'}\) on \(X_\theta\).
In particular, the dynamics on the slow manifold
\[
  \partial_t\theta = \mathcal{F}(\theta,h_\varepsilon(\theta))
\]
are independent of $r$, and the notation $h_\varepsilon := h_{\varepsilon,r}$ introduced in
the proof of Theorem~\ref{thm:sec4-slow-manifold} is unambiguous.
The admissible range of $\varepsilon$ does depend on $r$ through the constants
$\eta_r$ and $C_1(r)$, but the manifold graph itself does not.
\end{remark}

\section{Identification of the Reduced Dynamics for Positive Epsilon}\label{sec:identification}

This section uses the manifold produced in Section~\ref{sec:infdim-reduction} to derive the reduced equation for the slow variable and to compare it with the continuum reduced model from Section~\ref{sec:reduced-model}. In particular, the main task is to show that the reduced flow on the slow manifold \(S_\varepsilon\) has the form
\begin{equation}\label{eq:sec5-goal}
\partial_t\theta = \omega + K[\theta] + \eps P[\theta;\omega] + \eps T[\theta] + R_\varepsilon^{\mathrm{CF}}[\theta],
\qquad
\|R_\varepsilon^{\mathrm{CF}}[\theta]\|_{X_\theta}\le C\varepsilon^2,
\end{equation}
in the same notation as Section~\ref{sec:reduced-model}, where $K$, $P$, and $T$ are the operators defined in \eqref{eq:sec3-Kop}--\eqref{eq:sec3-Top}.  {Starting from the transformed system \eqref{eq:sec4-transformed-fast}--\eqref{eq:sec4-transformed-slow}, we show that the continuum-first route produces the same first-order terms as those obtained by the reduce-first route.} Together with the uniform-in-\(N\) transfer in Corollary~\ref{cor:sec5-uniform-order}, this identification establishes first-order vector-field compatibility for the diagram from Section~\ref{sec:introduction}, with controlled \(O(\varepsilon^2)\) remainders.

We continue to use the Banach spaces \(X_\theta=\Linfty(\I)\) and \(X_a=\Linfty(\I^2)\), the maps \(\mathcal{F}\), \(\mathcal{G}\), \(h_0\), and the constants from Section~\ref{sec:infdim-reduction}, and Assumption~\ref{ass:section4} is in force throughout.
When the frequency profile is varied, we write \(\mathcal F_\Omega\), \(h_{\varepsilon,\Omega}\),
and \(\mathcal V_{\varepsilon,\Omega}^{\mathrm{CF}}\) for the corresponding slow vector field,
slow-manifold graph, and continuum-first reduced vector field. Constants below are uniform for
\(\Omega\) in a fixed \(L^\infty\)-bounded set.

\subsection{Expansion of the slow-manifold graph}

Let $h_\eps:X_\theta\to X_a$ be the slow-manifold graph constructed in Theorem \ref{thm:sec4-slow-manifold}.  {By Lemma~\ref{lem:sec4-cutoff-inactive}, the cut-off is inactive on the slow manifold, so \(\mathcal{F}_r(\theta,h_\eps(\theta))=\mathcal{F}(\theta,h_\eps(\theta))\) in \(X_\theta\). 
We therefore write \(\mathcal{F}\) rather than \(\mathcal{F}_r\) throughout this section.} On the slow manifold \(S_\varepsilon\), the fast variable satisfies $a=h_\eps(\theta)$, so any trajectory $(\bar a,\bar\theta)$ lying on the manifold obeys $\bar a(t)=h_\eps(\bar\theta(t))$. Differentiating in time and using the fast equation gives the invariance identity.

\begin{proposition}[Invariance equation]\label{prop:sec5-invariance}
The manifold graph $h_\eps$ satisfies
\begin{equation}\label{eq:sec5-invariance}
-h_\eps(\theta)+h_0(\theta)
=
\eps\,\DD h_\eps(\theta)\,\mathcal{F}\!\big(\theta,h_\eps(\theta)\big)
\end{equation}
for every \(\theta\in X_\theta\), where \(h_0(\theta)(x,y)=H\!\big(\theta(x),\theta(y)\big)\) is the critical-manifold graph.
\end{proposition}

\begin{proof}
Let $(\bar a,\bar\theta)$ be the trajectory on $S_{\eps}$ with $\bar\theta(0)=\theta$. Since $\bar a(t)=h_\eps(\bar\theta(t))$, differentiating gives
\[
\partial_t \bar a = \DD h_\eps(\bar\theta)\,\partial_t\bar\theta
= \DD h_\eps(\bar\theta)\,\mathcal{F}(\bar\theta,h_\eps(\bar\theta)).
\]
On the other hand, the fast equation reads
\[
\partial_t \bar a
= \eps^{-1}\mathcal{G}(\bar\theta,h_\eps(\bar\theta))
= \eps^{-1}\big({-h_\eps(\bar\theta)+h_0(\bar\theta)}\big).
\]
Equating at $t=0$ and multiplying by $\eps$ yields \eqref{eq:sec5-invariance}.
Since \(\theta\in X_\theta\) is arbitrary and the slow manifold is invariant, one can always choose the trajectory \((\bar a, \bar\theta)\) with \(\bar\theta(0) = \theta\), so \eqref{eq:sec5-invariance} holds for every \(\theta \in X_\theta\).
\end{proof}

The invariance equation \eqref{eq:sec5-invariance} allows us to extract the first-order correction $h_1$ by a formal expansion, exactly as in the finite-dimensional case.

\begin{proposition}[First-order correction]\label{prop:sec5-h1}
Write $h_\eps(\theta)=h_0(\theta)+\eps h_1(\theta)+r_\eps(\theta)$, where $h_1:X_\theta\to X_a$ is defined by
\begin{equation}\label{eq:sec5-h1-abstract}
h_1(\theta) := -\DD h_0(\theta)\,\mathcal{F}\big(\theta,h_0(\theta)\big).
\end{equation}
Since \(\|h_0(\theta)\|_{X_a}\le M_H<r\), the cut-off is already inactive on the critical manifold, so \eqref{eq:sec5-h1-abstract} is also the first-order term obtained from the truncated invariance equation.
Then there is \(C_r>0\) such that
\(\|r_\eps(\theta)\|_{X_a}\le C_r\eps^2\) uniformly in \(\theta\in X_\theta\).

Explicitly, for a.e.\ $(x,y)\in\I^2$,
\begin{align}\label{eq:sec5-h1-explicit}
h_1[\theta](x,y)
&=
-\partial_1 H\!\big(\theta(x),\theta(y)\big)
\bigg[\omega(x)+\int_0^1 H\!\big(\theta(x),\theta(z)\big)\,
\Gamma\!\big(\theta(z)-\theta(x)\big)\,\dd z\bigg]
\nonumber\\
&\quad
-\partial_2 H\!\big(\theta(x),\theta(y)\big)
\bigg[\omega(y)+\int_0^1 H\!\big(\theta(y),\theta(z)\big)\,
\Gamma\!\big(\theta(z)-\theta(y)\big)\,\dd z\bigg].
\end{align}
\end{proposition}

\begin{proof}
Insert $h_\eps=h_0+\eps h_1 + r_\eps$ into \eqref{eq:sec5-invariance}. The left-hand side becomes
\[
-h_0(\theta)-\eps h_1(\theta) - r_\eps(\theta)+h_0(\theta)
=
-\eps h_1(\theta) - r_\eps(\theta).
\]
By the definition \eqref{eq:sec5-h1-abstract}, the left-hand side is
\(\eps\,\DD h_0(\theta)\mathcal{F}(\theta,h_0(\theta))-r_\eps(\theta)\). Rearranging the exact invariance equation gives
\begin{align}\label{eq:sec5-r-split}
r_\eps(\theta)
&=
\eps\,\DD h_0(\theta)\,\mathcal{F}(\theta,h_0(\theta))
-\eps\,\DD h_\eps(\theta)\,\mathcal{F}(\theta,h_\eps(\theta))\nonumber\\
&=
\eps\bigl(\DD h_0(\theta)-\DD h_\eps(\theta)\bigr)
\mathcal{F}(\theta,h_\eps(\theta))
+\eps\,\DD h_0(\theta)\bigl(\mathcal{F}(\theta,h_0(\theta))
-\mathcal{F}(\theta,h_\eps(\theta))\bigr).
\end{align}
Thus the two possible sources of error are derivative-closeness and graph-closeness.

\emph{Term involving derivative closeness.} By Theorem~\ref{thm:sec4-slow-manifold}(iii),
\[
\|\DD h_\varepsilon(\theta)-\DD h_0(\theta)\|_{\mathcal B(X_\theta,X_a)}
\le C_{D,r}\varepsilon
\]
uniformly in \(\theta\). The slow vector field is uniformly bounded along the slow manifold by
\[
\|\mathcal{F}(\theta,h_\eps(\theta))\|_{X_\theta}
\le M_\omega+M_\Gamma(M_H+C_1(r)\varepsilon_{0,r});
\]
call this bound \(\overline{M}_{\mathcal F,r}\). Hence
\[
\bigl\|\bigl(\DD h_0(\theta)-\DD h_\eps(\theta)\bigr)\mathcal{F}(\theta,h_\eps(\theta))\bigr\|_{X_a}
\le C_{D,r}\,\overline{M}_{\mathcal F,r}\,\varepsilon,
\qquad\text{uniformly in }\theta\in X_\theta.
\]

\emph{Term involving graph closeness.} The Lipschitz bound for \(\mathcal F\) in its second argument and \eqref{eq:close-final} give
\[
\|\mathcal{F}(\theta,h_\eps(\theta))-\mathcal{F}(\theta,h_0(\theta))\|_{X_\theta}
\le M_\Gamma\|h_\eps(\theta)-h_0(\theta)\|_{X_a}
\le M_\Gamma C_1(r)\varepsilon,
\]
together with \(\|\DD h_0(\theta)\|_{\mathcal{B}(X_\theta,X_a)}\le 2\Mdh\), yielding
\[
\bigl\|\DD h_0(\theta)\bigl(\mathcal{F}(\theta,h_0(\theta))-\mathcal{F}(\theta,h_\eps(\theta))\bigr)\bigr\|_{X_a}
\le 2\Mdh M_\Gamma C_1(r)\,\varepsilon,
\qquad\text{uniformly in }\theta\in X_\theta.
\]

Combining the two estimates in \eqref{eq:sec5-r-split} gives
\[
\|r_\eps(\theta)\|_{X_a}
\le
\bigl(C_{D,r}\overline{M}_{\mathcal F,r}+2\Mdh M_\Gamma C_1(r)\bigr)\varepsilon^2,
\qquad\text{uniformly in }\theta\in X_\theta.
\]

This proves the claimed \(O(\varepsilon^2)\) bound.

For the explicit formula, recall from \eqref{eq:Dh0} that
\[
\big(\DD h_0(\theta)\,\varphi\big)(x,y)
=
\partial_1 H\!\big(\theta(x),\theta(y)\big)\,\varphi(x)
+\partial_2 H\!\big(\theta(x),\theta(y)\big)\,\varphi(y),
\]
and on the critical manifold
\[
\mathcal{F}(\theta,h_0(\theta))(x)
=
\omega(x)+\int_0^1 H\!\big(\theta(x),\theta(z)\big)\,\Gamma\!\big(\theta(z)-\theta(x)\big)\,\dd z.
\]
Substituting \(\varphi=\mathcal{F}(\theta,h_0(\theta))\) and applying the minus sign gives \eqref{eq:sec5-h1-explicit}.
\end{proof}

\begin{remark}[Comparison with the finite-$N$ correction]
Formula~\eqref{eq:sec5-h1-explicit} is the direct continuum analogue of the finite-dimensional first-order correction from the preceding paper~\cite{KuehnMurphy}. In the finite-$N$ setting one obtains
\[
h_{1,ij}(\theta)
=
-\partial_{\theta_i}H(\theta_i,\theta_j)\,V_i(\theta)
-\partial_{\theta_j}H(\theta_i,\theta_j)\,V_j(\theta),
\]
where
\[
V_i(\theta)=\omega_i+\frac{1}{N}\sum_{k=1}^N H(\theta_i,\theta_k)\Gamma(\theta_k-\theta_i)
\]
is the reduced phase velocity evaluated along the critical manifold.
The continuum version replaces the discrete index and finite sum by the continuous label and Lebesgue integral, respectively.
\end{remark}

\subsection[Step-field embedding and uniform-in-N control]{Step-field embedding and uniform-in-\(N\) control}
\label{sec:sec5-stepfield}

Proposition~\ref{prop:sec5-h1} bounds the continuum remainder \(r_\varepsilon\) at the level of the Banach-space graph \(h_\varepsilon\). The purpose of this subsection is to transfer that bound to the finite-\(N\) setting by defining the finite-\(N\) Lyapunov--Perron representative used here through the same restricted construction and identifying that representative with the restriction of \(h_\varepsilon\) to step fields. The resulting Corollary~\ref{cor:sec5-uniform-order} provides the uniform-in-\(N\) \(O(\varepsilon^2)\) estimate anticipated in Remark~\ref{rem:sec3-order}; this estimate is the input needed for the reduce-first vector-field comparison in Section~\ref{sec:identification}.

Fix \(N\in\mathbb N\) and assume that the natural-frequency profile in~\eqref{eq:sec2-continuum} is a step field, \(\omega=\omega^N\in X_\theta\), constant on each \(I_i^N\). Write
\[
\mathfrak{S}_N^\theta
:=
\{\theta\in X_\theta:\theta\text{ is constant on each }I_i^N\},
\qquad
\mathfrak{S}_N^a
:=
\{a\in X_a:a\text{ is constant on each }I_i^N\times I_j^N\},
\]
and \(\mathfrak{S}_N:=\mathfrak{S}_N^\theta\times \mathfrak{S}_N^a\). The step-coefficient map identifies \(\mathfrak{S}_N^\theta\) with \(\R^N\) and \(\mathfrak{S}_N^a\) with \(\R^{N\times N}\); under this identification the \(L^\infty\)-norm restricts to the \(\ell^\infty\)-norm of the coefficients.
For \(\vartheta=(\vartheta_1,\dots,\vartheta_N)\in\R^N\), write
\[
[\vartheta]:=([\vartheta_1],\dots,[\vartheta_N])\in\mathbb{T}^N
\]
for its componentwise torus class, and let \(\theta^{N,\vartheta}\in \mathfrak{S}_N^\theta\) denote the
associated step field, \(\theta^{N,\vartheta}(x)=\vartheta_i\) on \(I_i^N\).

By Lemma~\ref{lem:sec2-stepfield-preservation}, the unreduced continuum flow preserves
these step-field subspaces and restricts to the finite-\(N\) unreduced ODE. When the
frequency profile is the step field \(\omega^N\), write \(h_{\varepsilon,\omega^N}\) for the
continuum Lyapunov--Perron graph. We write \(h_{\varepsilon,\omega^N}^N\) for the graph obtained
by applying the same cut-off Lyapunov--Perron contraction to the restricted finite-dimensional
ODE on \(\R^N\times\R^{N\times N}\). This is the finite-\(N\) graph used in the uniform remainder
transfer below; under the restricted Lyapunov--Perron construction used here, it is the
representative denoted \(h_\varepsilon^N\) in the finite-dimensional notation of
\cite{KuehnMurphy}. We only record the identification of the continuum and finite-\(N\) graphs
below.

{
This is a selection statement, not a generic uniqueness statement for slow manifolds. The object
\(h_{\varepsilon,\omega^N}^N\) denotes the coefficient graph of the fixed point selected by the
same cut-off Lyapunov--Perron problem as \(h_{\varepsilon,\omega^N}\), with the cut-off level
\(r\) and weight \(\eta_r\) fixed below, and with the same time-zero anchor condition, restricted to the finite-dimensional
step-field subspace. Thus the comparison below identifies two representatives selected by one
fixed-point construction; it does not assert uniqueness among all nearby invariant slow manifolds.
}

\begin{lemma}[Identification of the continuum and finite-\(N\) Lyapunov--Perron graphs on step fields]
\label{lem:sec5-stepfield-identify}
Fix a cut-off level \(r>M_H\) and \(\varepsilon\in(0,\varepsilon_{0,r}]\). Let
\(h_{\varepsilon,\omega^N}:X_\theta\to X_a\) be the continuum graph constructed in
Theorem~\ref{thm:sec4-slow-manifold} with frequency profile \(\omega^N\), and let
\(h_{\varepsilon,\omega^N}^N:\mathbb T^N\to\R^{N\times N}\) be the
finite-\(N\) Lyapunov--Perron-selected coefficient graph obtained by restricting
the same weighted cut-off fixed-point problem to step fields. Then
\(h_{\varepsilon,\omega^N}(\theta^{N,\vartheta})\in \mathfrak{S}_N^a\) for every \(\vartheta\in\R^N\), and
the coefficients of \(h_{\varepsilon,\omega^N}(\theta^{N,\vartheta})\) are exactly
\(h_{\varepsilon,\omega^N}^N([\vartheta])\). In particular, this identification is independent of
the chosen lift \(\vartheta\).
\end{lemma}

\begin{proof}
The critical-manifold graph \(h_0\), the Nemytskii cut-off \(\mathsf N_r\), and the operators \(\mathcal F_{\omega^N}\) and \(\mathcal G\) all preserve \(\mathfrak{S}_N\), by Lemma~\ref{lem:sec2-stepfield-preservation} and pointwise inspection of their definitions. The Lyapunov--Perron operator \(\mathcal L_{\varepsilon,\theta_0,r}\) of Lemma~\ref{lem:lyapunov_perron} therefore maps the closed subspace \(C_{\eta_r}(\mathfrak{S}_N^a)\times C_{\eta_r}(\mathfrak{S}_N^\theta)\) into itself whenever \(\theta_0\in \mathfrak{S}_N^\theta\); its contractivity on the ambient weighted product space \(\mathcal X_{\eta_r}\) localises to this subspace, so the fixed point \((\xi^{\theta_0},\theta^{\theta_0})\) lies there. Evaluating at \(t=0\) and using \(h_{\varepsilon,\omega^N}(\theta_0)=h_0(\theta_0)+\xi^{\theta_0}(0)\) yields \(h_{\varepsilon,\omega^N}(\theta_0)\in \mathfrak{S}_N^a\).

If \(\widetilde\vartheta=\vartheta+2\pi m\) with \(m\in\mathbb{Z}^N\), then the associated step
fields \(\theta^{N,\vartheta}\) and \(\theta^{N,\widetilde\vartheta}\) induce the same restricted
finite-dimensional vector field because
\(\Gamma(\widetilde\vartheta_j-\widetilde\vartheta_i)=\Gamma(\vartheta_j-\vartheta_i)\) and
\(H(\widetilde\vartheta_i,\widetilde\vartheta_j)=H(\vartheta_i,\vartheta_j)\). Thus the restricted
fixed point depends only on the torus class \([\vartheta]\).

Read as a Lyapunov--Perron problem on \(\R^N\times\R^{N\times N}\) for the finite-\(N\) unreduced
ODE (Lemma~\ref{lem:sec2-stepfield-preservation}), the restricted fixed point is an exact
slow-manifold graph for that ODE, lying in the cut-off ball on which the cut-off is inactive
(Lemma~\ref{lem:sec4-cutoff-inactive}).
The identification now uses only uniqueness of this selected Lyapunov--Perron fixed point, not
generic uniqueness of slow manifolds. We do not claim that slow manifolds near \(S_0^N\) are unique
as invariant manifolds. The slow directions are parametrised by the anchor \(\theta_0\); once the
anchor, the cut-off level \(r\), the weight \(\eta_r\), and the Lyapunov--Perron boundary condition
on \((-\infty,0]\) are fixed, Lemma~\ref{lem:lyapunov_perron} gives a unique fixed point in
\(\mathcal X_{\eta_r}\). Concretely, the same Lyapunov--Perron operator
\(\mathcal L_{\varepsilon,\theta_0,r}\) used in Lemma~\ref{lem:lyapunov_perron} restricts, when
\(\theta_0\in \mathfrak{S}_N^\theta\), to a contraction on the closed subspace
\[
C_{\eta_r}(\mathfrak{S}_N^a)\times C_{\eta_r}(\mathfrak{S}_N^\theta)\simeq C_{\eta_r}(\R^{N\times N})\times C_{\eta_r}(\R^N),
\]
and produces a unique fixed point there. Since the closed step-field subspace is invariant under
the same operator, this ambient fixed point with step-field anchor coincides with the fixed point
of the restricted finite-dimensional contraction. By definition, the coefficient graph of
this unique fixed point is \(h_{\varepsilon,\omega^N}^N\). Since
\(h_{\varepsilon,\omega^N}(\theta^{N,\vartheta})\) is obtained by the same fixed point, restricted
to the same finite-dimensional subspace, its coefficients are exactly
\(h_{\varepsilon,\omega^N}^N([\vartheta])\). Passage from \(\vartheta\in\R^N\) to its torus class
\([\vartheta]\in\mathbb{T}^N\) uses the periodicity argument established just above.
\end{proof}

\begin{corollary}[Uniform \(O(\varepsilon^2)\) graph remainder]
\label{cor:sec5-uniform-order}
\leavevmode\par
For the Lyapunov--Perron-selected finite-\(N\) representative
\(h_{\varepsilon,\omega^N}^N\) identified in Lemma~\ref{lem:sec5-stepfield-identify}, after
choosing
\(\varepsilon_{0,r}\) uniformly for the step-frequency bound
\(M_\omega^{\mathrm{step}}\), there exists a constant \(C(r)>0\), depending only on
\(M_\Gamma,L_\Gamma,M_H,\Mdh,\MdTwoH,r\), and the uniform frequency bound
\(M_\omega^{\mathrm{step}}\) from Remark~\ref{rem:sec3-step-frequency-bound}, such that
\[
\bigl\|r_\varepsilon^N([\vartheta])\bigr\|_{\ell^\infty}
\le
C(r)\,\varepsilon^2
\qquad
\text{for all }\varepsilon\in(0,\varepsilon_{0,r}],\ N\in\mathbb N\text{ and }\vartheta\in\mathbb R^N.
\]
\end{corollary}

\begin{proof}
Let \(\vartheta\in\R^N\) and let \(\theta^{N,\vartheta}\in \mathfrak{S}_N^\theta\) be the associated step
field. Define the continuum remainder at the step-field anchor by
\[
r_\varepsilon(\theta^{N,\vartheta})
:=
h_{\varepsilon,\omega^N}(\theta^{N,\vartheta})-h_0(\theta^{N,\vartheta})-\varepsilon h_1(\theta^{N,\vartheta}).
\]
The algebraic formulas \(h_0(\theta)(x,y)=H(\theta(x),\theta(y))\) and~\eqref{eq:sec5-h1-explicit}
send step-field inputs to step-field outputs, and Lemma~\ref{lem:sec5-stepfield-identify} does the
same for \(h_{\varepsilon,\omega^N}\); hence \(r_\varepsilon(\theta^{N,\vartheta})\in \mathfrak{S}_N^a\), and its
step-coefficients coincide with \(r_\varepsilon^N([\vartheta])\). Since
\(\omega=\omega^N\) is a step field with
\(\|\omega^N\|_{\Linfty(\I)}\le M_\omega^{\mathrm{step}}\),
Proposition~\ref{prop:sec5-h1} therefore gives a constant independent of \(N\):
\[
\|r_\varepsilon^N([\vartheta])\|_{\ell^\infty}
=
\|r_\varepsilon(\theta^{N,\vartheta})\|_{X_a}
\le C(r)\,\varepsilon^2,
\]
with \(C(r)\) the constant tracked in the proof of Proposition~\ref{prop:sec5-h1}, depending on
the continuum data listed above and on \(M_\omega^{\mathrm{step}}\), but not on \(N\).
\end{proof}

The transfer works because the sup-norm operator bounds for the finite-\(N\) vector field coincide with their \(L^\infty\)-continuum counterparts: the \(1/N\) and \(1/N^2\) prefactors in the coupling sums exactly cancel the \(N\) and \(N^2\) summands, giving for instance \(\|\partial_a f^N\|_{\ell^\infty\to\ell^\infty}\le M_\Gamma\) and \(\|\DD h_0^N\|_{\ell^\infty\to\ell^\infty}\le 2\Mdh\), identical to the continuum estimates of Section~\ref{sec:infdim-reduction}. The Lyapunov--Perron contraction constant is consequently the same whether read in \(L^\infty(\I^2)\) or on the step-field subspace.

\subsection{Reduced vector field and first-order compatibility}

We now substitute the first-order expansion
\(h_\eps=h_0+\eps h_1+r_\varepsilon\) into the slow equation
\eqref{eq:slow-dynamics} and identify the resulting operators.

\begin{theorem}[Identification of the reduced dynamics]\label{thm:sec5-identification}
Under Assumption~\ref{ass:section4}, for every $\eps\in(0,\eps_{0,r}]$, the dynamics on $S_{\eps}$ satisfy
\begin{equation}\label{eq:sec5-reduced}
\partial_t\theta
=
\omega + K[\theta] + \eps\,P[\theta;\omega] + \eps\,T[\theta]
+R_\varepsilon^{\mathrm{CF}}[\theta],
\end{equation}
where $K$, $P$, and $T$ are the operators defined in \eqref{eq:sec3-Kop}--\eqref{eq:sec3-Top}. The remainder satisfies
\[
\sup_{\theta\in X_\theta}
\|R_\varepsilon^{\mathrm{CF}}[\theta]\|_{X_\theta}
\le C_{\mathrm{CF}}(r)\varepsilon^2 .
\]
\end{theorem}

\begin{proof}
{ By Lemma~\ref{lem:sec4-cutoff-inactive} the cut-off is inactive on the slow manifold, so \(\mathcal{F}(\theta,h_\eps(\theta))=\mathcal{F}_r(\theta,h_\eps(\theta))\).}
The reduced dynamics on the slow manifold are
\[
\partial_t\theta(x)
=
\mathcal{F}(\theta,h_\eps(\theta))(x)
=
\omega(x) + \int_0^1 h_\eps[\theta](x,y)\,\Gamma\!\big(\theta(y)-\theta(x)\big)\,\dd y.
\]
Substituting \(h_\eps=h_0+\eps h_1+r_\varepsilon\) gives
\begin{align*}
\partial_t\theta(x)
&=
\omega(x)
+\int_0^1 H\!\big(\theta(x),\theta(y)\big)\,\Gamma\!\big(\theta(y)-\theta(x)\big)\,\dd y \\
&\quad
+\eps\int_0^1 h_1[\theta](x,y)\,\Gamma\!\big(\theta(y)-\theta(x)\big)\,\dd y
+
R_\varepsilon^{\mathrm{CF}}[\theta](x),
\end{align*}
where
\[
R_\varepsilon^{\mathrm{CF}}[\theta](x):=
\int_0^1 r_\varepsilon[\theta](x,y)\,
\Gamma\!\big(\theta(y)-\theta(x)\big)\,\dd y .
\]
The first line is $\omega(x)+K[\theta](x)$.

It remains to show that $\int_0^1 h_1[\theta](x,y)\,\Gamma(\theta(y)-\theta(x))\,\dd y = P[\theta;\omega](x)+T[\theta](x)$.

Substituting the explicit formula \eqref{eq:sec5-h1-explicit} for $h_1$ yields
\begin{align*}
&\int_0^1 h_1[\theta](x,y)\,\Gamma\!\big(\theta(y)-\theta(x)\big)\,\dd y \\
&=
-\int_0^1 \Gamma\!\big(\theta(y)-\theta(x)\big)\,\partial_1 H\!\big(\theta(x),\theta(y)\big)
\bigg[\omega(x)+\int_0^1 H\!\big(\theta(x),\theta(z)\big)\,\Gamma\!\big(\theta(z)-\theta(x)\big)\,\dd z\bigg]\,\dd y\\
&\quad
-\int_0^1 \Gamma\!\big(\theta(y)-\theta(x)\big)\,\partial_2 H\!\big(\theta(x),\theta(y)\big)
\bigg[\omega(y)+\int_0^1 H\!\big(\theta(y),\theta(z)\big)\,\Gamma\!\big(\theta(z)-\theta(y)\big)\,\dd z\bigg]\,\dd y.
\end{align*}
We separate the $\omega$-dependent and the integral terms.

\emph{$\omega$-dependent terms.}
Collecting the terms involving $\omega$ gives
\begin{align*}
&-\int_0^1 \Gamma\!\big(\theta(y)-\theta(x)\big)
\Big(\partial_1 H\!\big(\theta(x),\theta(y)\big)\,\omega(x)
+\partial_2 H\!\big(\theta(x),\theta(y)\big)\,\omega(y)\Big)\,\dd y
=
P[\theta;\omega](x),
\end{align*}
by the definition \eqref{eq:sec3-Pop}.

\emph{Integral terms.}
The remaining double integrals are
\begin{align*}
&-\int_0^1 \Gamma\!\big(\theta(y)-\theta(x)\big)\,\partial_1 H\!\big(\theta(x),\theta(y)\big)
\int_0^1 H\!\big(\theta(x),\theta(z)\big)\,\Gamma\!\big(\theta(z)-\theta(x)\big)\,\dd z\,\dd y\\
&-\int_0^1 \Gamma\!\big(\theta(y)-\theta(x)\big)\,\partial_2 H\!\big(\theta(x),\theta(y)\big)
\int_0^1 H\!\big(\theta(y),\theta(z)\big)\,\Gamma\!\big(\theta(z)-\theta(y)\big)\,\dd z\,\dd y.
\end{align*}
Since $\Gamma$, $H$, $\partial_1 H$, $\partial_2 H$ are bounded and measurable, each integrand is in $L^\infty(\I^2)$. By Fubini's theorem we may write each term as a double integral over $(y,z)\in\I^2$:
\begin{align*}
&-\int_0^1\!\int_0^1
\Gamma\!\big(\theta(y)-\theta(x)\big)\,\partial_1 H\!\big(\theta(x),\theta(y)\big)\,
H\!\big(\theta(x),\theta(z)\big)\,\Gamma\!\big(\theta(z)-\theta(x)\big)\,\dd z\,\dd y\\
&-\int_0^1\!\int_0^1
\Gamma\!\big(\theta(y)-\theta(x)\big)\,\partial_2 H\!\big(\theta(x),\theta(y)\big)\,
H\!\big(\theta(y),\theta(z)\big)\,\Gamma\!\big(\theta(z)-\theta(y)\big)\,\dd z\,\dd y
=
T[\theta](x),
\end{align*}
by the definition \eqref{eq:sec3-Top}. Combining gives \eqref{eq:sec5-reduced}.

The remainder in \eqref{eq:sec5-reduced} comes entirely from the manifold remainder
\(r_\varepsilon\). Since
\(\|r_\varepsilon(\theta)\|_{X_a}\le C_r\varepsilon^2\) uniformly in \(X_\theta\) and
\(\Gamma\) is bounded,
\[
\|R_\varepsilon^{\mathrm{CF}}[\theta]\|_{X_\theta}
\le M_\Gamma C_r\varepsilon^2,
\]
which gives the stated bound.
\end{proof}

We can now state the main structural conclusion of the paper: along admissible step-field
approximations, the two routes have the same first-order continuum vector-field truncation, with
controlled \(O(\varepsilon^2)\) remainders.

\begin{theorem}[First-order compatibility along admissible step approximations]\label{thm:sec5-commutation}
Assume Assumption~\ref{ass:section4}, and let \(\{\omega^N\}_{N\ge1}\) be an admissible
equal-cell step sequence with \(\omega^N\to\omega\) in \(X_\theta\) and
\(\sup_N\|\omega^N\|_{X_\theta}<\infty\). After choosing \(\varepsilon_{0,r}>0\) uniformly for
this frequency bound, the following holds for all
\(\varepsilon\in(0,\varepsilon_{0,r}]\).
\begin{enumerate}[label=\textup{(\roman*)}]
\item \textbf{Reduce-first route on step fields.}
For every \(N\in\mathbb N\) and every step field
\(\theta=\theta^{N,\vartheta}\in \mathfrak{S}_N^\theta\), let
\(\mathcal V_\varepsilon^{N,\mathrm{RF}}[\theta]\) denote the embedded finite-\(N\)
reduced phase vector field obtained from the restricted Lyapunov--Perron graph
\(h_{\varepsilon,\omega^N}^N\) above, with frequency vector
\((\omega_1^N,\ldots,\omega_N^N)\). Then
\[
\mathcal V_\varepsilon^{N,\mathrm{RF}}[\theta]
=\omega^N+K[\theta]+\varepsilon P[\theta;\omega^N]+\varepsilon T[\theta]
+R_\varepsilon^{N,\mathrm{RF}}[\theta],
\]
where, for \(x\in I_i^N\),
\[
R_\varepsilon^{N,\mathrm{RF}}[\theta](x)
:=
\frac1N\sum_{j=1}^N
r_{\varepsilon,ij}^N([\vartheta])\,
\Gamma(\vartheta_j-\vartheta_i).
\]
The graph-remainder estimate of Corollary~\ref{cor:sec5-uniform-order} gives
\[
\bigl\|R_\varepsilon^{N,\mathrm{RF}}[\theta]\bigr\|_{X_\theta}
\le M_\Gamma C(r)\varepsilon^2,
\]
with constant independent of \(N\) and \(\vartheta\).

\item \textbf{Continuum-first route.}
For any frequency profile \(\Omega\) in the same \(L^\infty\)-bounded class, define
\[
\mathcal V_{\varepsilon,\Omega}^{\mathrm{CF}}[\theta]
:=\mathcal F_\Omega(\theta,h_{\varepsilon,\Omega}(\theta)).
\]
Then
\[
\mathcal V_{\varepsilon,\Omega}^{\mathrm{CF}}[\theta]
=
\Omega+K[\theta]+\varepsilon P[\theta;\Omega]+\varepsilon T[\theta]
+R_{\varepsilon,\Omega}^{\mathrm{CF}}[\theta],
\]
with
\[
\sup_{\theta\in X_\theta}
\|R_{\varepsilon,\Omega}^{\mathrm{CF}}[\theta]\|_{X_\theta}
\le C_{\mathrm{CF}}(r)\varepsilon^2,
\]
where the constant is uniform over the chosen frequency class.

\item \textbf{Compatibility.}
For each \(N\), the two exact vector fields agree on \(\mathfrak{S}_N^\theta\) when the
continuum-first route is constructed with the same step frequency:
\begin{equation}\label{eq:sec5-exact-step-equality}
\mathcal V_{\varepsilon,\omega^N}^{\mathrm{CF}}[\theta^{N,\vartheta}]
=
\mathcal V_\varepsilon^{N,\mathrm{RF}}[\theta^{N,\vartheta}].
\end{equation}
Consequently both routes have the common first-order truncation
\[
\mathcal V_\varepsilon^{(1)}[\theta;\Omega]
:=\Omega+K[\theta]+\varepsilon P[\theta;\Omega]+\varepsilon T[\theta].
\]
If the continuum-first vector field is instead evaluated with the limiting frequency \(\omega\),
then for every \(N\) and \(\vartheta\),
\begin{equation}\label{eq:sec5-commutation}
\bigl\|
\mathcal V_\varepsilon^{N,\mathrm{RF}}[\theta^{N,\vartheta}]
-
\mathcal V_{\varepsilon,\omega}^{\mathrm{CF}}[\theta^{N,\vartheta}]
\bigr\|_{X_\theta}
\le
\bigl(1+2\varepsilon M_\Gamma\Mdh\bigr)\|\omega^N-\omega\|_{X_\theta}
+C(r)\varepsilon^2 .
\end{equation}
Thus, along admissible step-field approximations, the reduce-first and continuum-first routes
share the same first-order continuum vector-field truncation, with \(O(\varepsilon^2)\) remainders.
\end{enumerate}
\end{theorem}

\begin{proof}
For part~(i), write the finite-\(N\) graph expansion as
\[
h_\varepsilon^N=h_0^N+\varepsilon h_1^N+r_\varepsilon^N.
\]
Substituting this into the finite-\(N\) phase vector field shows that the graph remainder
contributes exactly
\[
R_{\varepsilon,i}^{N,\mathrm{RF}}([\vartheta])
=
\frac1N\sum_{j=1}^N
r_{\varepsilon,ij}^N([\vartheta])\Gamma(\vartheta_j-\vartheta_i).
\]
Therefore
\[
\|R_\varepsilon^{N,\mathrm{RF}}[\theta^{N,\vartheta}]\|_{X_\theta}
\le
M_\Gamma\|r_\varepsilon^N([\vartheta])\|_{\ell^\infty}
\le
M_\Gamma C(r)\varepsilon^2
\]
by Corollary~\ref{cor:sec5-uniform-order}. The remaining zeroth- and first-order terms are the
embedded forms of \(K\), \(P[\cdot;\omega^N]\), and \(T\), as in Section~\ref{sec:reduced-model}.

Part~(ii) is Theorem~\ref{thm:sec5-identification}, applied with the frequency profile
\(\Omega\). The constants are uniform over the stated frequency class because the estimates in
Sections~\ref{sec:infdim-reduction}--\ref{sec:identification} depend on the frequency only through
the common \(L^\infty\)-bound.

For \eqref{eq:sec5-exact-step-equality}, apply Lemma~\ref{lem:sec5-stepfield-identify} to the
continuum system with frequency profile \(\omega^N\). The continuum graph
\(h_{\varepsilon,\omega^N}(\theta^{N,\vartheta})\) then lies in \(\mathfrak{S}_N^a\) and has step
coefficients \(h_{\varepsilon,\omega^N}^N([\vartheta])\). Since
\(\mathcal F_{\omega^N}\) restricted to \(\mathfrak{S}_N^\theta\times \mathfrak{S}_N^a\) is precisely the embedded
finite-\(N\) phase vector field, the exact vector fields agree.

Finally, subtract the expansions in (i) and (ii) with \(\Omega=\omega\). The operators \(K\) and
\(T\) cancel, while
\[
\|P[\theta;\omega^N]-P[\theta;\omega]\|_{X_\theta}
\le 2M_\Gamma\Mdh\,\|\omega^N-\omega\|_{X_\theta}
\]
by the explicit formula for \(P\). The two named remainders are bounded by constants times
\(\varepsilon^2\), giving \eqref{eq:sec5-commutation} after renaming the constant.
\end{proof}

\begin{remark}[Scope of the compatibility statement]\label{rem:sec5-commutation-scope}
The reduce-first route is treated here on step fields, where the embedded finite-\(N\)
reduced vector field associated with the restricted Lyapunov--Perron graph is canonical for the
present comparison. We do not construct an
exact reduce-first continuum vector field on all of \(X_\theta\). The symbolic
compatibility statement is therefore only a statement about the common first-order vector-field
truncation, together with the explicit \(O(\varepsilon^2)\) remainders above and the admissible
\(\omega^N\to\omega\) step-field limit.
\end{remark}

\begin{remark}[Vector-field versus trajectory compatibility]\label{rem:sec5-vf-vs-trajectory}
{
Theorem~\ref{thm:sec5-commutation} is stated at the vector-field level, but this does not obstruct
the usual fixed-time trajectory comparison. Let \(\theta_{\mathrm{RF}}^N\) be the embedded
reduce-first trajectory generated by \(\mathcal V_\varepsilon^{N,\mathrm{RF}}\) from a step-field
initial datum \(\theta_0^N\), and let \(\theta_{\mathrm{CF}}\) be the continuum-first trajectory
generated by \(\mathcal V_{\varepsilon,\omega}^{\mathrm{CF}}\) from the same initial datum. Combining
\eqref{eq:sec5-commutation} with the reduced-flow Lipschitz estimate of
Proposition~\ref{prop:sec3-stability} gives, for every fixed \(T>0\),
\[
\sup_{t\in[0,T]}
\|\theta_{\mathrm{RF}}^N(t)-\theta_{\mathrm{CF}}(t)\|_{X_\theta}
\le
C_T^{\mathrm{red}}
\bigl(\|\omega^N-\omega\|_{X_\theta}+\varepsilon^2\bigr),
\]
where \(C_T^{\mathrm{red}}\) is independent of \(N\) and of
\(\varepsilon\in(0,\varepsilon_{0,r}]\). If the continuum-first route is instead constructed with
the same step frequency \(\omega^N\), then the exact vector-field identity
\eqref{eq:sec5-exact-step-equality}, together with step-field preservation and common initial
data, gives equality of the corresponding Lyapunov--Perron-selected trajectories. Thus standard
finite-time continuous-dependence estimates are available. What is not claimed here is a statement
uniform on diverging time horizons \(T=T(\varepsilon)\to\infty\), a comparison involving arbitrary
nearby slow-manifold representatives, or a comparison with an exact reduce-first continuum flow on
all of \(X_\theta\), which is not constructed in the present paper. Such stronger statements would
require additional long-time stability or higher-order matching assumptions beyond the present
scope.
}
\end{remark}

\begin{remark}[Persistence of higher-order structure]
The same first-order triplet operator $T$ defined in \eqref{eq:sec3-Top} appears in both route expansions. In the reduce-first route it arises as the continuum limit of the finite-$N$ triplet sum $N^{-2}\sum_{j,k}T_{ijk}$ from \eqref{eq:sec3-Tijk}. In the continuum-first route it emerges from the Fenichel correction $h_1$ through the substitution carried out in Theorem~\ref{thm:sec5-identification}. The agreement confirms that the triplet structure identified in the finite-dimensional setting of \cite{KuehnMurphy} persists in the macroscopic continuum description; its continuum nonpairwise character is then certified by the criterion below when the stated mixed-derivative condition holds.
\end{remark}

\subsection{Continuum nonpairwise criterion for the triplet operator}

 {To formulate a rigorous continuum analogue of the mixed-derivative test from the finite-$N$ setting in \cite{KuehnMurphy}, it is convenient to work on the closed subspace \(C(\I)\subset \Linfty(\I)\) with the supremum norm. Any pairwise representation valid on all of \(\Linfty(\I)\) restricts in particular to \(C(\I)\). Consequently, it suffices to rule out pairwise representability on \(C(\I)\).}

\begin{definition}[Continuum pairwise representability]\label{def:sec5-pairwise}
{ 
Let \(V:C(\I)\to C(\I)\) be a \(C^2\)-map. We say that \(V\) is \emph{pairwise-representable} if there exists a jointly measurable kernel
\[
\kappa:\R^2\times \I^2\to \R
\]
such that, for every \((x,y)\in \I^2\), the map \((u,v)\mapsto \kappa(u,v;x,y)\) belongs to \(C^2(\R^2)\), its partial derivatives up to order two are bounded uniformly in \((x,y)\), and
\begin{equation}\label{eq:sec5-pairwise-form}
V[\theta](x)=\int_0^1 \kappa\bigl(\theta(x),\theta(y);x,y\bigr)\,\dd y
\qquad
\text{for all }\theta\in C(\I),\ x\in \I.
\end{equation}
If no such kernel exists, then \(V\) is called \emph{genuinely nonpairwise}.}
\end{definition}

\begin{remark}[Scope of ``genuinely nonpairwise'']\label{rem:sec5-nonpairwise-scope}
Throughout this paper, the term \emph{genuinely nonpairwise} is used relative to the smooth bounded-kernel class of Definition~\ref{def:sec5-pairwise}: kernels \(\kappa\in C^2(\R^2)\) jointly measurable in \((x,y)\), with partial derivatives up to order two uniformly bounded in \((x,y)\in\I^2\). Rougher representation classes --- for example pairwise representations through distributional kernels, only-measurable kernels, or kernels with unbounded second derivatives --- are not considered here, and the continuum mixed-derivative test of Theorem~\ref{thm:sec5-triplet-criterion} below does not rule them out. The criterion thus says that the triplet operator \(T\) admits no representation by a smooth bounded pairwise kernel; it leaves open whether more singular pairwise representations might exist outside this class.
\end{remark}

\begin{proposition}[Off-diagonal second variations vanish for pairwise fields]\label{prop:sec5-pairwise-vanishing}
{ 
Let \(V:C(\I)\to C(\I)\) be pairwise-representable in the sense of Definition~\ref{def:sec5-pairwise}. Then, for every \(\theta,\phi,\psi\in C(\I)\) and every \(x\in \I\),
\begin{equation}\label{eq:sec5-pairwise-vanishing}
D^2V[\theta](\phi,\psi)(x)=0
\end{equation}
whenever
\[
\phi(x)=\psi(x)=0,
\qquad
\operatorname{supp}\phi\cap\operatorname{supp}\psi=\varnothing.
\]
}
\end{proposition}

\begin{proof}
{ 
Write \(V\) in the form \eqref{eq:sec5-pairwise-form}. Since \(\kappa\) has bounded derivatives up to order two, differentiation under the integral sign is justified by dominated convergence. Hence
\begin{align*}
DV[\theta]\phi(x)
&=
\int_0^1
\Big(
\partial_1\kappa\bigl(\theta(x),\theta(y);x,y\bigr)\phi(x)
+\partial_2\kappa\bigl(\theta(x),\theta(y);x,y\bigr)\phi(y)
\Big)\,\dd y,
\end{align*}
and a second differentiation gives
\begin{align*}
D^2V[\theta](\phi,\psi)(x)
&=
\int_0^1
\partial_{11}\kappa\bigl(\theta(x),\theta(y);x,y\bigr)\phi(x)\psi(x)\,\dd y\\
&\quad
+\int_0^1
\partial_{12}\kappa\bigl(\theta(x),\theta(y);x,y\bigr)
\bigl(\phi(x)\psi(y)+\psi(x)\phi(y)\bigr)\,\dd y\\
&\quad
+\int_0^1
\partial_{22}\kappa\bigl(\theta(x),\theta(y);x,y\bigr)\phi(y)\psi(y)\,\dd y.
\end{align*}
If \(\phi(x)=\psi(x)=0\) and \(\operatorname{supp}\phi\cap\operatorname{supp}\psi=\varnothing\), then \(\phi(y)\psi(y)=0\) for every \(y\in \I\), so all three terms vanish. This proves \eqref{eq:sec5-pairwise-vanishing}.}
\end{proof}

\begin{lemma}[Extension from \(C(\I)\) to \(L^\infty(\I)\)]\label{lem:sec5-extension}
Let \(V:\Linfty(\I)\to\Linfty(\I)\) be an operator whose restriction to \(C(\I)\) maps
\(C(\I)\) into \(C(\I)\). If \(V|_{C(\I)}\) is genuinely nonpairwise in the sense of
Definition~\ref{def:sec5-pairwise}, then \(V\) cannot admit a pairwise representation, valid for
all \(L^\infty\) inputs, by a kernel satisfying the regularity and boundedness conditions of
Definition~\ref{def:sec5-pairwise}.
\end{lemma}

\begin{proof}
If such a pairwise kernel representation were valid for all \(\theta\in\Linfty(\I)\), then the
same formula, with the same smooth bounded-kernel assumptions, would hold after restricting to the
subspace \(C(\I)\subset\Linfty(\I)\). This would make \(V|_{C(\I)}\) pairwise-representable,
contradicting the assumed genuine nonpairwiseness on \(C(\I)\).
\end{proof}

\begin{proposition}[Continuum nonpairwise certificate]\label{prop:sec5-nonpairwise-certificate}
{ 
Let \(V:C(\I)\to C(\I)\) be a \(C^2\) vector field. Suppose there exist \(\theta,\phi,\psi\in C(\I)\) and \(x\in \I\) such that
\[
\phi(x)=\psi(x)=0,
\qquad
\operatorname{supp}\phi\cap\operatorname{supp}\psi=\varnothing,
\]
and
\begin{equation}\label{eq:sec5-certificate}
D^2V[\theta](\phi,\psi)(x)\neq 0.
\end{equation}
Then \(V\) is genuinely nonpairwise.
}
\end{proposition}

\begin{proof}
{ 
This is the contrapositive of Proposition~\ref{prop:sec5-pairwise-vanishing}.}
\end{proof}

\begin{theorem}[Continuum criterion for the triplet operator]\label{thm:sec5-triplet-criterion}
Assume in addition that \(\Gamma\in C_b^2(\R)\) and \(H\in C_b^3(\R^2)\). Define
\begin{align}
F(u,v,w)
&:=
-\Gamma(v-u)\partial_1H(u,v)\,H(u,w)\Gamma(w-u)
\nonumber\\
&\quad
-\Gamma(v-u)\partial_2H(u,v)\,H(v,w)\Gamma(w-v),
\label{eq:sec5-triplet-integrand}
\end{align}
so that
\[
T[\theta](x)=\int_0^1\!\int_0^1 F\bigl(\theta(x),\theta(y),\theta(z)\bigr)\,\dd z\,\dd y.
\]
If there exist \(u_\ast,v_\ast,w_\ast\in \R\) such that
\begin{equation}\label{eq:sec5-triplet-mixed}
M(u_\ast,v_\ast,w_\ast)
:=
\partial_v\partial_wF(u_\ast,v_\ast,w_\ast)
+\partial_v\partial_wF(u_\ast,w_\ast,v_\ast)
\neq 0,
\end{equation}
then the triplet operator \(T:C(\I)\to C(\I)\) is genuinely nonpairwise relative to the smooth bounded-kernel class of Definition~\ref{def:sec5-pairwise}. In particular, \(T\) is not pairwise-representable in the original \(L^\infty\)-based continuum setting either, within the same smooth bounded-kernel class.
\end{theorem}

\begin{proof}
Under the stated regularity assumptions, \(F\in C_b^2(\R^3)\). Standard Nemytskii calculus on \(C(\I)\) shows that \(T:C(\I)\to C(\I)\) is \(C^2\), and differentiation under the double integral yields
\begin{align}
D^2T[\theta](\phi,\psi)(x)
&=
\int_0^1\!\int_0^1
\partial_{11}F(\Theta)\,\phi(x)\psi(x)\,\dd z\,\dd y
\nonumber\\
&\quad
+\int_0^1\!\int_0^1
\partial_{12}F(\Theta)\bigl(\phi(x)\psi(y)+\psi(x)\phi(y)\bigr)\,\dd z\,\dd y
\nonumber\\
&\quad
+\int_0^1\!\int_0^1
\partial_{13}F(\Theta)\bigl(\phi(x)\psi(z)+\psi(x)\phi(z)\bigr)\,\dd z\,\dd y
\nonumber\\
&\quad
+\int_0^1\!\int_0^1
\partial_{22}F(\Theta)\,\phi(y)\psi(y)\,\dd z\,\dd y
\nonumber\\
&\quad
+\int_0^1\!\int_0^1
\partial_{33}F(\Theta)\,\phi(z)\psi(z)\,\dd z\,\dd y
\nonumber\\
&\quad
+\int_0^1\!\int_0^1
\partial_{23}F(\Theta)\bigl(\phi(y)\psi(z)+\psi(y)\phi(z)\bigr)\,\dd z\,\dd y,
\label{eq:sec5-second-variation-T}
\end{align}
where \(\Theta=(\theta(x),\theta(y),\theta(z))\).

Choose \(x_\ast\in \I\) and two disjoint closed intervals \(I_y,I_z\subset \I\setminus\{x_\ast\}\) with nonempty interiors. Let \(\theta\in C(\I)\) satisfy
\[
\theta(x_\ast)=u_\ast,
\qquad
\theta\equiv v_\ast \text{ on }I_y,
\qquad
\theta\equiv w_\ast \text{ on }I_z,
\]
and choose nonnegative functions \(\phi,\psi\in C(\I)\) such that
\[
\operatorname{supp}\phi\subset I_y,
\qquad
\operatorname{supp}\psi\subset I_z,
\qquad
\int_0^1 \phi(y)\,\dd y>0,
\qquad
\int_0^1 \psi(z)\,\dd z>0.
\]
Then \(\phi(x_\ast)=\psi(x_\ast)=0\) and \(\operatorname{supp}\phi\cap\operatorname{supp}\psi=\varnothing\). Hence every term in \eqref{eq:sec5-second-variation-T} vanishes except the \(\partial_{23}F\) contribution, and because \(\theta\) is constant on \(\operatorname{supp}\phi\) and \(\operatorname{supp}\psi\), we obtain
\begin{align*}
D^2T[\theta](\phi,\psi)(x_\ast)
&=
\int_0^1\!\int_0^1
\partial_{23}F\bigl(u_\ast,v_\ast,w_\ast\bigr)\phi(y)\psi(z)\,\dd z\,\dd y\\
&\quad
+\int_0^1\!\int_0^1
\partial_{23}F\bigl(u_\ast,w_\ast,v_\ast\bigr)\psi(y)\phi(z)\,\dd z\,\dd y\\
&=
M(u_\ast,v_\ast,w_\ast)
\Bigl(\int_0^1\phi(y)\,\dd y\Bigr)
\Bigl(\int_0^1\psi(z)\,\dd z\Bigr).
\end{align*}
By \eqref{eq:sec5-triplet-mixed}, this quantity is nonzero. Proposition~\ref{prop:sec5-nonpairwise-certificate} therefore implies that \(T\) is genuinely nonpairwise on \(C(\I)\). By Lemma~\ref{lem:sec5-extension}, \(T\) is also genuinely nonpairwise in the original \(L^\infty\)-based continuum setting, within the same smooth bounded-kernel class.
\end{proof}

We now illustrate the abstract nonpairwise criterion from Theorem~\ref{thm:sec5-triplet-criterion} with the canonical adaptive Kuramoto choice of coupling functions.
This example shows concretely that the continuum triplet operator $T$ produced by the slow-manifold reduction is not merely formally triplet-valued, but is genuinely nonpairwise in the representability sense introduced above.
\begin{example}[Adaptive Kuramoto case]\label{ex:sec5-kuramoto-triplet}
Let
\[
\Gamma(\phi)=\sin\phi,
\qquad
H(u,v)=\alpha+\cos(u-v),
\qquad
\alpha\neq 0.
\]
Then the continuum triplet operator \(T\) is genuinely nonpairwise relative to the smooth bounded-kernel class of Definition~\ref{def:sec5-pairwise}.
\end{example}

\begin{proof}
For this choice of \(\Gamma\) and \(H\), the integrand \eqref{eq:sec5-triplet-integrand} becomes
\[
F(u,v,w)
=
-\sin^2(v-u)\bigl(\alpha+\cos(u-w)\bigr)\sin(w-u)
+\sin^2(v-u)\bigl(\alpha+\cos(v-w)\bigr)\sin(w-v).
\]
Set
\[
P(y):=\sin^2 y,
\qquad
Q(x):=(\alpha+\cos x)\sin x.
\]
Then
\[
F(u,v,w)=-P(v-u)Q(w-u)+P(v-u)Q(w-v).
\]
Hence
\[
\partial_v\partial_wF(u,v,w)
=
-P'(v-u)Q'(w-u)+P'(v-u)Q'(w-v)-P(v-u)Q''(w-v).
\]
We evaluate at \((u_\ast,v_\ast,w_\ast)=(0,\pi/2,0)\). The point is chosen because
\(P'(\pi/2)=0\) annihilates the two pairwise-type cross terms \(P'(v-u)Q'(\cdot)\) in
\(\partial_v\partial_wF\), isolating the single genuinely-triplet contribution
\(-P(v-u)Q''(w-v)\), which is exactly the mixed second derivative that no pairwise kernel can reproduce. Any
triple with \(P'(v-u)=0\), \(P(v-u)\neq0\), and \(Q''(w-v)\neq0\) would serve equally well.
Concretely, using
\[
P'(\pi/2)=0,
\qquad
P(\pi/2)=1,
\qquad
Q''(-\pi/2)=\alpha,
\]
gives
\[
\partial_v\partial_wF(0,\pi/2,0)=-\alpha.
\]
For the swapped arguments,
\[
F(u,w,v)=-P(w-u)Q(v-u)+P(w-u)Q(v-w),
\]
so
\[
\partial_v\partial_wF(u,w,v)
=
-P'(w-u)Q'(v-u)+P'(w-u)Q'(v-w)-P(w-u)Q''(v-w).
\]
Evaluating \(\partial_v\partial_w[F(u,w,v)]\) at \((u,v,w)=(0,\pi/2,0)\) gives
\[
\partial_v\partial_wF(0,0,\pi/2)=0,
\]
because \(P(0)=P'(0)=0\). Therefore
\[
M(0,\pi/2,0)
=
\partial_v\partial_wF(0,\pi/2,0)
+\partial_v\partial_wF(0,0,\pi/2)
=
-\alpha
\neq 0.
\]
Theorem~\ref{thm:sec5-triplet-criterion} now applies and shows that \(T\) is genuinely nonpairwise.
\end{proof}

\section{Discussion and Outlook}\label{sec:discussion}

The central question addressed in this paper is whether two natural operations associated with
adaptive fast--slow oscillator networks, Fenichel reduction and the dense-graph continuum limit,
are compatible. The main conclusion is a qualified compatibility statement at the
vector-field level through first order in \(\varepsilon\): along admissible equal-cell step approximations, both the
reduce-first route and the continuum-first route produce the same first-order continuum
truncation, with the same leading-order operator \(K\), the same first-order pairwise correction
\(P\), and the same triplet operator \(T\), up to explicitly controlled
\(O(\varepsilon^2)\) remainders.

The conceptual significance is that the higher-order continuum interaction is not inserted
phenomenologically. It is inherited from pairwise adaptive microscopic dynamics and is preserved,
to this order, under both reduction and passage to the continuum. Together with the
finite-dimensional analysis in~\cite{KuehnMurphy}, Theorem~\ref{thm:sec5-triplet-criterion}, and
Example~\ref{ex:sec5-kuramoto-triplet}, this shows that for coupling functions satisfying the
continuum criterion, and in particular for the adaptive Kuramoto example, the first-order triplet
operator is not pairwise-representable in the smooth bounded-kernel class. In that precise sense,
the effective nonpairwise structure is not a finite-\(N\) artefact.

At the technical level, the proof succeeds because two favourable structures align. The continuum
limit has the dense-graph integral form familiar from continuum and graphon-type
models~\cite{Neunzert1984,Lancellotti2005,Medvedev2014,ChibaMediano2019}. At the same time, the
fast subsystem has the explicit relaxation law \(a_t=\varepsilon^{-1}(-a+h_0(\theta))\), so the
critical manifold is the graph \(a=h_0(\theta)\) and the normal linearisation is the fixed
operator \(-I\). This makes it possible to combine continuum-limit arguments with a
Lyapunov--Perron construction in a comparatively transparent Banach-space
setting~\cite{Hummel-Kuehn,Kuehn-Sulzbach}. The closest neighbouring continuum-limit results for
adaptive networks include~\cite{GkogkasKuehnXu2023,Throm2024EJAM,CestnikMartens2025}; the present
work differs by focusing on the fast-adaptation regime and on compatibility with Fenichel
reduction.

The present theory also has limitations: it focuses on the emergence of higher-order interaction
terms on the slow manifold rather than on optimal regularity. In particular, the discrete-to-continuum results are
proved in the dense-interaction regime, using labelled \(L^\infty\) equal-cell step embeddings and
data satisfying the compatible step-approximation hypotheses. Extending the argument beyond this
framework, for example to graphons considered only up to relabelling, \(L^p\)-graphons, graphops, or
sparse, random, or heterogeneous graph limits, is likely to require different convergence notions,
stability estimates, and limit objects~\cite{KaliuzhnyiVerbovetskyiMedvedev2017}. The explicit reduced-vector-field
expansion is carried out only through the first correction in \(\varepsilon\): a second-order
expansion would require explicit \(\varepsilon^2\)-terms with an \(O(\varepsilon^3)\) remainder, and
a trajectory-level comparison between the exact reduce-first and continuum-first reduced flows
would require additional arguments beyond those developed here. 

These limitations point to several natural directions for further work. One is to continue the
slow-manifold expansion to higher orders, where one expects additional higher-order nonlinear
integral terms beyond the triplet term and hence a richer hierarchy of effective interactions
generated by fast adaptation. A specific question in this direction is whether the first-order pairwise correction \(P[\theta;\omega]\) admits an analogous nonpairwise certificate. By construction \(P\) is built from a single integral against a smooth bounded kernel and is therefore pairwise-representable in the sense of Definition~\ref{def:sec5-pairwise}; the question is rather whether higher-order corrections to the slow vector field beyond first order in \(\varepsilon\) carry genuinely nonpairwise structure detectable by an analogue of Theorem~\ref{thm:sec5-triplet-criterion}, perhaps via a third-order mixed-derivative test on disjointly supported triples \((\phi_1,\phi_2,\phi_3)\). A positive answer would extend the present rigour-of-emergence framework to a hierarchy of higher-order operators rather than the single triplet term certified here.

\begin{conjecture}[All-order compatibility under higher regularity]\label{conj:all-order}
Suppose \(\Gamma\) and \(H\) are smooth with all derivatives bounded, and the
initial data are correspondingly regular and compatible. Then the reduce-first and continuum-first
routes induce reduced slow vector fields with identical asymptotic \(\varepsilon\)-expansions: for
every \(k\ge1\) the two routes agree up to an \(O(\varepsilon^{k+1})\) remainder, so that
\(\mathrm{CL}\circ\mathrm{Red}_\varepsilon\) and \(\mathrm{Red}^\infty_\varepsilon\circ\mathrm{CL}\)
commute to all orders in \(\varepsilon\). The present paper establishes the case \(k=1\).
\end{conjecture}

This is necessarily a statement about the asymptotic \(\varepsilon\)-series rather
than about exact equality of slow manifolds. Because the slow directions are neutral
(centre-type), the Lyapunov--Perron-selected manifolds are pinned down only up to exponentially
small \(O(e^{-c/\varepsilon})\) corrections, so beyond-all-orders discrepancies between the two
routes are not excluded. Proving even the term-by-term agreement would require a higher-order
Lyapunov--Perron expansion with \(C^k\)-control of the remainders, together with matching of the
selected representatives at each order.

A second direction is to study the continuum reduced equation as a
dynamical system in its own right, for example with respect to synchronisation, multistability,
clustering, and bifurcation phenomena, now in a model where the higher-order terms are derived
rather than imposed~\cite{BickAshwinRodrigues2016Chaos,SkardalArenas1}. A third is to connect the
present dense-graph framework with more general continuum descriptions of large oscillator systems
and adaptive graph limits.

More broadly, the message of this work is that continuum limits can preserve
structural effects created by model reduction. When microscopic dynamics are adaptive and
multiscale, the macroscopic limit can retain a precise memory of those mechanisms in the form of
effective higher-order operators. Developing a general theory for when such operators emerge,
persist, and shape collective dynamics seems a promising step towards a more principled
mathematical description of adaptive network systems.

\appendix

\section{Technical Lemmas for the Lyapunov--Perron Construction}\label{app:lp-technical}

This appendix collects the auxiliary estimates used in Section~\ref{sec:infdim-reduction}.
\begin{lemma}[Finite-window \(L^\infty\) superposition calculus]
\label{lem:sec4-nemytskii-calculus}
Let \((\Omega,\mu)\) be a finite measure space and let
\(f\in C_b^k(\R^m;\R^\ell)\), \(k\ge1\). Then the Nemytskii operator
\[
\mathcal N_f:L^\infty(\Omega;\R^m)\to L^\infty(\Omega;\R^\ell),
\qquad
(\mathcal N_f u)(\zeta):=f(u(\zeta)),
\]
is \(C^k\). For \(1\le j\le k\),
\begin{equation}\label{eq:sec4-nemytskii-derivatives}
\bigl(\DD^j\mathcal N_f(u)[v_1,\ldots,v_j]\bigr)(\zeta)
=
\DD^j f(u(\zeta))[v_1(\zeta),\ldots,v_j(\zeta)],
\end{equation}
and
\[
\|\DD^j\mathcal N_f(u)\|_{\mathcal B^j(L^\infty,L^\infty)}
\le
\|\DD^j f\|_{L^\infty(\R^m)}.
\]
The same conclusion remains valid after precomposition and postcomposition with bounded linear
maps; in particular it applies to the pullbacks
\(\theta\mapsto ((x,y)\mapsto (\theta(x),\theta(y)))\) and to the averaging map
\[
J:L^\infty(\I^2)\to L^\infty(\I),
\qquad
(Jq)(x):=\int_0^1 q(x,y)\,\dd y .
\]

Moreover, let \(E,Y\) be Banach spaces, let \(\Phi\in C^1(E,Y)\), fix \(T>0\), and let
\(z\in C([-T,0];E)\). If \(w\in C([-T,0];E)\) and
\(\|w\|_{C([-T,0];E)}\to0\), then the Taylor remainders are uniform on the finite window:
\[
\sup_{s\in[-T,0]}
\frac{
\|\Phi(z(s)+w(s))-\Phi(z(s))-\DD\Phi(z(s))w(s)\|_Y
}{
\|w\|_{C([-T,0];E)}
}
\longrightarrow0 .
\]
\end{lemma}

\begin{proof}
For \(u\in L^\infty(\Omega;\R^m)\), the essential range of \(u\) lies in a compact ball of
\(\R^m\). The pointwise Taylor formula for \(f\), together with the uniform continuity of the
relevant derivatives on a slightly larger compact ball, gives the Frechet derivative in
\(L^\infty\) after taking essential suprema. Iterating the argument gives
\eqref{eq:sec4-nemytskii-derivatives}; the displayed operator bound follows immediately from
the essential-supremum norm. Bounded linear pre- and postcomposition preserve \(C^k\)-regularity.
For the finite-window assertion, the set \(z([-T,0])\) is compact in \(E\). The continuity of
\(\DD\Phi\), a finite-cover argument on this compact set, and the smallness of
\(\|w\|_{C([-T,0];E)}\) give the uniform Taylor remainder.
\end{proof}

\begin{lemma}[Regularity of the truncated nonlinearities]
\label{lem:sec4-Ar-Br-C1}
The maps \(A_r:X_a\times X_\theta\to X_a\) and
\(B_r:X_a\times X_\theta\to X_\theta\) are \(C^1\). Moreover, for all
\((\xi,\theta)\in X_a\times X_\theta\),
\[
\|\DD A_r(\xi,\theta)\|_{\mathcal B(X_a\times X_\theta,X_a)}\le C_3(r),
\qquad
\|\DD B_r(\xi,\theta)\|_{\mathcal B(X_a\times X_\theta,X_\theta)}\le C_4(r),
\]
where \(X_a\times X_\theta\) is equipped with the sum norm.
\end{lemma}

\begin{proof}
The \(C^1\)-regularity follows from Lemma~\ref{lem:sec4-nemytskii-calculus}, the smoothness of
the cut-off \(\sigma_r\), the bounded periodic coefficient hypotheses, and boundedness of the
averaging operator \(J:L^\infty(\I^2)\to L^\infty(\I)\). To make the estimate explicit, write
\[
\rho(x,y):=H(\theta(x),\theta(y))+\xi(x,y),
\qquad
\Delta(x,y):=\theta(y)-\theta(x).
\]
For \((u,v)\in X_a\times X_\theta\),
\[
\delta\rho_{u,v}(x,y):=
u(x,y)+\partial_1H(\theta(x),\theta(y))v(x)
+\partial_2H(\theta(x),\theta(y))v(y).
\]
\begin{align}
\bigl(\DD B_r(\xi,\theta)[u,v]\bigr)(x)
&=
\int_0^1
\sigma_r'(\rho(x,y))\delta\rho_{u,v}(x,y)\Gamma(\Delta(x,y))\,\dd y
\nonumber\\
&\quad+
\int_0^1
\sigma_r(\rho(x,y))\Gamma'(\Delta(x,y))(v(y)-v(x))\,\dd y .
\label{eq:sec4-DBr-formula}
\end{align}
Using \(|\sigma_r'|\le1\), \(|\sigma_r|\le2r\), and the bounds on \(H\) and \(\Gamma\), this gives
\[
\|\DD B_r(\xi,\theta)[u,v]\|_{X_\theta}
\le
M_\Gamma\|u\|_{X_a}
+\bigl(2M_\Gamma\Mdh+4rL_\Gamma\bigr)\|v\|_{X_\theta}
\le C_4(r)\|(u,v)\|.
\]
Since \(A_r(\xi,\theta)=\DD h_0(\theta)B_r(\xi,\theta)\), the product rule gives
\begin{equation}\label{eq:sec4-DAr-formula}
\DD A_r(\xi,\theta)[u,v]
=
\DD^2h_0(\theta)[v,B_r(\xi,\theta)]
+
\DD h_0(\theta)\DD B_r(\xi,\theta)[u,v].
\end{equation}
Together with
\[
\|B_r(\xi,\theta)\|_{X_\theta}\le C_2(r),\qquad
\|\DD h_0(\theta)\|\le2\Mdh,
\]
and
\[
\|\DD^2h_0(\theta)[v,\phi]\|_{X_a}
\le4\MdTwoH\|v\|_{X_\theta}\|\phi\|_{X_\theta},
\]
this yields the stated \(C_3(r)\)-bound.
\end{proof}

\begin{lemma}[Anchor and variational Lyapunov--Perron estimates]
\label{lem:sec4-variational-lp}
Let \((\xi_{\eps,r}^{\theta_0},\theta_{\eps,r}^{\theta_0})\) be the fixed point of
\(\mathcal L_{\eps,\theta_0,r}\), and write
\[
z^{\theta_0}(t):=(\xi_{\eps,r}^{\theta_0}(t),\theta_{\eps,r}^{\theta_0}(t)).
\]
Then, for all \(\theta_0,\widetilde\theta_0\in X_\theta\),
\begin{equation}\label{eq:sec4-anchor-lip}
\|z^{\theta_0}-z^{\widetilde\theta_0}\|_{C_{\eta_r,X_a}\times C_{\eta_r,X_\theta}}
\le
\frac{1}{1-q_*}\|\theta_0-\widetilde\theta_0\|_{X_\theta}.
\end{equation}
Moreover, for every \(\phi\in X_\theta\), the linear variational Lyapunov--Perron system
\begin{align}
U(t)
&=
-\int_{-\infty}^{t} e^{-\eps^{-1}(t-s)}
\DD A_r(z^{\theta_0}(s))[U(s),V(s)]\,\dd s,
\label{eq:sec4-var-U}\\
V(t)
&=
\phi+\int_0^t
\DD B_r(z^{\theta_0}(s))[U(s),V(s)]\,\dd s
\label{eq:sec4-var-V}
\end{align}
has a unique solution
\((U_{\theta_0}^{\phi},V_{\theta_0}^{\phi})
\in C_{\eta_r,X_a}\times C_{\eta_r,X_\theta}\). The map
\(\phi\mapsto(U_{\theta_0}^{\phi},V_{\theta_0}^{\phi})\) is linear and bounded, and
\begin{equation}\label{eq:sec4-variational-bound}
\|(U_{\theta_0}^{\phi},V_{\theta_0}^{\phi})\|_{C_{\eta_r,X_a}\times C_{\eta_r,X_\theta}}
\le
\frac{1}{1-q_*}\|\phi\|_{X_\theta}.
\end{equation}
\end{lemma}

\begin{proof}
The estimate \eqref{eq:sec4-anchor-lip} is the contraction estimate for two fixed points with
different anchors: the second component of \(\mathcal L_{\eps,\theta_0,r}\) changes by exactly
\(\theta_0-\widetilde\theta_0\), while the remaining terms are \(q_*\)-Lipschitz. For
\eqref{eq:sec4-var-U}--\eqref{eq:sec4-var-V}, the linear part has the same contraction constant
because Lemma~\ref{lem:sec4-Ar-Br-C1} gives the bounds \(C_3(r)\) and \(C_4(r)\) for
\(\DD A_r\) and \(\DD B_r\). The affine term is only the anchor direction \(\phi\) in the
second component. Banach's fixed-point theorem therefore gives existence, uniqueness, linearity in
\(\phi\), and \eqref{eq:sec4-variational-bound}.
\end{proof}

\begin{lemma}[Finite-window Taylor remainders]
\label{lem:sec4-finite-window-remainders}
Fix \(T>0\), \(\theta_0\in X_\theta\), and \(0<\eps\le\eps_{0,r}\). For
\(\psi\in X_\theta\), set \(z^\psi:=z^{\theta_0+\psi}\), \(z^0:=z^{\theta_0}\), and define
\[
\mathcal E_A^\psi(s)
:=
A_r(z^\psi(s))-A_r(z^0(s))
-\DD A_r(z^0(s))[z^\psi(s)-z^0(s)],
\]
with \(\mathcal E_B^\psi\) defined analogously using \(B_r\). Then
\[
\sup_{-T\le s\le0}
\frac{
\|\mathcal E_A^\psi(s)\|_{X_a}+\|\mathcal E_B^\psi(s)\|_{X_\theta}
}{
\|\psi\|_{X_\theta}
}
\longrightarrow0
\qquad\text{as }\|\psi\|_{X_\theta}\to0 .
\]
\end{lemma}

\begin{proof}
By \eqref{eq:sec4-anchor-lip},
\[
\sup_{-T\le s\le0}\|z^\psi(s)-z^0(s)\|_{X_a\times X_\theta}
\le
\frac{e^{-\eta_rT}}{1-q_*}\|\psi\|_{X_\theta}.
\]
Thus \(z^\psi\to z^0\) uniformly on every fixed finite window. Applying the finite-window part of
Lemma~\ref{lem:sec4-nemytskii-calculus} to the \(C^1\)-maps \(A_r\) and \(B_r\) gives Taylor
remainders which are \(o(\|z^\psi-z^0\|_{C([-T,0])})\), and the displayed estimate follows from
the preceding Lipschitz bound.
\end{proof}

\begin{lemma}[Fast-kernel tail estimate]\label{lem:sec4-tail-evaluation}
There exists \(K_r>0\), independent of
\(\theta_0,\psi,\eps\in(0,\eps_{0,r}]\), and \(T>0\), such that, with
\((U_{\theta_0}^{\psi},V_{\theta_0}^{\psi})\) as in
Lemma~\ref{lem:sec4-variational-lp},
\[
\left\|
\int_{-\infty}^{-T} e^{s/\eps}
\Bigl(
A_r(z^\psi(s))-A_r(z^0(s))
-\DD A_r(z^0(s))[U_{\theta_0}^{\psi}(s),V_{\theta_0}^{\psi}(s)]
\Bigr)\,\dd s
\right\|_{X_a}
\le
K_r\|\psi\|_{X_\theta}e^{-(\eps^{-1}+\eta_r)T}.
\]
\end{lemma}

\begin{proof}
The Lipschitz bound for \(A_r\), the derivative bound for \(A_r\), and
\eqref{eq:sec4-anchor-lip}--\eqref{eq:sec4-variational-bound} give, for \(s\le0\),
\[
\begin{aligned}
&\|A_r(z^\psi(s))-A_r(z^0(s))
-\DD A_r(z^0(s))[U_{\theta_0}^{\psi}(s),V_{\theta_0}^{\psi}(s)]\|_{X_a}\\
&\qquad\le
\frac{2C_3(r)}{1-q_*}\|\psi\|_{X_\theta}e^{\eta_r s}.
\end{aligned}
\]
Since \(1+\eps\eta_r\ge1/2\),
\[
\int_{-\infty}^{-T} e^{(\eps^{-1}+\eta_r)s}\,\dd s
=
\frac{\eps}{1+\eps\eta_r}e^{-(\eps^{-1}+\eta_r)T}
\le
2\eps_{0,r}e^{-(\eps^{-1}+\eta_r)T}.
\]
The result follows, for instance with
\(K_r=4C_3(r)\eps_{0,r}/(1-q_*)\).
\end{proof}

\begin{lemma}[Differentiability of the evaluated graph]
\label{lem:sec4-evaluation-C1}
For fixed \(0<\eps\le\eps_{0,r}\), the map
\[
\Theta_\eps:X_\theta\to X_a,
\qquad
\Theta_\eps(\theta_0):=\xi_{\eps,r}^{\theta_0}(0),
\]
is \(C^1\). Its derivative is
\[
\DD\Theta_\eps(\theta_0)\phi=U_{\theta_0}^{\phi}(0),
\]
where \((U_{\theta_0}^{\phi},V_{\theta_0}^{\phi})\) is the solution of the variational
Lyapunov--Perron system \eqref{eq:sec4-var-U}--\eqref{eq:sec4-var-V}.
\end{lemma}

\begin{proof}

The point is to differentiate the evaluated fast component at \(t=0\), not the full weighted
half-line superposition map \(\xi\mapsto\bigl(t\mapsto\int_{-\infty}^t e^{-\varepsilon^{-1}(t-s)}A_r(\xi(s),\theta(s))\,\dd s\bigr)\). The latter need not be Fr\'echet differentiable as a map
from \(C_{\eta_r,X_a}\times C_{\eta_r,X_\theta}\) into itself: a perturbation
\(w\in C_{\eta_r,X_\theta}\) only satisfies \(\|w(s)\|_{X_\theta}\lesssim e^{\eta_r s}\) as
\(s\to-\infty\), which is small in the weighted norm but unbounded pointwise; the candidate
linearisation \(\DD A_r(z^0(s))[U(s),V(s)]\) can therefore have, on the same scale, growth
that the kernel \(e^{-\varepsilon^{-1}(t-s)}\) does not damp uniformly in \(t\le 0\). The
Taylor remainder of the half-line superposition then fails to be \(o(\|w\|)\) in operator norm.

Evaluation at \(t=0\) escapes this obstruction because the kernel
\(s\mapsto e^{s/\varepsilon}\) decays so rapidly as \(s\to-\infty\) that, when paired with the
\(e^{\eta_r s}\)-controlled growth of the perturbation, the contribution of the remote past is
absorbed into a quantitative tail bound (Lemma~\ref{lem:sec4-tail-evaluation}). The
differentiability proof therefore proceeds by combining a finite-window Taylor remainder on
\([-T,0]\) (Lemma~\ref{lem:sec4-finite-window-remainders}) with the fast-kernel tail estimate on
\((-\infty,-T)\), letting \(\|\psi\|_{X_\theta}\to 0\) at fixed \(T\) and then \(T\to\infty\); the
two limits commute precisely because the fast kernel decay is independent of the weighted-space
modulus of \(\psi\).

Let \(R^\psi(t):=z^\psi(t)-z^0(t)-
(U_{\theta_0}^{\psi}(t),V_{\theta_0}^{\psi}(t))\). Subtracting the fixed-point equations for
\(z^\psi\) and \(z^0\), and then subtracting the variational equations, gives on every finite
window \([-T,0]\) a linear Volterra system for \(R^\psi\). Its forcing terms are the Taylor
remainders \(\mathcal E_A^\psi,\mathcal E_B^\psi\), plus the fast-component tail from
\((-\infty,-T)\). The same kernel estimates as in Lemma~\ref{lem:lyapunov_perron} give
\[
\begin{aligned}
\|R^\psi_1(0)\|_{X_a}
&\le
\frac{1}{1-q_*}
\left[
\frac{\eps}{1+\eps\eta_r}
\sup_{-T\le s\le0}\|\mathcal E_A^\psi(s)\|_{X_a}
\right.\\
&\qquad\left.
+\frac{1}{|\eta_r|}
\sup_{-T\le s\le0}\|\mathcal E_B^\psi(s)\|_{X_\theta}
+K_r\|\psi\|_{X_\theta}e^{-(\eps^{-1}+\eta_r)T}
\right],
\end{aligned}
\]
where \(R^\psi_1\) denotes the \(X_a\)-component. Dividing by \(\|\psi\|_{X_\theta}\), applying
Lemma~\ref{lem:sec4-finite-window-remainders}, first letting \(\|\psi\|_{X_\theta}\to0\) for
fixed \(T\), and then letting \(T\to\infty\), yields
\[
\frac{\|\Theta_\eps(\theta_0+\psi)-\Theta_\eps(\theta_0)
-U_{\theta_0}^{\psi}(0)\|_{X_a}}{\|\psi\|_{X_\theta}}
\longrightarrow0 .
\]
Thus \(\Theta_\eps\) is Frechet differentiable with the stated derivative.

It remains to record continuity of the derivative. Let \(\theta_n\to\theta_0\) in \(X_\theta\).
The anchor estimate \eqref{eq:sec4-anchor-lip} gives uniform convergence
\(z^{\theta_n}\to z^{\theta_0}\) on each finite window. By
Lemma~\ref{lem:sec4-Ar-Br-C1}, the coefficient maps
\(\DD A_r(z^{\theta_n}(\cdot))\) and \(\DD B_r(z^{\theta_n}(\cdot))\) converge uniformly on that
window in operator norm. Subtracting the two variational equations and using the same
contraction estimate gives convergence of
\(\phi\mapsto U_{\theta_n}^{\phi}(0)\) to \(\phi\mapsto U_{\theta_0}^{\phi}(0)\) in
\(\mathcal B(X_\theta,X_a)\); the remote fast tail is again controlled by
Lemma~\ref{lem:sec4-tail-evaluation}, uniformly for \(\|\phi\|_{X_\theta}\le1\). Hence
\(\Theta_\eps\in C^1(X_\theta,X_a)\).
\end{proof}

\begin{remark}[{On the differentiability estimate}]
The bound on \(R^\psi_1(0)\) above is obtained from the same weighted Volterra
estimates as in Lemma~\ref{lem:lyapunov_perron}. On the finite window \([-T,0]\) the forcing terms
\(\mathcal E_A^\psi,\mathcal E_B^\psi\) are \(o(\|\psi\|_{X_\theta})\) by
Lemma~\ref{lem:sec4-finite-window-remainders}, while the contribution of \((-\infty,-T)\) is
controlled by the fast-kernel tail bound of Lemma~\ref{lem:sec4-tail-evaluation}; the weighted
contraction factor \(1/(1-q_*)\) and the kernel factors \(\varepsilon/(1+\varepsilon\eta_r)\) and
\(1/|\eta_r|\) are exactly those of the fixed-point argument in
Lemma~\ref{lem:lyapunov_perron}. Letting \(\|\psi\|_{X_\theta}\to0\) at fixed \(T\) and then
\(T\to\infty\) gives the stated limit. The remaining bookkeeping---writing out the linear Volterra
system for \(R^\psi\) on each finite window and summing the two contributions---is routine and is
omitted.
\end{remark}

\begin{lemma}[\(O(\varepsilon)\) derivative closeness]
\label{lem:sec4-derivative-closeness}
For every \(\theta_0,\phi\in X_\theta\),
\[
\|U_{\theta_0}^{\phi}(0)\|_{X_a}
\le
\frac{2C_3(r)}{1-q_*}\,\eps\,\|\phi\|_{X_\theta}.
\]
Consequently
\[
\sup_{\theta_0\in X_\theta}
\|\DD\Theta_\eps(\theta_0)\|_{\mathcal B(X_\theta,X_a)}
\le
\frac{2C_3(r)}{1-q_*}\,\eps .
\]
\end{lemma}

\begin{proof}
Evaluating \eqref{eq:sec4-var-U} at \(t=0\) and using
\eqref{eq:sec4-variational-bound} gives
\[
\|U_{\theta_0}^{\phi}(0)\|_{X_a}
\le
C_3(r)\int_{-\infty}^{0}e^{s/\eps}e^{\eta_r s}\,\dd s\,
\frac{\|\phi\|_{X_\theta}}{1-q_*}
=
C_3(r)\frac{\eps}{1+\eps\eta_r}
\frac{\|\phi\|_{X_\theta}}{1-q_*}.
\]
The assumption \(1+\eps\eta_r\ge1/2\) gives the displayed estimate.
\end{proof}

\section*{Acknowledgements}
This work was supported by the European Union's Horizon Europe Marie Sk\l odowska-Curie Actions under the ``BeyondTheEdge: Higher-Order Networks and Dynamics'' project (Grant Agreement No.~101120085).

\bibliographystyle{siam}
\bibliography{main_bibliography}

\end{document}